\documentclass[11pt, onecolumn, svgnames]{article}

\usepackage[svgnames]{xcolor}
\usepackage{constants}

\usepackage{stmaryrd}
\usepackage{tikz}
\usetikzlibrary{shapes.misc}
\usetikzlibrary{decorations.pathreplacing}
\tikzset{cross/.style={cross out, draw=black, minimum size=2*(#1-\pgflinewidth), inner sep=0pt, outer sep=0pt},
cross/.default={1pt}}

\usepackage{pstricks}
\usepackage[utf8]{inputenc}
\usepackage[T1]{fontenc}
\usepackage{lmodern}
\usepackage{xr}
\usepackage{bold-extra}
\usepackage{dsfont}

\usepackage{amsmath}
\usepackage{amssymb}
\usepackage{mathrsfs}
\usepackage{mathtools}

\usepackage{amsfonts}
\usepackage{amsthm}

\usepackage{enumerate}
\usepackage{multirow}
\usepackage{graphicx}
\usepackage{bm}
\usepackage[hypertexnames=false]{hyperref}

\hypersetup{
    colorlinks,
    linkcolor={Black},
    citecolor={Black},
    urlcolor={Black}
}

\usepackage[top=2.5cm,bottom=2.5cm,left=2cm,right=2cm,marginparwidth=2.5cm]{geometry}
\reversemarginpar

\makeatletter
\renewcommand{\paragraph}{%
  \@startsection{paragraph}{4}%
  {\z@}{1.5ex \@plus 1ex \@minus .2ex}{-1em}%
  {\normalfont\normalsize\bfseries}%
}
\makeatother

{\theoremstyle{plain}
\newtheorem{Proposition}{\textbf{Proposition}}[section]

\newtheorem{Lemma}[Proposition]{\textbf{Lemma}}

\newtheorem{Claim}[Proposition]{\textbf{Claim}}
\newtheorem{Corollary}[Proposition]{\textbf{Corollary}}
 }

{\theoremstyle{plain}
\newtheorem{Theorem}[Proposition]{Theorem}}

{\theoremstyle{definition}
\newtheorem{Definition}[Proposition]{Definition}}

{\theoremstyle{remark}

\newtheorem{Remark}[Proposition]{Remark}

}

\numberwithin{equation}{section}

\newcommand{\eps}{\varepsilon}

\newcommand{\bbE}{\mathbb{E}}

\newcommand{\bbR}{\mathbb{R}}
\newcommand{\bbS}{\mathbb{S}}

\newcommand{\Z}{\mathbb{Z}}
\newcommand{\R}{\mathbb{R}}
\newcommand{\Q}{\mathbb{Q}}
\renewcommand{\S}{\mathbb S} 

\newcommand{\cG}{\mathcal{G}}

\newcommand{\cX}{\mathcal{X}}
\newcommand{\cY}{\mathcal{Y}}

\newcommand{\fK}{\mathfrak{K}}

\newcommand{\rF}{\mathrm{F}}

\newcommand{\rS}{\mathrm{S}}

\let\limsup\relax
\let\liminf\relax
\DeclareMathOperator* \limsup {\overline{lim}}
\DeclareMathOperator* \liminf {\underline{lim}}

\DeclareMathOperator*{\argmin}{arg\,min}

\newcommand{\inclim}[1]{\lim_{#1}\!\!\uparrow\!}

\let\originalleft\left
\let\originalright\right
\renewcommand{\left}{\mathopen{}\mathclose\bgroup\originalleft}
\renewcommand{\right}{\aftergroup\egroup\originalright}

\newcommand{\p}[1]{\left( #1 \right)}
\newcommand{\acc}[1]{\left\{ #1 \right\}}
\newcommand{\cro}[1]{\left[ #1 \right]}

\newcommand{\set}[2]{ \left\{ #1 \,\middle\vert\, #2  \right\}  }

\newcommand{\ind}[1]{\mathds{1}_{#1}}

\newcommand{\dpe}{\coloneqq}

\newcommand{\eol}{\notag\\}

\def\restriction#1#2{\mathchoice
              {\setbox1\hbox{${\displaystyle #1}_{\scriptstyle #2}$}
              \restrictionaux{#1}{#2}}
              {\setbox1\hbox{${\textstyle #1}_{\scriptstyle #2}$}
              \restrictionaux{#1}{#2}}
              {\setbox1\hbox{${\scriptstyle #1}_{\scriptscriptstyle #2}$}
              \restrictionaux{#1}{#2}}
              {\setbox1\hbox{${\scriptscriptstyle #1}_{\scriptscriptstyle #2}$}
              \restrictionaux{#1}{#2}}}
\def\restrictionaux#1#2{{#1\,\smash{\vrule height .8\ht1 depth .85\dp1}}_{\,#2}} 

\newcommand{\module}[1]{\left\lvert #1 \right\rvert}
\newcommand{\norme}[2][1]{\left\| #2 \right\|_{#1}}
\newcommand{\HD}{\mathfrak d_{\mathrm H}}

\newcommand{\ceil}[1]{\left\lceil #1 \right\rceil}
\newcommand{\floor}[1]{\left\lfloor #1 \right\rfloor}

\newcommand{\ball}[2][]{\mathrm{B}_{#1}\p{#2}}
\newcommand{\clball}[2][]{\overline{\mathrm{B}}_{#1}\p{#2}}

\newcommand{\Compacts}{\fK}

\newcommand{\intervalle}[4]{#1#2\,,#3#4}

\newcommand{\intervalleff}[2]{\intervalle{\left[}{#1}{#2}{\right]}}
\newcommand{\intervalleof}[2]{\intervalle{\left(}{#1}{#2}{\right]}}
\newcommand{\intervallefo}[2]{\intervalle{\left[}{#1}{#2}{\right)}}
\newcommand{\intervalleoo}[2]{\intervalle{\left(}{#1}{#2}{\right)}}

\newcommand{\intint}[2]{\left\llbracket#1\,,#2\right\rrbracket}

\renewcommand{\d}{\mathrm{d}}

\DeclareMathOperator \petito {o}
\DeclareMathOperator \grando {O}

\DeclareMathOperator \grad {grad}

\newcommand{\base}[1]{\mathrm e_{#1}}

\newcommand\ps[2]{\langle #1, #2 \rangle}

\newcommand{\Leb}{\operatorname{Leb}}

\newcommand{\E}[2][]{\mathbb{E}_{#1} \left[ #2\right]}

\newcommand{\Pb}[2][]{\mathbb{P}_{#1}\left( #2\right)}

\newcommand{\Path}[1]{\overset{#1}{\rightsquigarrow}}

\newcommand{\ResPath}[2]{\underset{ #1 }{\overset{#2}{\rightsquigarrow}}}

\newcommand{\edges}[2][]{\mathrm E_{#1}\p{#2}}

\newcommand{\FastE}{E_\mathrm{fast}}

\newcommand{\Pathedges}[1]{\edges[\mathrm{path}]{#1}}

\DeclareMathOperator \Corridor {Corridor}

\DeclareMathOperator \Tile {Tile}
\DeclareMathOperator \Hard {Hard}

\newcommand{\RF}{I}
\newcommand{\RFinf}{\underline I}
\newcommand{\RFsup}{\overline I}

\newcommand{\MonRFPhi}[1][\Phi]{I_{#1}^+ }
\newcommand{\MonRFinfPhi}[1][\Phi]{\underline I_{#1}^+ }
\newcommand{\MonRFsupPhi}[1][\Phi]{\overline I_{#1}^+ }

\newcommand{\FdTn}{J}

\newcommand{\MonRFppxC}{I_{ \mathrm{pp}(x), C}^+    }
\newcommand{\MonRFppxCinf}{\underline I_{ \mathrm{pp}(x), C}^+    }
\newcommand{\MonRFppxCsup}{\overline I_{ \mathrm{pp}(x), C}^+    }
\newcommand{\RFppx}{I_{\mathrm{pp}(x)}}

\newcommand{\RFppxC}{I_{\mathrm{pp}(x), C} }
\newcommand{\RFppxCinf}{\underline I_{\mathrm{pp}(x), C} }
\newcommand{\RFppxCsup}{\overline I_{\mathrm{pp}(x), C} }

\newcommand{\MonRF}{I^+}
\newcommand{\MonRFsup}{\overline{I}^+}
\newcommand{\MonRFinf}{\underline{I}^+}

\newcommand{\MonRFelem}{I_{\mathrm{elem} }^+}
\newcommand{\MonRFelemsup}{\overline{I}_{\mathrm{elem} }^+}
\newcommand{\MonRFeleminf}{\underline{I}_{\mathrm{elem} }^+}

\newcommand{\MonRFCTsup}{\overline{I}_{ \mathrm{cross }}^+}
\newcommand{\MonRFCTinf}{\underline{I}_{ \mathrm{cross }}^+}

\newcommand{\RFCT}{I_{ \mathrm{cross }}}
\newcommand{\RFCTsup}{\overline{I}_{ \mathrm{cross }}}
\newcommand{\RFCTinf}{\underline{I}_{ \mathrm{cross }}}

\newcommand{\MonRFPtS}{I_{ \mathrm{ps }}^+}
\newcommand{\MonRFPtSsup}{\overline{I}_{ \mathrm{ps }}^+}
\newcommand{\MonRFPtSinf}{\underline{I}_{ \mathrm{ps }}^+}

\newcommand{\RFPtS}{I_{ \mathrm{ps }}}
\newcommand{\RFPtSsup}{\overline{I}_{ \mathrm{ps }}}
\newcommand{\RFPtSinf}{\underline{I}_{ \mathrm{ps }}}

\newcommand{\LD}{\mathrm{LD}}
\DeclareMathOperator \Fav {Fav}

\newcommand{\Metrics}[1][b]{\mathcal D_{#1}}
\newcommand{\GeoMetrics}[1][a,b]{\mathcal D_{#1}^*}

\newcommand{\AdmNorms}{\mathcal N}
\newcommand{\exNorms}[1]{\mathcal N^{(\alpha)}}
\newcommand{\NoncritAdmNorms}{\mathcal N^*}

\newcommand{\UnifDistance}{\mathfrak d_\infty}

\newcommand{\Extension}[2][]{\overline{#2}^{#1} }

\newcommand{\ContHom}{\mathcal C_{\mathrm{Hom}}\p{\R^d, \R^+} }
\newcommand{\normeHom}[1]{\norme[\mathrm{Hom}]{#1} }

\newcommand{\HW}[2]{#1 \cro{ #2} }

\newcommand{\pc}{p_c(\Z^d)}

\newcommand{\EPT}[1][]{\tau_{#1}}

\newcommand{\FastEPT}[2][]{\tau_{#1}^{[#2]}}

\newcommand{\PathPT}[1]{\tau\p{ #1 } }

\newcommand{\PT}{\mathbf T}
\newcommand{\BoxPT}[1][n]{{\mathbf T}_{ #1 } }

\newcommand{\SPT}[1][n]{\widetilde {\mathbf T}_{#1} }
\newcommand{\CPT}[1][n]{\widehat {\mathbf T}_{#1} }

\newcommand{\BoxSPT}[1][n,X]{\widetilde {\mathbf T}_{ #1 } }

\newcommand{\FastBoxSPT}[2][n,X]{\widetilde {\mathbf T}^{[#2]}_{ #1 } }
\newcommand{\FFastBoxSPT}[1][n,X]{\widetilde{\mathbf{T}}^*_{ #1 } }

\newcommand{\CT}{\widetilde {\mathbf T}_{\mathrm{cross}} }
\newcommand{\FFastCT}{ \widetilde{\mathbf T}^*_{\mathrm{cross}}}

\newcommand{\PtS}{\widetilde{ \mathbf T}_{\mathrm{ps}} }

\newcommand{\RUB}{\widetilde {\mathbf B} }

\newcommand{\TimeConstant}{\mu}

\newcommand{\Gag}{\gamma_{\mathrm{\scriptscriptstyle L}}}
\newcommand{\Gam}{\gamma_{\mathrm{M}}}
\newcommand{\Gad}{\gamma_{\mathrm{\scriptscriptstyle R}}}

\newcommand{\bj}{\mathbf j}

\newcommand{\rmg}{\mathrm{\scriptscriptstyle L}}

\newcommand{\rd}{\mathrm{\scriptscriptstyle R}}

\newcommand{\GB}[1][z]{\hat g_{#1} }
 
\newcommand{\TBox}{\operatorname{Box}}

\makeatletter
\newcommand{\SymbolsTileFav}[1]{%
  \ensuremath{%
    \ifcase#1
    \or 
      *%
    \or 
      \dagger
    \or 
      \ddagger
    \or 
      **
    \or 
      \dagger\dagger
    \or 
      \ddagger \ddagger
    \or 
      \diamond 
    \fi
  }%
}   
\makeatother

\newcounter{tile}

\newcommand{\cTile}{\Tile^{\SymbolsTileFav{\thetile} }}
\newcommand{\cFav}{\Fav^{\SymbolsTileFav{\thetile} }}

\title{Large deviation principle at speed $n^d$ for the random metric in first-passage percolation}
\author{Julien \textsc{Verges}}
\date{\today}

\begin{document}
\maketitle
\begin{abstract}
	We consider the standard first passage percolation model on $\Z^d$ with bounded weights. We show a large deviation estimate for the probability that the rescaled passage time $\SPT$ restricted to a box is greater than a target metric. Moreover, the associated rate function can be written as the integral over the box of an elementary cost. Under the additional assumption that weights are bounded away from zero, $\SPT$ follows a large deviation principle (LDP) at speed $n^d$. Consequences include LDPs at speed $n^d$ for the point-to-point passage time, the crossing passage time and the random ball of radius $n$. Our strategy consists in proving the existence of $\lim_{n\to\infty}-\frac{1}{n^d}\log \Pb{\SPT \gtrsim g}$ for any seminorm $g$ with a multidimensional subadditivity argument, then using this result as an elementary building block to estimate $\Pb{\SPT\gtrsim D}$ for any metric $D$. 
\end{abstract}
\section{Introduction}
\label{sec : intro}

\subsection{Background and aims}
\label{subsec : Intro/Background}

\paragraph{First passage percolation.}
Let $d\ge2$ be an integer. Let $\bbE^d$ be the set of all non-oriented nearest-neighbour edges in $\Z^d$. A finite sequence $\pi\dpe (x_0,\dots, x_r)$ of elements of $\Z^d$ is called a \emph{discrete path} if for all $i\in \intint0{r-1}$, $(x_i,x_{i+1}) \in \bbE^d$.

Let $\nu$ be a probability distribution on $\intervallefo0\infty$ and denote by $a$ and $b$ the infimum and supremum of its support respectively. We consider a family $\p{\EPT[e]}_{e\in \bbE^d}$ of i.i.d.\ random variables with distribution $\nu$. The variable $\EPT[e]$ is called the \emph{passage time along the edge} $e$. Given a discrete path $\pi=(x_0,\dots , x_r)$, the \emph{passage time along} $\pi$ is defined as
\begin{align}
	\label{eqn : Intro/Background/def_PathPT}
	\PathPT{\pi} &\dpe \sum_{i=0}^{r-1} \EPT[ (x_i, x_{i+1})].
	\intertext{For $x,y\in \Z^d$, the \emph{passage time between $x$ and $y$} is defined as}
	\label{eqn : Intro/Background/def_PT_original}
	\PT(x,y) &\dpe \inf_{x \Path{\pi} y} \PathPT{\pi},
\end{align}
where the infimum is taken over all discrete paths whose endpoints are $x$ and $y$. The map $\PT(\cdot, \cdot)$ is a pseudometric on $\Z^d$. We call \emph{discrete geodesic} between $x$ and $y$ any minimizer in~\eqref{eqn : Intro/Background/def_PT_original}. Given a subset $A \subseteq \Z^d$ and $x,y\in A$, the \emph{passage time between $x$ and $y$ restricted to $A$} is defined as
\begin{equation}
\label{eqn : Intro/Background/def_PT}
	\BoxPT[A](x,y) \dpe \inf_{ \substack{x \Path{\pi} y\\ \pi \subseteq A} } \PathPT{\pi},
\end{equation}
where the infimum is taken over all discrete paths included in $A$ whose endpoints are $x$ and $y$. For all $n\ge 1$, we will write $\BoxPT[n] \dpe \BoxPT[\intint0n^d]$ to lighten the notations. A well-known result (Equation~2.4 in \cite{50yFPP}) states that under a moment condition on $\nu$, there exists a homogeneous function $\TimeConstant$ on $\R^d$, known as the \emph{time constant}, such that almost surely, for all $x\in \Z^d$,
\begin{equation}
	\label{eqn : Intro/Background/def_TimeConstant}
	\frac{\PT(0,nx)}{n} \xrightarrow[n\to\infty]{} \TimeConstant(x).
\end{equation}
The time constant is a norm if $\nu(\acc 0)< \pc$, where $\pc$ is the critical parameter for bond percolation in $\Z^d$. Otherwise $\mu(x) = 0$ for all $x\in \R^d$. As a consequence of~\eqref{eqn : Intro/Background/def_TimeConstant} the probability of an event of the form $\acc{\PT(0,n\base 1)\le \zeta n}$, with $\zeta < \TimeConstant(\base 1)$ or $\acc{\PT(0,n\base 1)\ge \zeta n}$, with $\zeta > \TimeConstant(\base 1)$ (the so-called lower tail and upper tail large deviation events) converges to $0$ as $n\to\infty$. In 1984, Kesten \cite{KestenStFlour} obtained estimates for the speed of convergence: there exists (Theorem~5.2 in \cite{KestenStFlour}) a convex decreasing function $\FdTn:\intervalleoo{a}{\TimeConstant(\base{1})} \rightarrow \intervalleoo{0}{\infty}$ such that for all $a < \zeta < \TimeConstant(\base 1)$,
\begin{equation}
	\label{eqn : Intro/Background/monotone_rate_function_n}
	\lim_{n\to\infty} -\frac1n \log\Pb{ \PT\p{0, n\base1}\le \zeta n } = \FdTn(\zeta).
\end{equation}
Moreover (Theorem~5.9 in \cite{KestenStFlour}), under the assumption $b<\infty$, for all $\mu(\base 1)<\zeta < b$,
\begin{equation}
		\label{eqn : Intro/Background/LD_order_n^d}
		0 < \liminf_{n\to\infty} -\frac{1}{n^d} \log\Pb{ \PT\p{0, n\base1}\ge \zeta n } \le \limsup_{n\to\infty} -\frac{1}{n^d} \log\Pb{ \PT\p{0, n\base1}\ge \zeta n } < \infty.
\end{equation}
These results can be generalized to any direction with minor proof adaptations. Under an additional regularity assumption on $\nu$, Basu, Ganguly and Sly \cite{Bas21} proved in 2021 that the limit
\begin{equation}
    \label{eqn : Intro/Background/LD_order_n^d/ExistenceLimit}
    \lim_{n\to\infty} -\frac{1}{n^d} \log\Pb{ \PT\p{0, n\base1}\ge \zeta n }
\end{equation}
exists in the case $d=2$ and as stated in their paper, their arguments are still valid for any dimension and in any direction. 

The goals of this article are to:
\begin{itemize} \item prove an analogue of \eqref{eqn : Intro/Background/LD_order_n^d/ExistenceLimit} that involves the probability that $\PT$ is lower bounded by a target \emph{pseudometric}; \item in the special case where $a>0$, deduce a large deviation principle (LDP) at speed $n^d$ for the random metric. \end{itemize} By the so-called contraction principle, we get LDPs for a variety of functionals of the random pseudometric, for any bounded distribution.

\subsection{Framework}
\label{subsec : Intro/Framework}
In the whole paper, we assume that $b<\infty$.

\paragraph{The rescaled random metric.} Given a polygonal path $\pi=(x_0,\dots, x_r)$ with $\forall i\in \intint0r, x_i \in \R^d$, we generalize~\eqref{eqn : Intro/Background/def_PathPT} by defining
\begin{equation}
  \label{eqn : Intro/main_thm/generalized_PathPT1}
  \PathPT{\pi} \dpe \sum_{i=0}^{r-1}\PathPT{x_i, x_{i+1}},
\end{equation}
where for all $z,z'\in \R^d$,
\begin{equation}
  \label{eqn : Intro/main_thm/generalized_PathPT2}
  \PathPT{z,z'}  \dpe%
    \begin{cases}
      \EPT[\acc{v,v'}] \norme{z-z'} \qquad &\text{if there exist adjacent vertices $v,v'\in\Z^d$ such that $z,z'\in \intervalleff{v}{v'}$, }\\
      b\norme{z-z'}          &\text{otherwise.}
    \end{cases}
\end{equation}
Given $A\subseteq \R^d$ and $x,y\in A$ we then define
\begin{equation}
\label{eqn : Intro/Framework/BoxPT}
  \CPT[A](x,y) \dpe \inf_{ \substack{x \Path{\pi} y \\ \pi \subseteq A} }  \PathPT{\pi},
\end{equation}
where the infimum is taken over all finite sequences included in $A$ whose endpoints are $x$ and $y$. To lighten the notations, we will write, for all $n\ge 1$, $\CPT[n] \dpe \CPT[\intervalleff0n^d]$. We show in Proposition \ref{prop : GeoMetrics/Geodesics/extension_tranquille} that $\CPT$ is indeed a pseudometric that extends $\BoxPT$. We then define the rescaled metric
\begin{equation}\begin{split}
  \label{eqn : Intro/Framework/SPT}
  \SPT : \intervalleff01^d \times \intervalleff01^d &\longrightarrow \intervallefo0\infty \\
  (x,y) &\longmapsto \frac1n \CPT(nx,ny).
\end{split}\end{equation}

\paragraph{The limit space.} For all compact subsets $X\subseteq \R^d$, we denote by $\Metrics(X)$ the set of pseudometrics $D$ on $X$ such that for all $x,y\in X$, $D(x,y) \le b\norme[1]{x-y}$. We endow $\Metrics(X)$ with the uniform distance, defined for all $D_1,D_2 \in \Metrics(X)$ by
\begin{equation}
  \UnifDistance(D_1, D_2) \dpe \max_{x,y\in X} \module{D_1(x,y) - D_2(x,y)}.
\end{equation}
It follows from the Arzelà--Ascoli theorem that $(\Metrics(X),\UnifDistance)$ is compact. If $X = \intervalleff01^d$, we simply write $\Metrics = \Metrics\p{\intervalleff01^d}$. Let $\AdmNorms$ denote the set of seminorms on $\R^d$ upper bounded by $b\norme[1]\cdot$, and $\AdmNorms^*$ the set of seminorms upper bounded by $(b-\eps)\norme\cdot$, for some $\eps>0$. One key property of a pseudometric $D\in \Metrics$ is that around almost every $z\in \intervalleff01^d$, $D$ ``behaves like'' a seminorm $(\grad D)_z \in \AdmNorms$ we will properly define in Section~\ref{sec : Gradient}. The process $(\SPT)_{n\ge1}$ takes values in the set $\Metrics$.

\paragraph{Large deviations.}
We give here some large deviation definitions and tools. See Dembo and Zeitouni (2009) \cite{LDTA} for the general theory.
\begin{Definition}
  Let $\cX$ be a Hausdorff topological space. We call \emph{rate function} a lower semi-continuous map $J:\cX\rightarrow \intervalleff0\infty$, i.e. a map whose sublevels are closed. We further say that $J$ is a \emph{good rate function} if its sublevels are compact.

  We say that a random process $(X_n)_{n\ge 1}$ with values in $\cX$ satisfies the \emph{large deviation principle}, at speed $a_n$, with the rate function $J$ if for every Borel set $A\subseteq \cX$,
\begin{equation}
  \label{eqn : Intro/Background/LDP_general case}
  \inf_{x\in \overline{A}} J(x)%
  \le \liminf_{n\to\infty}-\frac{1}{a_n} \log \Pb{ X_n \in A } %
  \le \limsup_{n\to\infty}-\frac{1}{a_n} \log \Pb{ X_n \in A } %
  \le \inf_{x\in \mathring{A}} J(x).
\end{equation}  
\end{Definition}
In this article only the case $a_n = n^d$ is considered. Lemma~\ref{lem : Intro/LD/UpperLowerRF} will be of constant use.
\begin{Lemma}
  \label{lem : Intro/LD/UpperLowerRF}
  Let $(\cX, \d_\cX)$ be a metric space and $(X_n)_{n\ge 1}$ a random process with values in $\cX$. Define, for all $x\in \cX$,
  \begin{equation}\begin{split}\label{eqn : Intro/LD/UpperLowerRF}
    \overline{J}(x)&\dpe\inclim{\eps \to 0} \limsup_{n\to\infty} -\frac{1}{n^d}\log \Pb{\d_\cX(x, X_n)\le \eps }\\
    \text{and  }\underline{J}(x)&\dpe\inclim{\eps \to 0} \liminf_{n\to\infty} -\frac{1}{n^d}\log \Pb{\d_\cX(x, X_n)\le \eps }.
  \end{split}\end{equation}
  Then
  \begin{enumerate}[(i)]
    \item \label{item : Intro/LD/UpperLowerRF/RF}%
    $\overline J$ and $\underline J$ are rate functions on $\cX$.
    \item For every open set $U\subseteq \cX$,
      \begin{equation}
        \label{eqn : Intro/LD/UpperLowerRF/UB}
        \limsup_{n\to\infty}-\frac{1}{n^d}\log\Pb{X_n \in U} \le \inf_{x\in U} \overline{J}(x).
      \end{equation}
    \item For every compact set $K\subseteq \cX$,
      \begin{equation}
        \label{eqn : Intro/LD/UpperLowerRF/LB}
        \liminf_{n\to\infty}-\frac{1}{n^d}\log\Pb{X_n \in K} \ge \min_{x\in K} \underline{J}(x).
      \end{equation}
  \end{enumerate}
\end{Lemma}
Although they are straightforward adaptations of the proof of Theorem~4.1.11 in~\cite{LDTA}, we prove them in Appendix~\ref{appsec : LD} for completeness, and because the formalisms are quite different. Note that on a compact metric space, with the notations of Lemma~\ref{lem : Intro/LD/UpperLowerRF}, $J\dpe\overline J = \underline J$ implies that $(X_n)_{n\ge1}$ satisfies a LDP at speed $n^d$, with the rate function $J$.

\subsection{Main theorems}

\paragraph{One-sided large deviation principle for the rescaled pseudometric.} In the case where $\nu$'s support is bounded but not necessarily bounded away from $0$, we do not prove a large deviation principle for $\SPT$. Instead, we get a one-sided version of Lemma \ref{lem : Intro/LD/UpperLowerRF}'s assumptions. For all $D\in \Metrics$, we define
\begin{equation}\begin{split}\label{eqn : MonRF}
  \MonRFinf(D) &\dpe \lim_{\eps \to 0} \liminf_{n\to\infty}-\frac{1}{n^d}\log \Pb{ \forall x,y \in \intervalleff01^d,\quad \SPT(x,y) \ge D(x,y) -\eps } \\ \text{  and  }\MonRFsup(D) &\dpe \lim_{\eps \to 0} \limsup_{n\to\infty}-\frac{1}{n^d}\log \Pb{ \forall x,y \in \intervalleff01^d,\quad \SPT(x,y) \ge D(x,y) -\eps }.
\end{split}\end{equation}
Note that both $\MonRFinf$ and $\MonRFsup$ take values in $\intervalleff0\infty$. They are rate functions on $\Metrics$ (see Lemma \ref{lem : LD/RateFunctions}). Our first result states that they are equal and have an integral expression.
\begin{Theorem}
  \label{thm : MonLDP}
  Assume that $b<\infty$. Then for all $D\in\Metrics$, $\MonRFinf(D) = \MonRFsup(D) \eqqcolon \MonRF(D)$. Moreover:
  \begin{enumerate}[(i)]
    \item There exists a function $\MonRFelem : \AdmNorms \rightarrow \intervalleff0\infty$ such that for all $D\in \Metrics$,
    \begin{equation}
        \label{eqn : MonLDP/Expression}
      \MonRF(D) = \int_{\intervalleff01^d} \MonRFelem\p{ (\grad D)_z } \d z.   
    \end{equation}
    \item For all $D\in \Metrics$, $\MonRF(D) = 0$ if and only if for all $x,y\in \intervalleff01^d$, $D(x,y) \le \mu(x-y)$.
  \end{enumerate}
\end{Theorem}

\paragraph{Large deviation principle for the rescaled pseudometric.} Our second main result is the large deviation principle at speed $n^d$ for $\SPT$, in the case where $\nu$'s support is bounded and bounded away from $0$. Given Theorem \ref{thm : MonLDP}, one expects to have
\begin{equation}
  \label{eqn : Intro/Main/LDP/Intuition1}
  \forall D\in \Metrics, \quad \Pb{\SPT \approx D} \le \Pb{\SPT \gtrsim D} \approx \exp\p{ - \MonRF(D)n^d }.
\end{equation}
In at least two cases, this inequality is not sharp: indeed, if for some $x,y\in \intervalleff01^d$, $D(x,y) < a \norme[1]{x-y}$, then $\acc{\SPT \approx D} = \emptyset$. Moreover, $\SPT$ cannot approximate non-geodesic metrics. Outside those two regimes, the inequality \eqref{eqn : Intro/Main/LDP/Intuition1} is actually sharp. Let us make this statement precise.
\begin{Definition}
  \label{def : Intro/Main/LDP/GeoMetrics}
  Let $X$ be a compact subset of $\R^d$. We denote by $\GeoMetrics(X)$ the subset of $\Metrics(X)$ composed of metrics $D$ such that:
  \begin{enumerate}[(i)]
    \item For all $x,y\in X$, \begin{equation} \label{eqn : Intro/Main/LDP/GeoMetrics/Bounds} a\norme[1]{x-y} \le D(x,y) \le b\norme[1]{x-y}. \end{equation}
    \item For all $x,y\in X$, there exists a map $\sigma : \intervalleff{0}{T} \rightarrow X$ that satisfies $\sigma(0)=x$, $\sigma(T)=y$, and for all $0\le s \le t\le T$, \[ D\p{\sigma(s), \sigma(t)} = t-s. \]Such a map will be called a $D$-\emph{geodesic} from $x$ to $y$.
  \end{enumerate}
  We write $\GeoMetrics = \GeoMetrics\p{\intervalleff01^d}$.
\end{Definition}
For all $D \in \Metrics$, we define
\begin{equation}
    \label{eqn : Intro/Main/LDP/RF}
    \RF(D) \dpe \begin{cases}
      \MonRF(D) &\text{ if } D\in \GeoMetrics. \\
      \infty &\text{ otherwise. }
    \end{cases}
\end{equation}
\begin{Theorem}
  \label{thm : LDP}
  Assume that $0<a<b<\infty$. Then $(\SPT)_{n \ge 1} $ follows the large deviation principle at speed $n^d$, with the rate function $\RF$: for all Borel sets $A\subseteq \Metrics$, 
  \begin{equation}
  \label{eqn : Intro/Main/LDP/LDP_statement}
    \inf_{D \in \overline{A}} \RF(D)%
    \le \liminf_{n\to\infty}-\frac{1}{n^d} \log \Pb{ \SPT\in A } %
    \le \limsup_{n\to\infty}-\frac{1}{n^d} \log \Pb{ \SPT\in A } %
    \le \inf_{D \in \mathring{A}} \RF(D).
\end{equation}
\end{Theorem}

\subsection{Applications} 
\label{subsec : Intro/applications}
Under the assumptions of Theorem~\ref{thm : LDP}, by the so-called \emph{contraction principle} (see e.g. Theorem~4.2.1 in \cite{LDTA}), any image of the process $(\SPT)_{n\ge 1}$ through a continuous map follows the LDP at speed $n^d$. Examples include the rescaled ball of radius $n$, the point-to-point passage time, the point-to-sphere passage times and the crossing times. In some cases, the corresponding rate function is defined by a variational problem that can be expressed either as an infimum spanning over a set of pseudometrics. 

For all real numbers $v<w$, we denote by $\cG\p{\intervalleff vw^d}$ the set of measurable functions $z\mapsto g_z$ defined on $\intervalleff vw^d$, with values in $\AdmNorms$. Given $g \in \cG\p{\intervalleff vw^d}$, we define the pseudometric
\begin{align*}
   \mathbf D[g] : \p{\intervalleff{v}{w}^d}^2 &\longrightarrow \intervallefo0\infty \eol 
   (x,y) &\longmapsto \inf%
     \set{ \int_0^T g_{\gamma(t) }\p{ \gamma'(t)}\d t }{\begin{array}{c} \gamma : \intervalleff0T \rightarrow \intervalleff{v}{w}^d \text{ Lipschitz s.t.\ }\\ \gamma(0)=x,\; \gamma(T)=y\end{array} }. \label{eqn : Gradient/inverse_def}
 \end{align*}
Its properties are discussed in Proposition \ref{prop : Gradient/inverse}.

\paragraph{Point-to-point passage time.} Corollary \ref{cor : Intro/applications/point_point_Box} extends \eqref{eqn : Intro/Background/LD_order_n^d/ExistenceLimit} to any dimension, any direction, and more crucially, without any continuity assumption on $\nu$; we consider the point-to-point time restricted in a box though. Corollary \ref{cor : Intro/applications/point_point} gives the analogous statement for the actual point-to-point passage time in the case where the atom at $0$ is subcritical.
\begin{Corollary}
\label{cor : Intro/applications/point_point_Box}
	Assume that $0\le a < b <\infty$. Let $x \in \R^d\setminus \acc 0$ and $C> \norme x$ be an integer. Then the process $\p{\frac{\BoxPT[\intervalleff{-Cn}{Cn}^d](0, \floor{nx} )}{n}}_{n\ge 1}$ follows the LDP at speed $n^d$ with a good rate function $\RFppxC : \R \rightarrow \intervalleff0\infty$. Moreover:
    \begin{enumerate}[(i)]
        \item For all $\zeta \in \R\setminus \intervalleff{a\norme x}{b\norme x}$, $\RFppxC(\zeta)=\infty$.
        \item For all $\zeta \in  \intervalleff{a\norme x}{b\norme x}$,
        \begin{equation}\begin{split}
            \RFppxC(\zeta) %
                &= \min \set{\int_{\intervalleff{-C}{C}^d } \MonRFelem\p{(\grad D)_z} \d z }{D \in \Metrics\p{\intervalleff{-C}{C}^d} \text{ s.t.\ } D(0,x) \ge \zeta }\\
                &= \min \set{\int_{\intervalleff{-C}{C}^d } \MonRFelem\p{g_z} \d z }{g \in \cG\p{\intervalleff{-C}{C}^d} \text{ s.t.\ } \mathbf D[g](0,x) \ge \zeta }.
        \end{split}\end{equation}
        \item The function $\RFppxC$ is continuous and nondecreasing on $\intervalleff{a\norme x}{b\norme x}$.
        \item For all $a\norme x \le \zeta < b\norme x$,
	\begin{equation}
		\lim_{n\to\infty} -\frac{1}{n^d} \log \Pb{ \frac{\BoxPT[\intervalleff{-Cn}{Cn}^d](0,\floor{nx})}{n} \ge \zeta} = \RFppxC(\zeta).
	\end{equation}
    \end{enumerate}
\end{Corollary}

\begin{Corollary}
\label{cor : Intro/applications/point_point}
	Assume that $0\le a < b <\infty$ and $\nu\p{\acc 0}<\pc$. Let $x \in \R^d\setminus \acc 0$. Then the process $\p{\frac{\PT(0, \floor{nx} )}{n}}_{n\ge 1}$ follows the LDP at speed $n^d$ with a good rate function $\RFppx : \R \rightarrow \intervalleff0\infty$. Moreover:
    \begin{enumerate}[(i)]
        \item For all $\zeta \in \R\setminus \intervalleff{a\norme x}{b\norme x}$, $\RFppx(\zeta)=\infty$.
        \item For sufficiently large $C$, $\RFppx = \RFppxC$.
        \item The function $\RFppx$ is continuous and nondecreasing on $\intervalleff{a\norme x}{b\norme x}$.
        \item For all $a\norme x \le \zeta < b\norme x$,
	\begin{equation}
		\lim_{n\to\infty} -\frac{1}{n^d} \log \Pb{ \frac{\PT(0,\floor{nx})}{n} \ge \zeta} = \RFppx(\zeta).
	\end{equation}
    \end{enumerate}
\end{Corollary}
Note that by the definition of the time constant~\eqref{eqn : Intro/Background/def_TimeConstant} and the order of the probability of the upper tail large deviation events~\eqref{eqn : Intro/Background/LD_order_n^d}, for all $x\in \R^d\setminus\acc 0$ and $a\norme x\le \zeta \le b\norme x$,
\begin{equation*}
 	\RFppx(\zeta)>0 \Leftrightarrow \zeta > \mu(x).	
\end{equation*} 

\paragraph{Point-to-sphere passage time.}
For all $n\ge 1$, the \emph{rescaled point-to-sphere passage time} is defined as
\begin{equation}
\label{eqn : Intro/applications/PtS}
	\PtS(n) \dpe \frac1n\BoxPT[\intint{-n}{n}^d]\p{0, \partial\p{\intervalleff{-n}{n}^d}} = \frac1n\cdot \min_{ {0 \Path\gamma \partial\p{\intervalleff{-n}{n}^d} } } \tau(\gamma),
\end{equation}
where the minimum spans over the family of discrete paths starting at $0$ and ending at an integer point of $\partial\p{\intervalleff{-n}{n}^d}$. 

\begin{Corollary}
    \label{cor : Intro/applications/PtS}
    Assume $b<\infty$. The process $\p{\PtS(n)}_{n\ge 1}$ follows the LDP at speed $n^d$, with the  rate function $\RFPtS : \intervalleff{a}{b} \rightarrow \intervalleff0\infty$. Moreover:
    \begin{enumerate}[(i)]
        \item For all $\zeta \in  \intervalleff{a}{b}$,
        \begin{equation}\begin{split}
            \RFPtS(\zeta) %
                &= \min \set{\int_{\intervalleff{-1}{1}^d } \MonRFelem\p{(\grad D)_z} \d z }{D \in \Metrics\p{\intervalleff{-1}{1}^d} \text{ s.t.\ } D\p{0,\partial\p{\intervalleff{-1}{1}^d} } \ge \zeta }\\
                &= \min \set{\int_{\intervalleff{-1}{1}^d } \MonRFelem\p{g_z} \d z }{g \in \cG\p{\intervalleff{-1}{1}^d} \text{ s.t.\ } \mathbf D[g]\p{0,\partial\p{\intervalleff{-1}{1}^d} } \ge \zeta }.
        \end{split}\end{equation}
        \item The function $\RFPtS$ is continuous and nondecreasing.
        \item For all $a \le \zeta < b$,
	\begin{equation}
		\lim_{n\to\infty} -\frac{1}{n^d} \log \Pb{ \PtS(n) \ge \zeta} = \RFPtS(\zeta).
	\end{equation}
    \end{enumerate}
\end{Corollary}
The author uses Corollary \ref{cor : Intro/applications/PtS} in \cite{Cutsets} to prove an estimate on the tail of the number of disjoint closed cutsets between $0$ and $\infty$ in supercritical Bernoulli percolation. For that reason, it is crucial that the case where $\nu(\acc 0)> \pc $ is covered.

\paragraph{Crossing times.}
For all $n\ge 1$, the \emph{rescaled crossing times} of the box $\intervalleff0n^d$ are defined as
\begin{equation}
\label{eqn : Intro/applications/CT}
	\CT(n) \dpe \p{ \SPT(H_1, H_1'), \dotsc, \SPT(H_d, H_d')  },
\end{equation}
where $H_i \dpe \acc{z =(z_1,\dots, z_d) \in \intervalleff01^d \Bigm| z_i= 0 }$ and $H_i' \dpe \acc{z =(z_1,\dots, z_d) \in \intervalleff01^d \Bigm| z_i= 1 }$.

\begin{Corollary}
\label{cor : Intro/applications/crossing}
	Assume that $0\le a < b < \infty$. The process $\p{ \CT(n)}_{n\ge 1}$ follows the LDP at speed $n^d$ with the good rate function
	\begin{align}
    \RFCT : \intervalleff{a}{b}^d &\longrightarrow \intervalleff{0}{\infty}\eol
    \label{eqn : Intro/applications/def_FdTCT}%
    \zeta = (\zeta_1,\dots, \zeta_d) &\longmapsto \MonRFelem\p{\mathfrak g^\zeta},
  \end{align}
  where $\mathfrak g^\zeta$ is the seminorm defined by
  \begin{equation}
  \label{eqn : Intro/applications/def_magic_norm}
  		\mathfrak g^\zeta(u) \dpe \max_{1\le i \le d}\p{\zeta_i \module{u_i}}.
  \end{equation}
  Moreover:
  \begin{enumerate}[(i)]
  	\item \label{item : Intro/applications/crossing/continuite}%
  	$\RFCT$ is continuous on $\intervalleff{a}{b}^d$.
  	\item \label{item : Intro/applications/crossing/croissance}%
  	$\RFCT$ is nondecreasing on $\intervalleff{a}{b}^d$ for the componentwise order, \textit{i.e.} for all $\zeta= (\zeta_1,\dots, \zeta_d)$ and $\zeta'=(\zeta_1',\dots, \zeta_d')$ in $\intervalleff{a}{b}^d$, if $\zeta_i \le \zeta_i'$ for all $i$, then $\RFCT(\zeta) \le \RFCT(\zeta')$.
  	\item \label{item : Intro/applications/crossing/convexite}%
  	$\RFCT$ is separately convex on $\intervalleff{a}{b}^d$, \textit{i.e.} for all $1\le i \le d$ and $\zeta_1,\dots, \zeta_{i-1}, \zeta_{i+1},\dots, \zeta_d \in \intervalleff ab$, the function
  	\begin{equation*}
  		t \mapsto \RFCT\p{\zeta_1 ,\dots, \zeta_{i-1}, t, \zeta_{i+1}, \dots, \zeta_d }
  	\end{equation*}
  	is convex on $\intervalleff ab$.
  	\item \label{item : Intro/applications/crossing/cas_infty}%
  	$\RFCT(b,\dots,b)<\infty$ if and only if $\nu(\acc b)>0$. 
  \end{enumerate}
  In particular, for all $\zeta \in \intervallefo{a}{b}^d$,
  \begin{equation}
  \label{eqn : Intro/applications/crossing_monotone_evt}
  		\lim_{n\to\infty}-\frac{1}{n^d}\log \Pb{ \bigcap_{i=1}^d \acc{\BoxSPT[n](H_i, H_i') \ge \zeta_i} } = \RFCT(\zeta).
  \end{equation}
\end{Corollary}

Note that Corollaries \ref{cor : Intro/applications/point_point_Box},~\ref{cor : Intro/applications/point_point}, \ref{cor : Intro/applications/PtS} and \ref{cor : Intro/applications/crossing} are stated in a more general context than Theorem~\ref{thm : LDP}. The assumption $\nu\p{\acc 0}<\pc$ in Corollary~\ref{cor : Intro/applications/point_point} is used to show that with high probability, any geodesic from $0$ to $nx$ is included in a box of linear size. 

\paragraph{Rescaled growing ball.}
For all $n\ge 1$, let us consider the \emph{rescaled growing ball}
\begin{equation}
	\RUB(n) \dpe \frac1n \overline{\acc{x\in \R^d \Bigm| \PT(0,\floor x) \le n }}.
\end{equation}
\begin{Corollary}\label{cor : Intro/applications/ball}
	Assume that $0 <a <b<\infty$. The process $\p{\RUB(n)}_{n\ge 1}$ satisfies the LDP at speed $n^d$ with a good rate function, in the space of compact sets of $\R^d$ endowed with the Hausdorff distance defined as
	\begin{equation}
		\HD(K_1, K_2) \dpe \inf\acc{\eps >0 \Bigm| K_1 \subseteq K_2+\clball{0,\eps} \text{ and } K_2 \subseteq K_1+\clball{0,\eps}  }.
	\end{equation}
\end{Corollary}
Little is known about the rate function mentioned in Corollary~\ref{cor : Intro/applications/ball} beyond the properties granted by the contraction principle.

\subsection{Sketch of the proof}

In Section \ref{sec : ERF} we prove via a subadditivity argument that for all $g\in \AdmNorms$, there exists $\MonRFelem(g) \in \intervalleff0\infty$ such that \begin{equation}\label{eqn : Intro/Sketech/MonRFelem}
    \Pb{\SPT \gtrsim g} \approx \exp\p{- \MonRFelem(g)n^d },
\end{equation}
i.e.\ the special case of Theorem \ref{thm : MonLDP} where $D$ is the metric associated with $g$. We also establish properties of $\MonRFelem$.

Section \ref{sec : Gradient} defines the gradient of a metric. We prove that the rescaled restriction of a metric to small tiles converges to the metric associated with its gradient. Moreover, we write the length of Lipschitz paths with respect to a metric as an integral involving the gradient. 

Section \ref{sec : MonLDP} concludes the proof of Theorem \ref{thm : MonLDP}, which essentially amounts to showing that for all $D\in \Metrics$,
\[
    \Pb{\SPT \gtrsim D} \approx \exp\p{ -\p{\int_{\intervalleff01^d}\MonRFelem\p{(\grad D)_z}\d z }n^d  }.
\]
The general idea is to approximate $D$ by a metric whose gradient is constant on each element of a tile partition of $\intervalleff01^d$ then use \eqref{eqn : Intro/Sketech/MonRFelem} on each tile.

Section \ref{sec : GeoMetrics} contains the topological preliminaries used in the proof of Theorem \ref{thm : LDP}.

Section \ref{sec : LDP} is devoted to the proof of Theorem \ref{thm : LDP}. The core difficulty is to show that if $D\in \GeoMetrics$, modifying a small number of well-chosen edge weights in any realization of $\acc{\SPT \gtrsim D}$ creates a realization of $\acc{\SPT \approx D}$.

Finally, the corollaries are proven in Section \ref{sec : App}.
\subsection{Open questions and related works}
\label{subsec : Intro/oq_and_rw}

\paragraph{Lower tail large deviations.}
Theorem~\ref{thm : LDP} does not provide a satisfactory estimate for $\Pb{\SPT \simeq D}$ for metrics $D\in\GeoMetrics$ such that $D\le \TimeConstant$. Given~\eqref{eqn : Intro/Background/monotone_rate_function_n}, one expects that the appropriate speed for the study of such events is $n$. The author proved in \cite{LDP_Speed_n} that under an exponential moment assumption, $\SPT$ follows the LDP at speed $n$.

\paragraph{Upper tail large deviations under milder assumptions.}
Chow and Zhang proved in 2003 \cite{Cho03} that under a finite exponential moment condition, the probability of the upper tail large deviation event for the face-face passage time, i.e. $\acc{\BoxSPT[n](H_1, H_1') \ge \zeta}$ with $\zeta > \TimeConstant(\base 1)$ and the notations of Corollary~\ref{cor : Intro/applications/crossing}, is of order $\exp\p{ -\Theta(n^d)}$ and proved the existence of the rate function. Note that their result is neither a special case nor an extension of Corollary~\ref{cor : Intro/applications/crossing} because while the former is valid in a larger framework, the latter provides the LDP for all directions simultaneously.

Cranston, Gauthier and Mountford proved in 2009 \cite{Cra09} a criterion for $\Pb{\frac1n\PT(0,n\base 1)\ge \zeta}$ with $\zeta>\TimeConstant(\base 1)$ to be of order $\exp\p{-\Theta(n^d)}$, for a certain family of distributions. On the other hand Cosco and Nakajima recently showed \cite{Cos23} that if $\nu\p{\intervallefo{t}{\infty}} \simeq \exp\p{-\alpha t^r}$, with $\alpha>0$ and $0<r\le d$, then
\begin{equation*}
	\Pb{\frac1n\PT(0,n\base 1)\ge \zeta}=%
		\begin{cases}
			\exp\p{-\Theta(n^r)} &\text{ if } 0<r<d,\\
			\exp\p{-\Theta\p{\frac{n^d}{(\log n)^{d-1} }}} &\text{ if } r=d.
		\end{cases}
\end{equation*}
Moreover, they provide a description of the associated rate functions with a variational formula. 

In particular, outside the bounded support assumption, for some distributions, the probability of different upper tail large deviation events may be of different order. Hence, for these distributions, a single LDP for the random metric as in Theorem~\ref{thm : LDP} cannot give appropriate estimates for the probability of all upper tail large deviation behaviours; rather, several LDPs, at different speeds, would be needed.

\paragraph{The subadditivity argument.}
The main idea of Chow and Zhang \cite{Cho03} to prove the existence of the limit
\[
\lim_{n\to\infty} -\frac1{n^d}\log \Pb{\BoxSPT[n](H_1, H_1') \ge \zeta},
\]
with $\zeta > \TimeConstant(\base 1)$, is to assemble $k^d$ independent configurations on $\intint0n^d$ satisfying $\acc{\BoxSPT[n](H_1, H_1') \ge \zeta}$ to create a configuration on $\intint0{kn}^d$ satisfying $\acc{\BoxSPT[kn](H_1, H_1') \ge \zeta}$, leading to
\[
\Pb{\BoxSPT[n](H_1, H_1') \ge \zeta}^{k^d} \le \Pb{\BoxSPT[kn,\intervalleff01^d](H_1, H_1') \ge \zeta},
\]
and the conclusion is standard. However, this simple plan fails because there is no general link between the crossing times of the small boxes and the crossing time of the large one. The solution proposed by the authors is to consider a subset $Z_n$ of $\acc{\BoxSPT[n](H_1, H_1') \ge \zeta}$ with equivalent $\log$-probability, such that configurations on $\intint0n^d$ satisfying $Z_n$ are compatible, in the sense that assembling them creates no ``shortcuts''. 
 
To prove the existence of the limit
\[
\lim_{n\to\infty} -\frac1{n^d}\log \Pb{\frac1n\PT(0, n\base 1) \ge \zeta},
\]
Basu, Ganguly and Sly \cite{Bas21} adopt a similar approach, with an added subtlety regarding the way configurations on the small boxes are assembled: each small box is fragmented into smaller tiles and the configuration on the large box is constructed so that, for example, its top left corner is made of the top left corner tiles of all the small boxes. To ensure that corresponding tiles are compatible, they essentially proceed as follows:
\begin{enumerate}
	\item Prove that for a well-chosen (random but within a deterministic range) tile size, most of the tiles are \emph{stable}, in the sense that the restriction of the random metric to each tile is similar to some (random, tile-dependent) norm.
	\item Define $Z_n$ as a subset of $\acc{\frac1n\PT(0, n\base 1) \ge \zeta}$ for which the good tile size and the norms on each tile are prescribed, chosen so that its $\log$-probability is equivalent to $ \log \Pb{\frac1n\PT(0, n\base 1) \ge \zeta}$.
\end{enumerate}

In this article, we adopt a somewhat different approach, as the subadditivity argument is only used in the proof of~\eqref{eqn : Intro/Sketech/MonRFelem}. The fact that we are, at this step, only interested in uniform environments makes our version of the argument less sophisticated than the one in \cite{Bas21}, as tile compatibility is free. This enables us to avoid the technical difficulty of proving tile stability properties for the random metric that are uniform in the realization; this work is done in a simpler deterministic setting here.  

\paragraph{Streams and maximal flows.}
The general philosophy of our proof is inspired by the recent work of Dembin and Théret \cite{Dem21+} on large deviations of the streams and the maximal flows: in order to estimate the probability that there exists an admissible stream resembling a target stream, they first study the easier case where the target stream is uniform. They then use the estimate in the elementary case as a building block to get an estimate in the general case.

\subsection{Notations}
\label{subsec : Intro/notations}
\paragraph{Elements of $\R^d$.}
We denote by $\p{\base i}_{i=1}^d$ the canonical basis of $\R^d$ and $\ps{\cdot}{\cdot}$ the standard inner product on $\R^d$. We denote by $\le$ the coordinatewise order on $\R^d$.
For all $x=(x_1,\dots, x_d)\in \R^d$, we define
\begin{equation}
	\module x \dpe \p{\module{x_1}, \dots, \module{x_d} }, \; \floor{x} \dpe \p{\floor{x_1},\dots , \floor{x_d}} \text{  and  } \ceil{x} \dpe \p{\ceil{x_1},\dots , \ceil{x_d}}.
\end{equation}
A sequence $(x_n)_{n\ge 1}$ of elements of $\R^d$ is said to have \emph{monotone coordinates} if for all $i\in\intint1d$, the sequence $\p{\ps{x_n}{\base i}}_{n\ge 1}$ is monotone.

We denote by $\bbS$ the unit sphere for $\norme\cdot$.

\paragraph{Superscripts.}
Whenever several families of edge passage times are considered and distinguished by superscript, the associated metrics $\PT$, $\SPT$, etc. will be marked with the same superscript. For example, $\PT^{(1)}$ is the process defined by~\eqref{eqn : Intro/Background/def_PT_original}, with $\EPT[e]^{(1)}$ instead of $\EPT[e]$.

\paragraph{Cardinality and volume.}
We denote by $\#A$ the cardinality of the set $A$. If $A$ is a Borel subset of $\R^d$ we denote by $\Leb(A)$ its Lebesgue measure.

\paragraph{Edges.}
Given an edge $e$, the notation ``$z\in e$'' will mean ``$z$ belongs to the segment between endpoints of $e$''. For all $A\subseteq \R^d$ we will denote by $\edges{A}$ the set of all edges $e\in\bbE^d$ included in $A$. For all discrete paths $\alpha= (\alpha_j)_{j=0}^r$, we define
\begin{equation}
	\Pathedges{\alpha} \dpe \acc{(\alpha_j, \alpha_{j+1}), j\in\intint{0}{r-1}},
\end{equation}
which is consistent with the previous definition of $\edges{\cdot}$.

\paragraph{Paths.}
Given a polygonal path $\gamma = (x_0, \dots, x_r)$, we will write ``$x\Path{\gamma}y$'' as a shorthand for "$x_0=x$ and $x_r=y$".  We will write "$x\ResPath{X}{\gamma}y$" to further indicate that for all $0\le i \le r$, $x_i\in X$. For all $0\le i_1 \le i_2 \le r$, we define
\begin{equation}
	{\gamma}{\intervalleff{i_1}{i_2}} \dpe (x_i)_{i_1\le i \le i_2}.
\end{equation}

We call a \emph{continuous path} a continuous function $\gamma : \intervalleff0T \rightarrow \R^d$. The length of the interval on which a continuous path $\gamma$ is defined will be denoted by $T_\gamma$. Given a continuous path $\gamma$, we will write "$x\Path{\gamma}y$" as a shorthand for "$\gamma(0)=x$ and $\gamma(T_\gamma)=y$".  We will write "$x\ResPath{X}{\gamma}y$" to further indicate that for all $t\in\intervalleff{0}{T_\gamma}$, $\gamma(t)\in X$.

\paragraph{Metrics.}
For all seminorms $g$ on $\R^d$, we denote by $\d_g$ the pseudometric on $\intervalleff{0}{1}^d$ defined by
\[
    \d_g(x,y)\dpe g(x-y).
\]
Given a compact subset $X \subseteq \R^d$ and $D\in\Metrics(X)$, it will be useful to extend $D$ to $\R^d$ by defining
\begin{equation}
\label{eqn : Intro/notations/extension}
	\Extension D(x,y) \dpe \min\p{\min_{x',y'\in X}\p{b\norme{x-x'}+D(x',y')+b\norme{y-y'} } , b\norme{x-y} }.
\end{equation}
Given a metric $D$ on a set $X$ and two subsets $A,B\subseteq X$, we define
\begin{equation}
	D(A,B) \dpe \inf_{\substack{x\in A \\ y\in B} }D(x,y).
\end{equation}
If $\gamma : \intervalleff0T\rightarrow X$ is a continuous path, we define the $D$-length of $\gamma$ as
\begin{equation}
\label{eqn : Intro/notations/length_curve}
	D(\gamma) \dpe \sup_{0\le t_0< \dots < t_r\le T} \sum_{i=0}^{r-1} D\p{\gamma(t_i), \gamma(t_{i+1})}.
\end{equation}
By the triangle inequality, for all continuous paths $x\Path{\gamma}y$, $D(\gamma) \ge D(x,y)$. We will denote by $\norme\gamma$ the length of $\gamma$ with respect to the metric induced by $\norme\cdot$.

Except when otherwise specified, whenever we say that a map is Lipschitz, we mean Lipschitz for $\norme\cdot$.

\paragraph{Absolutely homogeneous functions.}
Let $\ContHom$ denote the set of continuous and absolutely homogeneous functions $f:\R^d\rightarrow \R^+$, i.e. satisfying $f(\lambda x) = \module\lambda f(x)$ for all $x\in \R^d$, $\lambda\in \R$. We endow this space with the norm defined as
\begin{equation}
	\label{eqn : Intro/notations/normes_ContHom}
	\normeHom{f} \dpe \sup_{u\in \S} \module{f(u)}.
\end{equation}
We endow $\ContHom$ with the cylinder $\sigma$-algebra, which is also its Borel $\sigma$-algebra.

\paragraph{Conventions about rate functions.}
Throughout the article, we introduce a number of rate functions similar to \eqref{eqn : Intro/LD/UpperLowerRF} and \eqref{eqn : MonRF}. We will respect the following naming rules for those functions:
\begin{itemize}
    \item The lower or upper bar indicates whether its definition involves a lower or an upper limit. In contexts where it is proved that both versions coincide, we drop the bars altogether.
    \item When the function's definition clearly implies that it is nondecreasing, the superscript $+$ will indicate it.
\end{itemize}

\section{The elementary rate function}
\label{sec : ERF}

In this section we construct the elementary rate function $\MonRFelem$. For all seminorms $g\in \AdmNorms$ and $\eps>0$, we define 
\begin{equation}
  \LD_{n,\mathrm{elem}}^+(g,\eps)\dpe \acc{\forall x,y\in \intint0n^d,\; \BoxPT(x,y)\ge g(x-y) - n \eps}.
\end{equation}
For all $g\in\AdmNorms$, we define
\begin{equation}
  \MonRFeleminf(g) \dpe \lim_{\eps \to 0} \liminf_{n\to\infty} -\frac{1}{n^d}\log \Pb{\LD_{n,\mathrm{elem}}^+(g,\eps)} \text{  and  } \MonRFelemsup(g) \dpe \lim_{\eps \to 0} \limsup_{n\to\infty} -\frac{1}{n^d}\log \Pb{\LD_{n,\mathrm{elem}}^+(g,\eps)}.
\end{equation}
The functions $\MonRFeleminf$ and $\MonRFelemsup$ are lower semi-continuous for $\normeHom\cdot$ (see Lemma \ref{lem : LD/RateFunctions}), and nondecreasing. The goal of the section is to prove Propositions \ref{prop : ERF/Existence}, \ref{prop : ERF/Continuity} and \ref{prop : ERF/LimitCases}.
\begin{Proposition}
  \label{prop : ERF/Existence}
  For all $g\in \AdmNorms$, $\MonRFeleminf(g) = \MonRFelemsup(g)$. We denote by $\MonRFelem(g)$ their common value.
\end{Proposition}
\begin{Proposition}
  \label{prop : ERF/Continuity}
  The map $\MonRFelem$ is continuous on $\AdmNorms^*$.
\end{Proposition}
\begin{Proposition}
  \label{prop : ERF/LimitCases}
    For all $g\in \AdmNorms$,
    \begin{enumerate}[(i)]
        \item $\MonRFelem(g)=\infty$ if and only if $g\in\AdmNorms \setminus \NoncritAdmNorms$ (i.e. $\exists x\neq 0, g(x)=b\norme x$)  and $\nu(\acc b)=0$.
        \item $\MonRFelem(g)=0$ if and only if $g\le \TimeConstant$, where $\mu$ is defined by~\eqref{eqn : Intro/Background/def_TimeConstant}.
    \end{enumerate}
\end{Proposition}


\subsection{Existence and continuity}
Both Propositions \ref{prop : ERF/Existence} and \ref{prop : ERF/Continuity} are consequences of Lemma \ref{lem : ERF/Existence/MainInequality} which gives a subadditivity inequality for $\Pb{\LD_{n,\mathrm{elem}}^+(g,\eps)}$ relying on the idea that if the random metric is lower bounded by $g$ on size $n$ boxes that cover an $m\gg n$ box, then the same bound holds on the larger scale. The main difficulty is the lack of control on $\tau(\gamma)$, for paths $\gamma$ that cross box boundaries too many times. Our workaround is to prescribe the edge weights to be close to $b$ near box boundaries.

\begin{Lemma}
  \label{lem : ERF/Existence/MainInequality}
    Consider a seminorm $g\in \AdmNorms$, and $0< \delta, \eta < 1/2$. Assume that for all $u\in \bbR^d$, $g(u) \le (b-2\eta)\norme u$. Then 
    \begin{equation}
    \label{eqn : ERF/Existence/MainInequality}
      \MonRFelemsup\p{g + \frac{\delta \eta}{2}\norme\cdot}%
            \le  \MonRFeleminf(g) %
            -d\cro{ 1- \p{ 1- 2\delta  }^d } \log \nu\p{\intervalleff{b-\eta}{b} }.
    \end{equation}
\end{Lemma}
\begin{proof}[Proof of Lemma~\ref{lem : ERF/Existence/MainInequality}]
  Let $g\in \AdmNorms$, and $0< \delta, \eta < 1/2$ be as in the statement, and $\eps>0$. Fix two integers $m\ge n \ge 2$, and $K\dpe \floor{ \frac{m}{n} }$. For all $k\in \intint{0}{K}^d$, we define
  \begin{equation}
    \label{eqn : ERF/Existence/MainInequality/Tile}
    \cTile(k) \dpe n k + \intint0n^d
  \end{equation}
  and the set
  \begin{equation}
      \Corridor \dpe \edges{ \intint0m^d } \setminus \p{ \bigcup_{k \in \intint0K^d} \edges{ nk + \intint{\ceil{n\delta} }{ \floor{n(1-\delta)} }^d  } } .
  \end{equation}
  For all $k\in\intint0K^d$, we define
  \begin{equation}
    \label{eqn : ERF/Existence/MainInequality/EvtTile}
    E_1(k) \dpe \acc{\forall x,y\in \cTile(k),\; \BoxPT[\cTile(k)](x,y) \ge g(x-y) - \frac12\delta\eta n}.
  \end{equation}
  We also define 
  \begin{equation}
    E_2 \dpe \bigcap_{e\in \Corridor } \acc{\tau_e \ge b- \eta}.
  \end{equation}
  Consider the event
  \begin{equation}
    \label{eqn : ERF/Existence/MainInequality/Fav}
    \cFav = \cFav(g,\delta, \eta, n, m) \dpe \p{ \bigcap_{k \in \intint0K^d} E_1(k) }\cap E_2.
  \end{equation}
  We claim that
  \begin{equation}
    \label{eqn : ERF/Existence/MainInequality/PbFav}
    \begin{split}
    \limsup_{m\to \infty }-\frac{1}{m^d} \log \Pb{\cFav}%
      &\le -\frac{1}{n^d} \log\Pb{\LD^+_{n, \mathrm{elem}}\p{g, \frac{\delta\eta}{2} } } \\ &\qquad-d \cro{ 1- \p{ 1- 2\delta - \frac3n }^d } \log \nu\p{\intervalleff{b-\eta}{b} }.
  \end{split}
  \end{equation}
  and that for large enough $m$,
  \begin{equation}
    \label{eqn : ERF/Existence/MainInequality/InclusionFav}
    \begin{split}
    \cFav \subseteq%
        \acc{\forall x,y\in \intint0m^d,\; \BoxPT[m](x,y) \ge g(x-y) + \frac12\norme{x-y}\delta\eta  - \eps m}.
  \end{split}
  \end{equation}

  Note that the events $E_1(k)$, for $k \in \intint0K^d$, and $E_2$ are all increasing with respect to the edge weights, therefore by the FKG inequality,
  \begin{equation}\label{eqn : ERF/Existence/MainInequality/PbFav/FKG}
    \Pb{\cFav}%
      \ge \p{ \prod_{k\in \intint0K^d}\Pb{E_1(k)} } \cdot \Pb{ E_2 }.
  \end{equation}
  By stationarity, for all $k\in \intint0K^d$,
  \begin{equation}\label{eqn : ERF/Existence/MainInequality/PbFav/E_1}
    \Pb{E_1(k)} = \Pb{\LD^+_{n, \mathrm{elem}}\p{g , \frac{\delta\eta}{2} } }.
  \end{equation}
  Moreover, 
  \[ \begin{split} \#\Corridor 
    &\le  \#\edges{ \intint0m^d } - K^d \#\edges{ \intint{ \ceil{n\delta} }{ \floor{(1-\delta)n} }^d  }\\
    &\le d(m+1)^d - dK^d \p{ (1-2\delta)n -3}^d \\
    &\le d(m+1)^d - d\p{\frac mn -1}^d \p{ (1-2\delta)n -3}^d
    \end{split}\]
  thus
  \begin{equation}\label{eqn : ERF/Existence/MainInequality/PbFav/E_2}
    \Pb{E_2} \ge \nu\p{\intervalleff{b-\eta}{b} }^{ d(m+1)^d - d\p{\frac mn -1}^d \cro{ (1-2\delta)n -3}^d }.
  \end{equation}
  Plugging \eqref{eqn : ERF/Existence/MainInequality/PbFav/E_1} and \eqref{eqn : ERF/Existence/MainInequality/PbFav/E_2} in \eqref{eqn : ERF/Existence/MainInequality/PbFav/FKG} yields
  \[ \begin{split}
    -\frac{1}{m^d} \log \Pb{\cFav}%
      &\le -\p{1+\frac nm}^d\cdot \frac{1}{n^d} \log\Pb{\LD^+_{n, \mathrm{elem}}\p{g, \frac{\delta\eta}{2} } } \\ &\qquad-\cro{d\p{1+\frac1m}^d - d\p{\frac 1n -\frac1m}^d \cro{ (1-2\delta)n -3}^d} \log \nu\p{\intervalleff{b-\eta}{b} }.
  \end{split}
  \]
  Taking the upper limit as $m\to\infty$ gives \eqref{eqn : ERF/Existence/MainInequality/PbFav}.
  
  Assume that $\cFav$ occurs and fix $x,y\in \intint0m^d$. Let $x\Path\gamma y$ be a discrete path included in $\intint0m^d$. Excursions of $\gamma$ in $\Corridor$ of length smaller than $\delta n$ are included in a tile $\cTile(k)$, thus $\gamma$ admits a decomposition of the form
  \[
    x = z_0 \Path{\gamma_0} z_0' \Path{\gamma_0'} z_1 \Path{\gamma_1} z_1' \Path{\gamma_1'} \dots \Path{\gamma_{R-1}'} z_R \Path{\gamma_R} z_R' = y,
  \]
  where \begin{enumerate}[(i)]
    \item for all $i\in\intint{0}{R}$, $\gamma_i \subseteq \cTile(k_i)$, for some $k_i \in \intint0K^d$,
    \item for all $i\in\intint{0}{R-1}$, $\gamma_i' \subseteq \Corridor$ and $\norme{\gamma_i'}\ge n\delta \p{ \norme[1]{k_i - k_{i+1} }\vee 1}$.
  \end{enumerate}
  Given $i\in \intint0{R-1}$, since the events $E_1(k_i)$ and $E_2$ occur we have
  \begin{align*}
    \tau(\gamma_i) + \tau(\gamma_i')%
      &\ge \BoxPT[\Tile(k_i)](z_i, z_i') + (b-\eta) \norme{\gamma_i'} \\
      &\ge g(z_i - z_i') - \frac12\delta\eta n + (b-\eta) \norme{\gamma_i'} \\
      &\ge g(z_i - z_i') + (b-2\eta) \norme{z_i'-z_{i+1} } - \frac12\delta\eta n + \eta \norme{\gamma_i'}
    \intertext{By the upper bound on $g$,}
    \tau(\gamma_i) + \tau(\gamma_i')  %
      &\ge g(z_i - z_i') + g(z_i' - z_{i+1}) -\frac12\delta\eta n + \eta \norme{\gamma_i'}.
    \intertext{Applying the triangle inequality and plugging in $\norme{ \gamma_i'}\ge \delta n  \p{ \norme[1]{k_i - k_{i+1} } \vee 1} $ yields }
    \tau(\gamma_i) + \tau(\gamma_i') %
        &\ge g(z_i -  z_{i+1}) +\frac12\delta\eta  n  \norme[1]{k_i - k_{i+1} }.
  \end{align*}
  Summing this inequality over $i$, we get
  \begin{align*}
      \tau(\gamma)%
        &\ge \sum_{i=0}^{R-1}\cro{g(z_i -  z_{i+1}) +\frac12\delta\eta n \norme[1]{k_i - k_{i+1} } }\\
        &\ge g(x-z_R ) +\frac12\delta\eta n \sum_{i=0}^{R-1} \norme[1]{k_i - k_{i+1} }.%
        \intertext{Applying the triangle inequality and the bound $g\le b\norme\cdot$, we get}
        &\ge g(x-y) +\frac12\delta\eta n \sum_{i=0}^{R-1} \norme[1]{k_i - k_{i+1} } -dbn.
  \intertext{But \[ \sum_{i=0}^{R-1} \norme[1]{k_i - k_{i+1} } \ge \norme[1]{k_0 - k_R} \ge \frac{\norme[1]{x-y} }{n} - d, \]
  thus }
  \tau(\gamma)%
        &\ge g(x-y) +\frac12\norme{x-y}\delta\eta -   nd\p{ \frac{\delta \eta}{2} + b } .
    \end{align*} 
    Consequently, $n$ being fixed, for large enough $m$, \eqref{eqn : ERF/Existence/MainInequality/InclusionFav} holds.

    By \eqref{eqn : ERF/Existence/MainInequality/PbFav} and \eqref{eqn : ERF/Existence/MainInequality/InclusionFav},
    \begin{equation}
      \label{eqn : ERF/Existence/MainInequality/AlmostThere}
       \begin{split}
         \limsup_{m\to\infty} -\frac{1}{m^d} \log \Pb{\LD^+_{m,\mathrm{elem}}\p{g + \frac{\delta \eta}{2}\norme\cdot , \eps }  }  &\le -\frac{1}{n^d} \log\Pb{\LD^+_{n, \mathrm{elem}}\p{g, \frac{\delta\eta}{2} } } \\ &\qquad-d \cro{ 1- \p{1-2\delta - \frac3n }^d } \log \nu\p{\intervalleff{b-\eta}{b} }.
       \end{split}
    \end{equation}
    Letting $\eps \to 0$ then taking the inferior limit as $n\to \infty$ gives
    \begin{align*}
        \MonRFelemsup\p{g + \frac{\delta \eta}{2}\norme\cdot}%
            &\le {\liminf_{n \to \infty } -\frac1{n^d} \log\Pb{\LD^+_{n, \mathrm{elem}}\p{g, \frac{\delta \eta}{2} } }  } %
            -d \cro{ 1- \p{ 1- 2\delta  }^d } \log \nu\p{\intervalleff{b-\eta}{b} } \\
            &\le   \MonRFeleminf(g) %
            -d\cro{ 1- \p{ 1- 2\delta  }^d } \log \nu\p{\intervalleff{b-\eta}{b} },
    \end{align*}
    which completes the proof.
\end{proof}


\begin{proof}[Proof of Proposition \ref{prop : ERF/Existence}]
    Fix $g\in \AdmNorms$. We first assume that $\normeHom g < b$. Then for some $\eta>0$, for all $u\in \R^d$, $g(u) \le (b-2\eta) \norme u$. Let $0< \delta < 1/2$. By \eqref{eqn : ERF/Existence/MainInequality},
    \begin{equation*}
      \MonRFelemsup\p{g }%
            \le  \MonRFelemsup\p{g + \frac{\delta \eta}{2}\norme\cdot }
            \le  \MonRFeleminf(g) %
            -d\cro{ 1- \p{ 1- 2\delta  }^d } \log \nu\p{\intervalleff{b- \eta}{b} }.
    \end{equation*}
    Letting $\delta \to0$ gives $\MonRFelemsup\p{g } \le \MonRFeleminf\p{g }$, thus $\MonRFelemsup\p{g } = \MonRFeleminf\p{g }$.

    To get this equality in full generality, let $(g_k)_{k\ge 1}$ be a nondecreasing sequence in $\AdmNorms$ that converges to $g$, such that for all $k\ge 1$, $\normeHom{g_k} < b$. Since $\MonRFelemsup$ and $\MonRFeleminf$ are both nondecreasing and lower semi-continuous, we have
    \[
        \lim_{k\to \infty} \MonRFelemsup(g_k) = \MonRFelemsup(g) \text{  and  } \lim_{k\to \infty} \MonRFeleminf(g_k) = \MonRFeleminf(g).
    \]
    By the previous case, for all $k\ge1$, $\MonRFelemsup(g_k) = \MonRFeleminf(g_k)$, thus $\MonRFelemsup(g) = \MonRFeleminf(g)$.
\end{proof}
\begin{proof}[Proof of Proposition \ref{prop : ERF/Continuity}]
    Only the upper semi-continuity is left to prove. Consider $g\in \AdmNorms^*$, and a sequence of seminorms $(g_k)_{k\ge 1}$ in $\AdmNorms^*$ that converges to $g$ for $\normeHom\cdot$. There exists $\eta>0$ such that for all $u\in \R^d$, $g(u) \le (b-2\eta)\norme u$. Let $0<\delta < 1/2$. For large enough $k$, 
    \[ g_k \le g + \frac{\delta\eta}{2} \norme\cdot, \]
    thus Lemma \ref{lem : ERF/Existence/MainInequality} gives
    \[
        \limsup_{k\to\infty} \MonRFelem\p{g_k}%
            \le \MonRFelem\p{g+\frac{\delta\eta}{2} \norme\cdot} \le  \MonRFelem(g) %
            -d\cro{ 1- \p{ 1- 2\delta  }^d } \log \nu\p{\intervalleff{b-\eta}{b} }.
    \]
    Letting $\delta\to0$ yields
    \[\limsup_{k\to\infty} \MonRFelem\p{g_k} \le \MonRFelem(g).\]
\end{proof}


\subsection{Characterization of the degenerate cases}

Proposition \ref{prop : ERF/LimitCases} is a direct consequence of Lemmas~\ref{lem : ERF/LimitCases/infty} and \ref{lem : ERF/LimitCases/zero}.
\begin{Lemma}
\label{lem : ERF/LimitCases/infty}
\leavevmode \vspace{-\baselineskip}
\begin{enumerate}[(i)]
    \item For all $0\le \zeta\le b$,
            \begin{equation}
            \label{eqn : ERF/LimitCases/infty/UB}
                \MonRFelem\p{\zeta\norme\cdot} \le -d \log \nu\p{\intervalleff{\zeta}{b}}.
            \end{equation}
    \item If $g\in \AdmNorms$ satisfies $\normeHom{g}\ge b-\eta$, then 
        \begin{equation} 
        \label{eqn : ERF/LimitCases/infty/LB}
        \MonRFelem(g) \ge -\log 2 - \frac12\log\nu\p{\intervalleff{b-4\eta}{b}}.
        \end{equation}
\end{enumerate} 
\end{Lemma}

\begin{Lemma}
\label{lem : ERF/LimitCases/zero}
\leavevmode \vspace{-\baselineskip}
\begin{enumerate}[(i)]
    \item $\MonRFelem(\mu)=0$.
    \item Let $g\in \AdmNorms$. If there exists $u\in\R^d$ such that $g(u) > \mu(u)$, then $\MonRFelem(g)>0$.
\end{enumerate}
\end{Lemma}

\begin{proof}[Proof of Lemma \ref{lem : ERF/LimitCases/infty}]
    \emph{Proof of (i).} 
    Let $\zeta\le b$. The inclusion
    \begin{equation}
        \bigcap_{e\in\edges{\intint0n^d}}\acc{ \tau_e \ge \zeta } \subseteq \acc{\forall x,y \in \intint0n^d,\; \BoxPT(x,y) \ge \zeta\norme{x-y} }\notag
    \end{equation}
    implies the statement.

    \emph{Proof of (ii).} %
    Assume that $\normeHom{g}\ge b-\eta$. Then there exists $u\in \S$ such that $g(u)\ge b-\eta$. By convexity and symmetry there exists $i\in \intint1d$ such that $g(\base i)\ge b-\eta$; without loss of generality we assume $i=1$. The argument consists in proving that on a large deviation event above $g$, a volumic number of edges must have a passage time close to $b$. For all $v\in \intint0n^{d-1}$, let $\pi_v$ denote the segment from $(0,v)$ to $(n,v)$.  Since $\LD^+_{n,\mathrm{elem} }(g, \eta) \subseteq \bigcap_v \acc{\tau(\pi_v)\ge n(b-2\eta)}$ and the $\p{\PathPT{\pi_v} }_{v}$ are i.i.d.,
    \begin{equation*}
        \Pb{ \LD^+_{n,\mathrm{elem} }(g, \eta) } \le \Pb{\tau(\pi_0) \ge n(b-2\eta)}^{n^{d-1}}.
    \end{equation*}
    Besides, if at least $n/2$ indices $i\in\intint0{n-1}$ are such that $\PathPT{i\mathrm e_1,(i+1)\mathrm e_1} < b-4\eta$, then
    \begin{equation*}
        \PathPT{\pi_0} < \frac n2(b-4\eta) + \frac{nb}{2} = n(b-2\eta).
    \end{equation*}
    Consequently,
    \begin{align}
        \Pb{\PathPT{\pi_0} \ge n(b-2\eta)} %
            &\le \Pb{\bigcup_{ \substack{A\subseteq \intint0{n-1}\\ \#A=\floor{n/2} } } \bigcap_{i\in A}\acc {\PathPT{i\mathrm e_1,(i+1)\mathrm e_1} \ge b-4\eta}} \eol
            &\le 2^n\nu\p{\intervalleff {b-4\eta}{b} }^{\floor{n/2} },
    \end{align}
    therefore
    \begin{equation}
    \MonRFelem(g)\ge  \liminf_{n\to\infty} -\frac{1}{n^d} \log\Pb{\LD^+_{n,\mathrm{elem} }(g, \eta) } \ge -\log 2 -\frac12 \log \nu\p{\intervalleff {b-4\eta}{b} }.
    \end{equation}\end{proof}


\begin{proof}[Proof of Lemma \ref{lem : ERF/LimitCases/zero}]
    \emph{Proof of (i).}  Let $x,y\in \intervalleff01^d$ and $\eps>0$. By the definition of $\mu$ (see~\eqref{eqn : Intro/Background/def_TimeConstant}),
    \begin{align}
     \label{eqn : ERF/LimitCases/zero/time_constant}
        \Pb{\BoxPT(\floor{nx}, \floor{ny} ) \ge n\mu(x-y) - n\eps} &\xrightarrow[n\to\infty]{} 1.
     \intertext{Besides, since the edge weights are bounded, for large enough $k$, for all $n \ge 1$, }
     \label{eqn : ERF/LimitCases/zero/fav}
        \bigcap_{x,y\in\frac1k\intint0k^d} \acc{\BoxPT(\floor{nx},\floor {ny})  \ge n\mu(x-y) - n\eps} &\subseteq \LD^+_{n, \mathrm{elem}}(\mu, 2\eps).
    \end{align}
    By~\eqref{eqn : ERF/LimitCases/zero/time_constant}, with fixed $k$, the probability of the left-hand side of~\eqref{eqn : ERF/LimitCases/zero/fav} converges to $1$ as $n\to\infty$, thus $\MonRFelem(\mu)=0$.

    \emph{Proof of (ii).} We adapt the arguments used by Kesten to prove a bound on the probability of the upper tail large deviation event for the point-to-point passage time, in the direction $\base 1$ \cite[Equation~(5.13)]{KestenStFlour}. Assume that there exists $u \in\R^d$ such that $g(u) > \mu(u)$. By homogeneity we can assume $u\in \S$. For all $n,k\ge 1$, define the tilted box
    \begin{equation*}
        \TBox(n,k) \dpe \acc{tu + x, x\in\clball{0,n}\cap u^\perp, t\in\intervalleff0{nk} },
    \end{equation*}
    where $u^\perp$ is the orthogonal complement of $u$. As in \cite[bottom of page 198]{KestenStFlour}, one shows that there exists $n\ge 1$ such that for all $k\ge 1$,
    \begin{equation*}
        \frac{\E{\BoxPT[\TBox(n,k)]\p{0,\floor{nku}}  }}{nk} < g(u).
    \end{equation*}
    Consequently, by Talagrand's inequality (see e.g. \cite[Theorem 3.13]{50yFPP}), there exists $\eps>0$ such that
    \begin{equation}
    \label{eqn : ERF/LimitCases/zero/Un_couloir} 
        \liminf_{k\to\infty}-\frac1{nk} \log \Pb{ \frac{\BoxPT[\TBox(n,k)]\p{0,\floor{nku}} }{nk} \ge g(u) - 2\eps }>0.
    \end{equation}
    Moreover, there exists $\Cl{cst : corridor_packing}>0$ such that for all $k \ge 1$, there exists $v_1,\dots, v_{\floor{\Cr{cst : corridor_packing} k^{d-1} }} \in \Z^d$, such that the boxes $\p{\TBox(n,k)+v_i}_{1\le i \le\floor{\Cr{cst : corridor_packing} k^{d-1} } }$ are pairwise disjoint and included in $\intervalleff{0}{2nk}^d$. This implies the inclusion
    \begin{align}
        \LD^+_{2nk, \mathrm{elem} }\p{g, \eps} &\subseteq \bigcap_{i=1}^{\floor{\Cr{cst : corridor_packing} k^{d-1} }} \acc{ \frac{\BoxPT[v_i+\TBox(n,k)]\p{v_i,v_i+\floor{nku} }}{nk} \ge g(u) - 3\eps }, \nonumber
        \intertext{and the terms of the intersection are independent. Consequently, by stationarity,}
        -\frac{1}{\p{nk}^d} \log \Pb{ \LD^+_{2nk, \mathrm{elem} }\p{g, \eps} }%
            &\ge -\frac{\floor{\Cr{cst : corridor_packing} k^{d-1} }}{\p{nk}^d} \log \Pb{ \frac{\BoxPT[\TBox(n,k)]\p{0,\floor{nku} }}{nk} \ge g(u) - 3\eps }.\nonumber
    \end{align}
    Letting $k\to\infty$ and applying~\eqref{eqn : ERF/LimitCases/zero/Un_couloir}, we get $\MonRFelem(g)>0$.
\end{proof}

\section{Gradient of a pseudometric}
  \label{sec : Gradient}
  In this section two real numbers $v\le w$ and a pseudometric $D\in \Metrics\p{ \intervalleff{v}{w}^d}$ are fixed. We define the gradient of $D$ at a point $z\in \intervalleff{v}{w}^d$. In Section \ref{subsec : Gradient/Approx}, we show that near almost every $z$, $D$ is locally described by $(\grad D)_z$. In Section \ref{subsec : Gradient/Paths}, we establish a link between $D$-lengths of Lipschitz paths and integrals involving the gradient. We heavily rely on well known results around the so-called \emph{metric derivative}, a tool initially introduced by Kirchheim~\cite{Kir94} to extend the area formula to functions taking values in a general metric space.
  
\subsection{Definition of the gradient and local description of a metric}
\label{subsec : Gradient/Approx}
\begin{Definition}\label{def : Gradient/gradient}
    The \emph{gradient} of $D$ at $z\in \intervalleff{v}{w}^d$ is defined as the function
    \begin{align}
    	(\grad D)_z : \R^d &\longrightarrow \intervallefo0\infty \notag \\
    		 u &\longmapsto \liminf_{h \to 0} \frac{ \overline D(z,z+hu)}{\module h},%
    	\label{eqn : Gradient/gradient}  
     \end{align}
    where $\overline D$ is defined by \eqref{eqn : Intro/notations/extension}. Note that if $z \in \intervalleoo vw^d$, one can write $D$ instead of $\overline D$ without changing the definition.
  \end{Definition}

  \begin{Proposition}[Elementary properties of the gradient]
  \label{prop : Gradient/ptes_elem}
  The gradient of $D$ satisfies the following.
    \begin{enumerate}[(i)]
      \item Let $z\in \intervalleff{v}{w}^d$. Then $(\grad D)_z$ is $b$-Lipschitz and absolutely homogeneous. In particular it belongs to the space $\ContHom$ defined in Subsection~\ref{subsec : Intro/notations}. 
      \item Let $z \in \intervalleff{v}{w}^d$. Then for all $u\in \R^d,$ 
      \begin{equation}
      \label{eqn : Gradient/encadrement}
        (\grad D)_z(u) \le b\norme{u}.
      \end{equation}
      \item The map \begin{align*} \intervalleff{v}{w}^d&\longrightarrow \ContHom \\ z&\longmapsto (\grad D)_z \end{align*} is measurable.
  \end{enumerate}
  \end{Proposition}
  \begin{proof}
  	Notice that for all $z \in \intervalleoo vw^d$, for all $u_1,u_2 \in \R^d$, for small enough $h \neq 0$,
    \begin{equation*}
      \module{ D(z,z+hu_1) -D(z,z+hu_2) } \le D(z+hu_1 , z+hu_2) \le b\module{h}\norme{u_2-u_1}.
    \end{equation*}
    Hence $(\grad D)_z$ is $b$-Lipschitz. The rest is clear.
  \end{proof}

  \begin{Proposition}
  \label{prop : Gradient/derivee}
  	For almost every $z\in \intervalleff{v}{w}^d$, $(\grad D)_z \in \AdmNorms$, and there exists a function $\eps : \intervallefo0\infty \rightarrow \intervallefo0\infty$ such that $\eps(0^+)=0$, and for all $x,y\in \intervalleff{v}{w}^d$,
  	\begin{equation}
  		\label{eqn : Gradient/derivee}
  	    \module{ D(x,y) - (\grad D)_z(y-x) } \le \eps\p{\norme{x-z} + \norme{y-z}}\cdot \p{\norme{x-z} + \norme{y-z}}.
  	\end{equation}
  \end{Proposition}
  \begin{proof}
  We first assume that $D$ is a metric; the full result will follow from a standard truncation argument. It is a special case of Theorem~2 in \cite{Kir94} for the identity map. However some minor adaptations need to be made, because their result is stated for maps taking values in a Banach space. The map
  \begin{align*}
  \iota : (\R^d, \Extension {D} ) &\longrightarrow \ell^{\infty}(\R^d)\\
  x &\longmapsto  \Extension D\p{x , \cdot} -  \Extension D\p{0, \cdot}
  \end{align*}
  is an isometry from $\R^d$ to the space $\ell^{\infty}( \R^d)$ of bounded real functions on $\R^d$ endowed with the sup norm denoted by $\norme[\ell^{\infty}(\R^d)]{\cdot}$, therefore
  \begin{align*}
  f : (\R^d, \norme{\cdot}) &\longrightarrow \ell^{\infty}(\R^d)\\
  x &\longmapsto \iota(x)
  \end{align*}
  is Lipschitz and takes values in a Banach space. Theorem~2 in \cite{Kir94} hence implies that almost every $z\in \intervalleff{v}{w}^d$ satisfies the following.
  \begin{enumerate}[(i)]
     	\item For all $u\in\R^d$, the limit 
  	\begin{equation}
  	\label{eqn : Gradient/MD}
  	    \mathrm{MD}(f,z)(u) \dpe \lim_{\substack{h\to 0 \\ h>0}} \frac1h \norme[\ell^\infty(\R^d)]{f(z+hu) - f(z)}
  	\end{equation}
  	exists (the notation $\mathrm{MD}$, for \emph{metric derivative} is from \cite{Kir94}).
      \item The function $\mathrm{MD}(f,z)(\cdot)$ is a seminorm.
      \item There exists a function $\eps : \intervallefo0\infty \rightarrow \intervallefo0\infty$ such that $\eps(0^+)=0$ and for all $x,y\in \intervalleoo vw^d$,
  	\begin{equation}
  	\label{eqn : Gradient/version_MD}
  		\module{ \norme[\ell^\infty(\R^d)]{f(x) - f(y)} - \mathrm{MD}(f,z)(y-x)  } \le \eps\p{\norme{x-z} + \norme{y-z}}\cdot \p{\norme{x-z} + \norme{y-z}}.
  	\end{equation}
  \end{enumerate}
  Let us fix such a $z$. As $\mathrm{MD}(f,z)(\cdot)$ is a seminorm, Equation~\eqref{eqn : Gradient/MD} can be rewritten as
  \begin{equation}
  	\mathrm{MD}(f,z)(u) = \lim_{h\to 0 } \frac1{ \module{h}} \norme[\ell^\infty( \R^d)]{f(z+hu) - f(z)}.
  \end{equation}
  Since $\iota$ is an isometry,
  \begin{align*}
      \mathrm{MD}(f,z)(u) = \lim_{h\to 0} \frac1{\module{h}} \Extension D\p{z+hu, z} &= \lim_{h\to 0} \frac1{\module{h}} D\p{z+hu, z}= (\grad D)_z(u),
      \intertext{and \eqref{eqn : Gradient/version_MD} can be rewritten as}
      \module{ D\p{x,y} - \mathrm{MD}(f,z)(y-x)  } &\le \eps\p{\norme{x-z} + \norme{y-z}}\cdot \p{\norme{x-z} + \norme{y-z}}.
  \end{align*} 
  Hence~\eqref{eqn : Gradient/derivee} holds and $(\grad D)_z$ is a seminorm. Finally, $(\grad D)_z$ satisfies~\eqref{eqn : Gradient/encadrement}, therefore $(\grad D)_z \in \AdmNorms$.

  To treat the case where $D$ is any pseudometric in $\Metrics$, given a rational $0<\delta<b$ we consider the metric $D_\delta$ defined for all $x,y\in \intervalleff{v}{w}^d$ by \[D_\delta(x,y) \dpe D(x,y)\vee \p{\delta \norme{x-y} }. \]
  Let $z \in \intervalleff{v}{w}^d$ be such that the conclusion of the proposition holds at $z$, for the metric $D_\delta$, for all rational $0< \delta < b$. We claim that it still holds for the metric $D$. Indeed, for all $x,y \in \intervalleff{v}{w}^d$,
    \[
        \module{ D_\delta(x,y) - D(x,y) } \le \delta \norme{x-y}
    \]
    and for all $z\in \intervalleoo vw^d$, $u\in \R^d$,
    \[
        \module{ (\grad D_\delta)_z(u) - (\grad D)_z(u) } \le \delta \norme{u}.
    \]
    The latter inequality implies that $(\grad D)_z \in \AdmNorms$. Moreover, for all $\delta$, there exists a function $\eps_\delta : \intervallefo0\infty \rightarrow \intervallefo0\infty$ such that $\eps_\delta(0^+)=0$, for all $x,y\in \intervalleff vw^d$,
  	\begin{equation}
  		\label{eqn : Gradient/derivee/Approx}
  	    \module{ D_\delta(x,y) - (\grad D_\delta)_z(y-x) } \le \eps_\delta\p{\norme{x-z} + \norme{y-z}}\cdot \p{\norme{x-z} + \norme{y-z}}.
  	\end{equation}
    Consequently, by the triangle inequality, for all rational $0<\delta<b$ and $x,y \in \intervalleff{v}{w}^d$,
    \begin{align*}
        &\module{D(x,y) - (\grad D)_z(x-y)} \\
            &\quad \le \module{D(x,y) - D_\delta(x,y)}+ \module{D_\delta(x,y) - (\grad D_\delta)_z(x-y)} + \module{(\grad D_\delta)_z(x-y) - (\grad D)_z(x-y)}  \\
            &\quad \le 2\delta  \norme{x-y} + \eps_\delta\p{\norme{x-z} + \norme{y-z}}\cdot \p{\norme{x-z} + \norme{y-z}} \\
            &\quad \le \cro{ \eps_\delta\p{\norme{x-z} + \norme{y-z}} +2\delta} \cdot \p{\norme{x-z} + \norme{y-z}}.
    \end{align*}
    Hence $D$ satisfies the inequality \eqref{eqn : Gradient/derivee} at the point $z$, with the function 
    \[ \eps : t \mapsto \inf_{ \substack{0<\delta<b\\ \delta \in \Q} }\p{\eps_\delta(t) + 2 \delta}, \]
    which has limit $0$ as $t\to0$.
  \end{proof}


\subsection{Length of Lipschitz paths}\label{subsec : Gradient/Paths}
  
Proposition~\ref{prop : Gradient/courbe} states that computing the $D$-length of a Lipschitz path amounts to integrating the gradient of $D$ along this path.
  \begin{Proposition}
  \label{prop : Gradient/courbe}
  	For any Lipschitz path $\gamma : \intervalleff0{T} \rightarrow \intervalleff{v}{w}^d$, 
  	\begin{equation}
  	\label{eqn : Gradient/length_gamma}
  	    D(\gamma) = \int_0^{T} (\grad D)_{\gamma(t)}(\gamma'(t)) \d t, \end{equation}
  	where $D(\gamma)$ is defined by~\eqref{eqn : Intro/notations/length_curve}. 
  \end{Proposition}

Proposition~\ref{prop : Gradient/inverse} gives in a sense a weak converse: it states that if $D'$ is a metric defined as the infimum of an integral analogous to the one in \eqref{eqn : Gradient/length_gamma}, with a family $(g_z)_{z\in\intervalleff{v}{w}^d}$ instead of $\p{\grad D}_{z\in\intervalleff vw^d}$, then $g_z$ is an upper bound for $(\grad D')_z$. Recall that we denote by $\cG\p{\intervalleff vw^d}$ the set of measurable functions $z\mapsto g_z$ defined on $\intervalleff vw^d$, with values in $\AdmNorms$.



\begin{Proposition}[Prescribing the gradient of a metric]
  \label{prop : Gradient/inverse}
   Fix $g \in \cG\p{\intervalleff vw^d}$. Consider
  \begin{align}
       \mathbf D[g] : \p{\intervalleff{v}{w}^d}^2 &\longrightarrow \intervallefo0\infty \eol 
       (x,y) &\longmapsto \inf%
       	 \set{ \int_0^T g_{\gamma(t) }\p{ \gamma'(t)}\d t }{\begin{array}{c} \gamma : \intervalleff0T \rightarrow \intervalleff{v}{w}^d \text{ Lipschitz s.t.\ }\\ \gamma(0)=x,\; \gamma(T)=y\end{array} }. \label{eqn : Gradient/inverse_def}
     \end{align}
  Then 
  \begin{enumerate}[(i)]
  \item $\mathbf D[g]\in\Metrics\p{ \intervalleff{v}{w}^d}$.
  \item For almost every $z\in \intervalleff{v}{w}^d$,
  \begin{equation}
  \label{eqn : Gradient/inegalite_gz_grad}
     (\grad \mathbf D[g])_z \le g_z.
   \end{equation} 
   \end{enumerate}
  \end{Proposition}

Lemma~\ref{lem : Gradient/Rademacher_kindof} is a consequence of the fact that \emph{absolutely continuous} functions satisfy the fundamental theorem of calculus (see e.g. \cite[Theorem~7.18]{Rudin}). The first part is a special case of Rademacher's theorem (see e.g. \cite[Theorem 3.1.6]{Fed69}).
\begin{Lemma}
  \label{lem : Gradient/Rademacher_kindof}
  Let $\gamma : \intervalleff0T \rightarrow \R^d$ be a Lipschitz path. Then $\gamma$ is differentiable almost everywhere and for all $0\le t_1 \le t_2 \le T$,
  \begin{equation}
  \label{eqn : Gradient/Rademacher_kindof}
    \gamma({t_2})-\gamma({t_1}) = \int_{t_1}^{t_2} \gamma'(t)\d t.
  \end{equation}
  \end{Lemma}
  
  \begin{proof}[Proof of Proposition \ref{prop : Gradient/courbe}]
  	Let $\gamma : \intervalleff0T \longrightarrow \intervalleff{v}{w}^d$ be a Lipschitz path from $x$ to $y$. Theorem~4.1.6 in~\cite{Amb04} implies that the limit
  	\begin{equation*}
  		\dot{\gamma}(t) \dpe \lim_{h\to 0}\frac{D(\gamma(t),\gamma(t+h))}{\module h}
  	\end{equation*}
  	exists for almost every $0\le t\le T$, and 
  	\begin{equation}
  	\label{eqn : Gradient/proof_curve1}
  	    D(\gamma) = \int_0^T \dot{\gamma}(t)\d t.
  	\end{equation}
  	Let $0\le t \le T$ be such that both $\dot{\gamma}(t)$ and $\gamma'(t)$ exist. We have
  	\begin{align}
  	  	\module{D(\gamma(t), \gamma(t+h)) - D(\gamma(t), \gamma(t) + \gamma'(t)h)}%
  	  		&\le D\p{\gamma(t+h) , \gamma(t) + \gamma'(t)h} \eol
  	    	&\le b\norme{ {\gamma(t+h) - \gamma(t) - \gamma'(t)h} }= \petito(h).\notag
  	\intertext{Dividing by $\module h$ and letting $h\to0 $ yields}
  	\label{eqn : Gradient/proof_curve2}
  		(\grad D)_{\gamma(t)}(\gamma'(t)) &= \dot{\gamma}(t).
  	\end{align}
  	The proposition follows from Equations~\eqref{eqn : Gradient/proof_curve1} and \eqref{eqn : Gradient/proof_curve2}.
  \end{proof}

    \begin{proof}[Proof of Proposition \ref{prop : Gradient/inverse}]
    Fix $D \dpe \mathbf D[g]$. The fact that the integral in~\eqref{eqn : Gradient/inverse_def} is well-defined is a consequence of the measurability of the maps $(z,u)\mapsto g_z(u)$, $\gamma$ and $\gamma'$.

    \emph{Proof of (i).}
    The map $D$ is clearly a pseudometric on $\intervalleff{v}{w}^d$. Let $x,y\in \intervalleff{v}{w}^d$. The affine path defined on $\intervalleff01$ by $\gamma(t) \dpe (1-t)x + ty$ is Lipschitz and takes values in $\intervalleff{v}{w}^d$, thus
    \begin{equation*}
      D(x,y)\le \int_0^{1} g_{\gamma(t)}(\gamma'(t)) \d t \le \int_0^{1} b\norme{\gamma'(t)}\d t = b\norme{x-y}.
    \end{equation*}
  
  \emph{Proof of (ii).} As the $g_z$ and the $(\grad D)_z$ are continuous and absolutely homogeneous, it is sufficient to prove that for all $u$ in a countable dense subset of $\S$,
    \begin{equation}
    \label{eqn : Gradient/grad_vs_gz}
      \text{for almost every $z\in \intervalleff{v}{w}^d$,}\quad (\grad D)_z(u) \le g_z(u).
    \end{equation}
    For the sake of simplicity, we only treat the case $u=\base 1$. The proof may be adapted to any $u\in \R^d$. The boundary of $\intervalleff{v}{w}^d$ has zero Lebesgue measure, thus it is sufficient to prove~\eqref{eqn : Gradient/grad_vs_gz} for almost every $z\in \intervalleoo  vw^d$. Let $z\in \intervalleff{v}{w}^d$ and $\eps>0$ such that $z+\intervalleoo 0\eps^d \subseteq \intervalleff vw^d$. Fubini's theorem yields
    \begin{equation*}
      \int_{\intervalleff0\eps^d} g_{z+s}(\base 1)\d s %
        = \int_{\intervalleff0\eps^{d-1}}\p{\int_0^\eps g_{z+t\base 1+ s_{2:d}}(\base 1)\d t }\d(s_2,\dots,s_d),
    \end{equation*}
    with the notation $s_{2:d}\dpe \sum_{i=2}^ds_i \base i$. Moreover, it follows from~\eqref{eqn : Gradient/inverse_def} that for all $(s_2,\dots,s_d) \in \intervalleff{0}{\eps}^{d-1}$,
    \begin{equation*}
      \int_0^\eps g_{z+t\base 1+ s_{2:d}}(\base 1)\d t \ge D\p{z+s_{2:d},z+\eps \base 1+ s_{2:d}},
    \end{equation*}
    hence
    \begin{equation}
      \frac{1}{\eps^d}\int_{\intervalleff0\eps^d} g_{z+s}(\base 1)\d s \ge \frac{1}{\eps^{d-1}} \int_{\intervalleff0\eps^{d-1}} \frac1\eps D\p{z+s_{2:d},z+\eps \base 1+ s_{2:d}} \d(s_2,\dots,s_d).
      \label{eqn : Gradient/presciption_vs_resultat}
    \end{equation}
    Lebesgue's differentiation theorem (see e.g. \cite[Theorem~7.10]{Rudin}) implies that as $\eps\to 0$, the left-hand side of \eqref{eqn : Gradient/presciption_vs_resultat} converges to $g_z(\base 1)$ for almost every $z\in \intervalleoo vw^d$. Besides, Proposition~\ref{prop : Gradient/derivee} implies that the right-hand side converges to $(\grad D)_z(\base 1)$ for almost every $z\in \intervalleoo vw^d$. This proves~\eqref{eqn : Gradient/grad_vs_gz}.
  \end{proof}


\section{One-sided large deviation principle}
\label{sec : MonLDP}

This section aims at proving Theorem \ref{thm : MonLDP}. Most of the work is contained in Propositions \ref{prop : MonLDP/UB} and \ref{prop : MonLDP/LB}, whose proofs are given in Sections \ref{subsec : MonLDP/UB} and \ref{subsec : MonLDP/LB}.
\begin{Proposition}
    \label{prop : MonLDP/UB}
    For all $D\in \Metrics$,
    \begin{equation}
        \label{eqn : MonLDP/UB}
        \MonRFsup(D) \le \int_{\intervalleff01^d} \MonRFelem\p{ (\grad D)_z }\d z.
    \end{equation}
\end{Proposition}
\begin{Proposition}
    \label{prop : MonLDP/LB}
    For all $D\in \Metrics$,
    \begin{equation}
        \label{eqn : MonLDP/LB}
        \MonRFinf(D) \ge \int_{\intervalleff01^d} \MonRFelem\p{ (\grad D)_z }\d z.
    \end{equation}
\end{Proposition}

For convenience, in the proofs of Propositions \ref{prop : MonLDP/UB} and \ref{prop : MonLDP/LB} we will use expressions of $\MonRFinf(D)$ and $\MonRFsup(D)$ that directly involve $\BoxPT$ rather than $\SPT$: for all $D\in \Metrics$ and $\eps>0$, define
\begin{equation}
    \LD_n^+(D,\eps) \dpe \acc{\forall x,y \in \intint0n^d, \quad \BoxPT(x,y) \ge nD\p{\frac xn, \frac yn} - n\eps }
\end{equation} By Proposition \ref{prop : GeoMetrics/Geodesics/extension_tranquille} and the Lipschitz property of $\CPT$, for all $n\ge 1$ and $x,y\in \intervalleff01^d$,
\[
    \module{ \BoxPT\p{ \floor{nx}, \floor{ny} } - \CPT\p{nx , ny} }\le 2bd.
\]
Consequently, for all $D\in \Metrics$, for all $\eps>0$, for large enough $n$,
\[
    \LD_n^+(D,\eps) \subseteq \acc{ \forall x,y \in \intervalleff01^d, \quad \SPT(x,y) \ge D\p{x, y} - 2\eps } \subseteq \LD_n^+(D,3\eps).
\]
Hence for all $D\in \Metrics$,
\begin{equation}
    \MonRFinf(D) = \lim_{\eps\to 0} \liminf_{n\to\infty} -\frac{1}{n^d} \log\Pb{ \LD_n^+(D,\eps) } \text{  and  } \MonRFsup(D) = \lim_{\eps\to 0} \limsup_{n\to\infty} -\frac{1}{n^d} \log\Pb{ \LD_n^+(D,\eps) }.
\end{equation}

\begin{proof}[Proof of Theorem \ref{thm : MonLDP}]
    Fix $D\in \Metrics$. By Propositions \ref{prop : MonLDP/UB} and \ref{prop : MonLDP/LB}, \[ \MonRFinf(D) = \MonRFsup(D) \eqqcolon \MonRF(D), \] and \eqref{eqn : MonLDP/Expression} holds.

    If for all $x,y \in \intervalleff01^d$, $D(x,y) \le \mu(x-y)$, then by Proposition \ref{prop : ERF/LimitCases}, we have
    \[
        \MonRF(D) \le \MonRF( \d_\mu) = \MonRFelem(\mu) = 0. 
    \]
    Conversely, assume that $\MonRF(D)=0$. Then by Proposition \ref{prop : ERF/LimitCases}, for almost all $z \in \intervalleff01^d$, $(\grad D)_z \le \mu$. Then by Proposition \ref{prop : Gradient/courbe}, for all $x,y\in \intervalleff01^d$ there exist Lipschitz paths of $D$-length arbitrarily close to $\mu(x-y)$. Consequently, $D\le \mu$.
\end{proof}

\subsection{Upper bounding the one-sided rate function}
\label{subsec : MonLDP/UB}
Fix a pseudometric $D\in \Metrics$ for this section. Our construction to prove the inequality \eqref{eqn : MonLDP/UB} is similar to the subadditivity used for Proposition \ref{prop : ERF/Existence}, but instead of considering an intersection of translates of the event $\LD_{n, \mathrm{elem} }^+(g,\eps)$, for some fixed seminorm $g$, we now select $g$ to be the gradient of $D$ (see Definition \ref{def : Gradient/gradient}). To handle the map $z \mapsto (\grad D)_z$ being discontinuous and the convergence in \eqref{eqn : Gradient/derivee} being non uniform with respect to $z$, we simply prescribe the edge weights to be close to $b$ in zones where $D$ is ill-behaved. Given $\delta>0$, we consider a compact set $A(\delta) \subseteq \intervalleff01^d$ with Lebesgue measure $1-\delta$ such that the restriction of $z\mapsto (\grad D)_z$ to $A(\delta)$ is continuous; existence of such a set is guaranteed by Lusin's theorem (see e.g. Theorem 2.24 in \cite{Rudin}). For all $\delta>0$ and all integers $K\ge 1$, we define
\begin{equation}
    A(\delta, K) \dpe \set{z\in A(\delta) }{ \forall x,y \in \p{ z + \intervalleff{-\frac3K}{\frac3K}^d}
    \cap \intervalleff01^d, \; D(x,y) \le (\grad D)_z(x-y) + \frac{\delta}{K} }
\end{equation}
By dominated convergence, for all $\delta>0$, 
\begin{equation}
    \lim_{K\to\infty} \Leb\p{A(\delta, K) } = \Leb\p{A(\delta) }.
\end{equation}
We say that $k\in \intint0{K-1}^d$ is $(\delta,K)$-\emph{stable}, or simply \emph{stable} if the parameters are clear from the context, if 
\[
     A(\delta, K) \cap \p{\frac kK + \intervalleff{0}{\frac1K}^d} \neq \emptyset,
\]
and in that case we will denote by $z_k$ an element of this set.

\begin{Lemma}
    \label{lem : MonLDP/UB/KeyInequality}
    Assume that for some $\eta>0$, \begin{equation}\label{eqn : MonLDP/UB/KeyInequality/Assumption} \forall x,y \in \intervalleff01^d, \quad D(x,y) \le (b-4\eta)\norme{x-y}.  \end{equation} Then for all $0<\delta<1/2$ and $K\ge1$,
    \begin{equation}\begin{split}\label{eqn : MonLDP/UB/KeyInequality}
         \limsup_{m\to\infty}-\frac1{m^d} \log \Pb{\LD_m^+\p{D, \frac{bd}{K} } }%
            &\le  \frac{1}{K^d} \sum_{\substack{ k\in \intint0{K-1}^d \\ \p{\eta \delta, K }\text{-stable} } } \MonRFelem\p{(\grad D)_{z_k} } \\
      &\qquad-d \cro{ 1- (1-2\delta)^d + {\Leb\p{A\p{\eta \delta, K }^\mathrm{c}  } }} \log \nu\intervalleff{b-\eta}{b}.
    \end{split}\end{equation}
\end{Lemma}
\begin{proof}[Proof of Lemma \ref{lem : MonLDP/UB/KeyInequality}]
    \stepcounter{tile}
    Fix $0<\delta<1/2$. Fix $m\ge 1$ and define $n \dpe \floor{\frac mK}$. For all $k\in \intint{0}{K-1}^d$, we define
  \begin{equation}
    \label{eqn : MonLDP/UB/KeyInequality/Tile}
    \cTile(k) \dpe n k + \intint0n^d
  \end{equation}
  and the set
  \begin{equation}
      \Corridor \dpe \edges{ \intint0m^d } \setminus \p{ \bigcup_{k \in \intint0{K-1}^d} \edges{ nk + \intint{\ceil{n\delta} }{ \floor{n(1-\delta)} }^d  } } .
  \end{equation}
  For all $\p{\eta \delta, K }$-stable $k\in \intint{0}{K-1}^d$, we define
  \begin{equation}
    \label{eqn : MonLDP/UB/KeyInequality/EvtTile}
    E_1(k) \dpe \acc{\forall x,y\in \cTile(k),\; \BoxPT[\cTile(k)](x,y) \ge (\grad D)_{z_k}(x-y) - \delta\eta n}.
  \end{equation}
    For all $\p{\eta \delta, K }$-unstable $k\in \intint{0}{K-1}^d$, we define
  \begin{equation}
    \label{eqn : MonLDP/UB/KeyInequality/EvtTileBis}
    E_1(k) \dpe \acc{\forall e\in \edges{\cTile(k)},\; \tau_e \ge b-\eta}.
  \end{equation}
  We also define 
  \begin{equation}
    E_2 \dpe \bigcap_{e\in \Corridor } \acc{\tau_e \ge b- \eta}.
  \end{equation}
  Consider the event
  \begin{equation}
    \label{eqn : MonLDP/UB/KeyInequality/Fav}
    \cFav = \cFav(D,\delta, \eta, m,K) \dpe \p{ \bigcap_{k \in \intint0{K-1}^d} E_1(k) }\cap E_2.
  \end{equation}
  We claim that
  \begin{equation}
      \label{eqn : MonLDP/UB/KeyInequality/PbFav}\begin{split}
      \limsup_{m\to \infty}-\frac{1}{m^d} \log \Pb{\cFav} &\le \frac{1}{K^d} \sum_{\substack{ k\in \intint0{K-1}^d \\ \p{\eta \delta, K }\text{-stable} } } \MonRFelem\p{(\grad D)_{z_k} } \\
      &\qquad-d \cro{ 1-(1-2\delta)^d + {\Leb\p{A\p{\eta \delta, K }^\mathrm{c}  } }} \log \nu\intervalleff{b-\eta}{b} 
  \end{split}\end{equation}
  and
  \begin{equation}
      \label{eqn : MonLDP/UB/KeyInequality/InclusionFav}
      \cFav \subseteq \LD_m^+\p{D, \frac{bd}{K}},
  \end{equation}
    which imply \eqref{eqn : MonLDP/UB/KeyInequality}.
    
  Indeed, note that the events $E_1(k)$, for $k \in \intint0{K-1}^d$, and $E_2$ are all increasing with respect to the edge weights, therefore by the FKG inequality,
  \begin{equation}\label{eqn : MonLDP/UB/KeyInequality/PbFav/FKG}
    \Pb{\cFav}%
      \ge \p{ \prod_{k\in \intint0{K-1}^d}\Pb{E_1(k)} } \cdot \Pb{ E_2 }.
  \end{equation}
  By stationarity, for all stable $k\in \intint0{K-1}^d$,
  \begin{equation}\label{eqn : MonLDP/UB/KeyInequality/PbFav/E_1}
    \Pb{E_1(k)} = \Pb{\LD^+_{n, \mathrm{elem}}\p{(\grad D)_{z_k} , \delta\eta } },
  \end{equation}
  and for all unstable $k\in \intint0{K-1}^d$,
  \begin{equation}\label{eqn : MonLDP/UB/KeyInequality/PbFav/E_1Bis}
    \Pb{E_1(k)} \ge \nu\p{\intervalleff{b-\eta}{b}}^{d(n+1)^d} \ge  \nu\p{\intervalleff{b-\eta}{b}}^{d(\frac mK+1)^d}.
  \end{equation}
  Moreover, 
  \[ \begin{split} \#\Corridor 
    &\le  \#\edges{ \intint0m^d } - K^d \#\edges{ \intint{ \ceil{n\delta} }{ \floor{(1-\delta)n} }^d  }\\
    &\le d(m+1)^d - dK^d \p{ (1-2\delta)n -2}^d \\
    &\le dm^d \cro{ \p{1+\frac1m}^d - \p{1-2\delta - \frac{3K}{m}}^d },
    \end{split}\]
  thus
  \begin{equation}\label{eqn : MonLDP/UB/KeyInequality/PbFav/E_2}
    \Pb{E_2} \ge \nu\p{\intervalleff{b-\eta}{b} }^{dm^d \cro{ \p{1+\frac1m}^d - \p{1-2\delta - \frac{3K}{m}}^d } }.
  \end{equation}
  Plugging \eqref{eqn : MonLDP/UB/KeyInequality/PbFav/E_1}, \eqref{eqn : MonLDP/UB/KeyInequality/PbFav/E_1Bis} and \eqref{eqn : MonLDP/UB/KeyInequality/PbFav/E_2} in \eqref{eqn : MonLDP/UB/KeyInequality/PbFav/FKG} yields
  \[ \begin{split}
    -\frac{1}{m^d} \log \Pb{\cFav}%
      &\le \p{\frac nm}^d\sum_{k\text{ stable}} -\frac{1}{n^d} \log\Pb{\LD^+_{n, \mathrm{elem}}\p{(\grad D)_{z_k}, \delta\eta} }  - d \Biggl[ \p{1+\frac1m}^d - \p{1-2\delta - \frac{3K}{m}}^d  \\
      &\quad + \p{\frac 1K+\frac1m}^d \cdot \#\acc{\text{unstable }k\in \intint0{K-1}^d } \Biggr] \log \nu\p{\intervalleff{b-\eta}{b} }.
  \end{split}
  \]
  Taking the upper limit as $m\to\infty$ and plugging in the straightforward inequality
  \[
    \frac1{K^d}\#\acc{\text{unstable }k\in \intint0{K-1}^d } \le \Leb\p{ A\p{\eta \delta, K}^\mathrm{c} }
  \]gives \eqref{eqn : MonLDP/UB/KeyInequality/PbFav}.

  Assume that $\cFav$ occurs and fix $x,y\in \intint0m^d$. Let $x\Path\gamma y$ be a discrete path included in $\intint0m^d$. Excursions of $\gamma$ in $\Corridor$ of length smaller than $\delta n$ are included in a tile $\cTile(k)$, thus $\gamma$ admits a decomposition of the form
  \[
    x = x_0 \Path{\gamma_0} x_0' \Path{\gamma_0'} x_1 \Path{\gamma_1} x_1' \Path{\gamma_1'} \dots \Path{\gamma_{R-1}'} x_R \Path{\gamma_R} x_R' = y,
  \]
  where \begin{enumerate}[(i)]
    \item for all $i\in\intint{0}{R}$, $\gamma_i \subseteq \cTile(k_i)$, for some $k_i \in \intint0{K-1}^d$,
    \item for all $i\in\intint{0}{R-1}$, $\gamma_i' \subseteq \Corridor$ and $\norme{\gamma_i'}\ge n\delta $.
  \end{enumerate}
  Given $i\in \intint0{R-1}$, with stable $k_i$, since the events $E_1(k_i)$ and $E_2$ occur we have
  \begin{align}
    \tau(\gamma_i) + \tau(\gamma_i')%
      &\ge \BoxPT[\cTile(k_i)](x_i, x_i') + (b-\eta) \norme{\gamma_i'}\nonumber \\
      &\ge (\grad D)_{z_{k_i}}(x_i - x_i') - \delta\eta n + (b-\eta) \norme{\gamma_i'}\nonumber \\
      &\ge (\grad D)_{z_{k_i}}(x_i - x_i') + (b-4\eta) \norme{x_i'-x_{i+1} } - \delta\eta n + 3\eta \norme{\gamma_i'}\nonumber
    \intertext{By the upper bound on $D$ and the inequality $\norme{ \gamma_i'}\ge \delta n $,}
    \tau(\gamma_i) + \tau(\gamma_i')  %
      &\ge m(\grad D)_{z_{k_i}}\p{ \frac{x_i}m - \frac{x_i'}m } + mD\p{ \frac{x_i'}m , \frac{x_{i+1}}m } +2\delta\eta n .\nonumber
    \intertext{If $m\ge K^2$, then one checks that $\frac{x_i}{m}, \frac{x_i'}m \in  z_{k_i} + \intervalleff{-\frac3K}{\frac3K}^d$, thus by the definition of $A\p{\delta \eta, K }$, }
    \tau(\gamma_i) + \tau(\gamma_i')  %
      &\ge m D\p{ \frac{x_i}m , \frac{x_i'}m } + mD\p{ \frac{x_i'}m , \frac{x_{i+1}}m } -\frac{m\delta \eta}{K} +2\delta\eta n \nonumber\\
      &\ge m D\p{ \frac{x_i}m , \frac{x_i'}m } + mD\p{ \frac{x_i'}m , \frac{x_{i+1}}m }.
      \label{eqn : MonLDP/UB/KeyInequality/InclusionFav/StableCase}
  \end{align}
  If $k_i$ is unstable then the same inequality is direct. Summing \eqref{eqn : MonLDP/UB/KeyInequality/InclusionFav/StableCase} over $i$ and applying the triangle inequality, we get
  \begin{equation*}
      \tau(\gamma)%
        \ge mD\p{ \frac xm,\frac{x_R}m } 
        \ge mD\p{ \frac xm, \frac ym } -nbd 
        \ge mD\p{ \frac xm, \frac ym } - \frac{bdm}K.\end{equation*}
    Consequently, for large enough $m $, \eqref{eqn : MonLDP/UB/KeyInequality/InclusionFav} holds.
\end{proof}
\begin{proof}[Proof of Proposition \ref{prop : MonLDP/UB}]
    We first treat the case where \eqref{eqn : MonLDP/UB/KeyInequality/Assumption} holds. This implies that for all $z\in \intervalleff01^d$, $(\grad D)_z \le (b-4\eta) \norme\cdot$. Let $0<\delta <1/2$. By the definition of $A(\delta)$ and Proposition \ref{prop : ERF/Continuity}, the restriction of the map $z\mapsto \MonRFelem\p{(\grad D)_z}$ to $A(\delta)$ is continuous. Fix a function $\omega_\delta : \intervalleoo0\infty \rightarrow\intervalleoo0\infty$, with limit $0$ at $0$, such that for all $x,y \in A(\delta)$,
    \begin{equation}
        \module{ \MonRFelem\p{(\grad D)_x} - \MonRFelem\p{(\grad D)_y}  }\le \omega_\delta\p{ \norme{x-y}}.
    \end{equation}
    Consequently, if $k\in\intint0{K-1}^d$ is $\p{ \eta \delta, K }$-stable then
    \begin{align*}
        &\frac{1}{K^d}\MonRFelem\p{(\grad D)_{z_k} }\\
            &= \int_{\frac kK  + \intervalleff0{\frac1K}^d } \MonRFelem\p{(\grad D)_{z_k} } \d z\\
            &\le \int_{A(\delta\eta, K) \cap \p{ \frac kK  + \intervalleff0{\frac1K}^d} } \MonRFelem\p{(\grad D)_{z_k} } \d z + \int_{A(\delta\eta, K)^\mathrm{c} \cap \p{ \frac kK  + \intervalleff0{\frac1K}^d}  } \MonRFelem\p{(\grad D)_{z_k} } \d z \\
            &\le \int_{\frac kK  + \intervalleff0{\frac1K}^d }\MonRFelem\p{(\grad D)_z } \d z + \frac1{K^d} \omega_{\eta\delta}\p{\frac dK} + \Leb\p{A(\delta\eta, K)^\mathrm{c} \cap \p{ \frac kK  + \intervalleff0{\frac1K}^d}} \MonRFelem\p{ (b-4\eta)\norme\cdot}.
    \end{align*}
    Summing over $k$ gives
    \begin{align*}
        &\frac{1}{K^d}\sum_{\substack{ k\in \intint0{K-1}^d \\ \p{\eta \delta, K }\text{-stable} } } \MonRFelem\p{(\grad D)_{z_k} }\\
            &\le \int_{\intervalleff01^d }\MonRFelem\p{(\grad D)_z } \d z + \omega_{\eta\delta}\p{\frac dK} + \Leb\p{A(\delta\eta, K)^\mathrm{c} } \MonRFelem\p{ (b-4\eta)\norme\cdot}.
    \end{align*}
    Consequently, letting $K\to\infty$ in \eqref{eqn : MonLDP/UB/KeyInequality} yields
    \begin{equation}\begin{split}
         \MonRFsup(D)
            &\le \int_{\intervalleff01^d }\MonRFelem\p{(\grad D)_z } \d z +  \Leb\p{A(\eta \delta)^\mathrm{c} }  \MonRFelem\p{ (b-4\eta)\norme\cdot}  \\%
            &\quad -d \p{ 1- (1-2\delta)^d + \Leb\p{A\p{\eta \delta  }^\mathrm{c}  } } \log \nu\intervalleff{b-\eta}{b}.
    \end{split}\end{equation}
    Finally, letting $\delta \to 0$ gives \eqref{eqn : MonLDP/UB}.

    We now turn to the general case. For all $\eta>0$  we define the pseudometric  $D_\eta \dpe \frac{b-4\eta}{b} D$. By the previous case,
    \begin{equation*}
        \MonRFsup(D_\eta) \le \int_{\intervalleff01^d} \MonRFelem\p{ (\grad D_\eta)_z }\d z.
    \end{equation*}
    But for all $z\in \intervalleff01^d$, $(\grad D_\eta)_z \le (\grad D)_z$, thus
    \begin{equation*}
        \MonRFsup(D_\eta) \le \int_{\intervalleff01^d} \MonRFelem\p{ (\grad D)_z }\d z.
    \end{equation*}
    By lower semi-continuity of $\MonRFsup$, letting $\eta \to 0$ gives \eqref{eqn : MonLDP/UB} in full generality.
\end{proof}

\subsection{Lower bounding the one-sided rate function}
\label{subsec : MonLDP/LB}
Fix a pseudometric $D\in \Metrics$ for this section. For all integers $K \ge 1$ and $k\in\intint0{K-1}^d$, we define a pseudometric $D_K^k$ on $\intervalleff01^d$ by setting, for all $x,y\in\intervalleff01^d$,
\begin{equation}
    \label{eqn : MonLDP/LB/DefD_kK}
 D_K^k(x,y) \dpe K D\p{ \frac{x+k}K  , \frac{y+k}K }.
\end{equation} 
In other words, $D_K^k$ is the pseudometric obtained from $D$ by rescaling its restriction to the box $\frac kK + \frac1K \intervalleff01^d$. Intuitively, if $\SPT \gtrsim D$, then the restriction to $\SPT$ on each of those boxes is lower bounded by the corresponding $D_K^k$, leading to Lemma \ref{lem : MonLDP/LB/WinCon}. 
\begin{Lemma}
    \label{lem : MonLDP/LB/WinCon}
    For all $K\ge 1$,
    \begin{equation}
        \label{eqn : MonLDP/LB/WinCon}
        \MonRFinf(D) \ge \frac{1}{K^d} \sum_{k \in \intint0{K-1}^d } \MonRFinf(D_K^k).
    \end{equation}
\end{Lemma}
For large values of $K$, $D_K^k$ can be thought of as the pseudometric associated with $(\grad D)_{k/K}$, thus letting $K\to \infty$ in \eqref{eqn : MonLDP/LB/WinCon} will give the desired inequality.

\begin{proof}[Proof of Lemma \ref{lem : MonLDP/LB/WinCon}]
    \stepcounter{tile}
    Fix $K\ge 1$ and $\eps>0$. For all $m\ge 2K$ and $k \in \intint0{K-1}^d$, we define the event    
    \begin{equation}
        E_m(k) \dpe \acc{\forall x,y \in \intint0n^d, \; \BoxPT[\cTile(k)]\p{ x+ (n+1)k, y+ (n+1)k} \ge nD_{K}^k \p{\frac{x}{n}, \frac{y}{n} }  - 2m\eps },   
    \end{equation}
    where
    \begin{equation}
        n\dpe \floor{\frac mK}-1 \text{  and  }\cTile(k) \dpe \intint0{n}^d + (n+1)k .
    \end{equation}
    We first claim that for large enough $m$, for all $k \in \intint{0}{K-1}^d$,
    \begin{equation}
        \label{eqn : MonLDP/LB/WinCon/InclusionSimple}
        \LD_m^+\p{D, \eps} \subseteq E_m(k).
    \end{equation}
    Indeed, fix such $k$ and $x,y\in \intint0n^d$. One checks that
    \begin{equation}
        \label{eqn : MonLDP/LB/WinCon/InclusionSimple/Adjustment}
        \norme{\frac{x + (n+1)k}{m} - \frac{x+nk}{nK}  },  \norme{\frac{y + (n+1)k}{m} - \frac{y+nk}{nK}  }  \le \frac{ \Cl{Adjustment} }n,
    \end{equation}
    where $\Cr{Adjustment}$ only depends on $d$ and $K$. On the event $\LD_m^+(D,\eps)$,
    \begin{align*}
        \BoxPT[\cTile(k)]\p{ x+ (n+1)k, y+ (n+1)k}%
            &\ge \BoxPT[m]\p{ x+ (n+1)k, y+ (n+1)k} \\
            &\ge mD\p{ \frac{x + (n+1)k}{m} , \frac{y + (n+1)k}{m} } - \eps m.
        \intertext{By the triangle inequality and \eqref{eqn : MonLDP/LB/WinCon/InclusionSimple/Adjustment},} 
        \BoxPT[\cTile(k)]\p{ x+ (n+1)k, y+ (n+1)k}%
            &\ge mD\p{ \frac{x+nk}{nK}, \frac{y+nk}{nK}  } - \p{\eps +\frac{2b \Cr{Adjustment}}n }m \\
            &= \frac mK D_{K}^k\p{\frac xn, \frac yn } - \p{\eps +\frac{2b \Cr{Adjustment}}n }m \\
            &\ge n D_{K}^k\p{\frac xn, \frac yn } - \p{\eps +\frac{2b \Cr{Adjustment}}n }m. 
    \end{align*}
    Since $n\to \infty$ when $m \to \infty$, the error term is greater than $-2\eps m$ for large enough $m$, thus \eqref{eqn : MonLDP/LB/WinCon/InclusionSimple} holds.

    Let $m$ be large enough so that \eqref{eqn : MonLDP/LB/WinCon/InclusionSimple} holds. By stationarity, $E_m(k)$ has the same probability as $\LD_n^+(D_{K}^k, 2m\eps/n)$. Moreover, the events $\p{E_m(k)}_{k\in \intint0{K-1}^d }$ live on disjoint sets of edges, hence they are independent. Consequently, for large $m$,
    \[
        \Pb{ \LD_m^+(D,\eps)} \le \prod_{k \in \intint0{K-1}^d} \Pb{ \LD_n^+(D_{K}^k ,2K\eps)},
    \]
    therefore
    \[
        -\frac{1}{m^d} \log \Pb{ \LD_m^+(D,\eps)} \ge \sum_{k\in \intint{0}{K-1}^d} \p{ \frac{n}{m}}^d \cdot \p{ -\frac{1}{n^d}\log \Pb{ \LD_n^+(D_{K}^k ,2K\eps)}  }
    \]
    Taking the lower limit as $m \to \infty$ then letting $\eps \to 0$ yields \eqref{eqn : MonLDP/LB/WinCon}.
\end{proof}
\begin{proof}[Proof of Proposition \ref{prop : MonLDP/LB}]
    For all $z\in \intervalleff01^d$ and $K\ge 1$, we consider $k(z,K) \in \intint0{K-1}^d$ such that $z\in \frac {k(z,K)}K + \frac1K\intervalleff01^d$. If it is not unique, we choose one in any measurable way. Note that \eqref{eqn : MonLDP/LB/WinCon} can be rewritten
    \begin{equation}
        \label{eqn : MonLDP/LB/WinConBis}
        \MonRFinf(D) \ge \int_{\intervalleff01^d} \MonRFinf\p{ D_K^{k(z,K)}} \d z.
    \end{equation}
    Moreover, by Proposition \ref{prop : Gradient/derivee}, for almost every $z\in \intervalleff01^d$,
    \[
        D_K^{k(z,K)} \xrightarrow[K\to \infty]{\Metrics} \d_{(\grad D)_z}.
    \]
    By lower semi-continuity of $\MonRFinf$ and Fatou's lemma, letting $K\to \infty$ in \eqref{eqn : MonLDP/LB/WinConBis} gives
    \begin{align*}
        \MonRFinf(D)%
            &\ge \liminf_{K\to \infty} \int_{\intervalleff01^d} \MonRFinf\p{ D_K^{k(z,K)}} \d z \\
            &\ge \int_{\intervalleff01^d} \liminf_{K\to \infty} \MonRFinf\p{ D_K^{k(z,K)}} \d z \\
            &\ge \int_{\intervalleff01^d} \MonRFinf\p{ \d_{(\grad D)_z} } \d z.
    \intertext{By Proposition \ref{prop : ERF/Existence}, we thus have}
        \MonRFinf(D)%
            &\ge \int_{\intervalleff01^d} \MonRFelem\p{(\grad D)_z } \d z.
    \end{align*}
\end{proof}


\section{Geodesic metrics}
\label{sec : GeoMetrics}
We assume that $0<a<b< \infty$ throughout the section. We gather here the topological results needed in the proof of Theorem \ref{thm : LDP}.

\subsection{Compactness}
\label{subsec : GeoMetrics/compactness}

This subsection aims to show Proposition~\ref{prop : GeoMetrics/compactness/compactness}.
\begin{Proposition}
\label{prop : GeoMetrics/compactness/compactness}
The spaces $\Metrics$ and $\GeoMetrics$ are compact for the uniform distance.
\end{Proposition}
We first state a characterization for elements of $\GeoMetrics$ which is a direct application of the Hopf-Rinow theorem (see e.g. \cite[Proposition~3.7]{Bridson1999}).
\begin{Lemma}
\label{lem : GeoMetrics/compactness/criterion}
Let $D$ be a metric on $\intervalleff01^d$ that satisfies \eqref{eqn : Intro/Main/LDP/GeoMetrics/Bounds}. Then $D\in\GeoMetrics$ if and only if for all $x,y\in \intervalleff01^d$ and $\eps>0$, there exists $z\in \intervalleff01^d$ such that
\begin{align}
\label{eqn : GeoMetrics/compactness/approx_MPP}
    \max\p{ D(x,z), D(z,y) }&\le \frac12 D(x,y) + \eps.
\intertext{Moreover, in this case, for all $x,y\in \intervalleff01^d$, there exists $z\in \intervalleff01^d$ such that}
\label{eqn : GeoMetrics/compactness/exact_MPP}
    D(x,z)= D(z,y) &= \frac12 D(x,y).
\end{align}
\end{Lemma}
\begin{proof}[Proof of Proposition~\ref{prop : GeoMetrics/compactness/compactness}]

Since $\GeoMetrics \subseteq \Metrics$, to prove that $\GeoMetrics$ is compact, it is sufficient to show that it is closed in $\Metrics$. Let $(D_n)_{n\ge 1}$ be a sequence of elements of $\GeoMetrics$ converging to $D$. It is clear that $D$ is a metric satisfying~\eqref{eqn : Intro/Main/LDP/GeoMetrics/Bounds}. In light of Lemma~\ref{lem : GeoMetrics/compactness/criterion}, it is therefore sufficient to prove~\eqref{eqn : GeoMetrics/compactness/approx_MPP}. Let $x,y\in \intervalleff01^d$. For all $n\ge 1$, there exists $z_n \in \intervalleff01^d$ such that 
\begin{equation*}
  D_n(x,z_n) =D_n(z_n,y) =  \frac12D_n(x,y).
\end{equation*}
Hence for all $n$,
\begin{align*}
 D(x, z_n) &\le D_n(x,z_n) + \mathfrak{d}_\infty\p{D, D_n} \\
    &=\frac12D_n(x,y)+ \mathfrak{d}_\infty\p{D, D_n} \\
    &\le \frac12D(x,y)+ 2\mathfrak{d}_\infty\p{D, D_n},
\end{align*}
and the same goes for $D(z_n,y)$. Thus $D$ satisfies~\eqref{eqn : GeoMetrics/compactness/approx_MPP} so $D\in\GeoMetrics$. 
\end{proof}
  

\subsection{Length of the geodesics}
\label{subsec : GeoMetrics/Geodesics}
Lemma~\ref{lem : GeoMetrics/Geodesics/localisation} gives a control on the length of $D$-geodesics, for $D \in \GeoMetrics$.
  \begin{Lemma}
    \label{lem : GeoMetrics/Geodesics/localisation}
  	Let $D\in \GeoMetrics$ and $x\Path \sigma y$ be a $D$-geodesic. Then its $\norme\cdot$-length satisfies
  	\begin{equation}
    \label{eqn : GeoMetrics/Geodesics/localisation_forte}
    \norme{\sigma} \le \frac{b}{a}\norme{x-y} \le \frac {bd}a.
    \end{equation}
  \end{Lemma}
  \begin{proof}
  	Let $0= t_0< t_1 < \dots < t_r= D(x,y)$. By definition of $\sigma$, for all $0\le i \le r-1$,
  	\begin{align*}
  		D\p{\sigma(t_i),\sigma(t_{i+1}) } &= t_{i+1}- t_i.
  		\intertext{The lower bound in~\eqref{eqn : Intro/Main/LDP/GeoMetrics/Bounds} implies}
  		a\norme{\sigma(t_i)-\sigma(t_{i+1}) } &\le  t_{i+1}- t_i.  \notag
      \intertext{Summing over $i$ and applying the upper bound in~\eqref{eqn : Intro/Main/LDP/GeoMetrics/Bounds}, we get}
      a \sum_{i=0}^{r-1} \norme{\sigma(t_i)-\sigma(t_{i+1}) } &\le D(x,y) \le b\norme{x-y}.
  	\end{align*}
  	Taking the supremum over all partitions $\p{0= t_0< t_1 < \dots < t_r= D(x,y)}$ yields~\eqref{eqn : GeoMetrics/Geodesics/localisation_forte}.
  \end{proof}
A straightforward adaptation of this proof gives an analogous result for the pseudometric $\PT_n$.
\begin{Lemma}
    \label{lem : GeoMetrics/Geodesics/localisation_discrete}
    Let $n\ge 1$ and $x\Path\gamma y$ be a discrete $\PT_n$-geodesic. Then its $\norme\cdot$-length satisfies
    \begin{equation}
    \label{eqn : GeoMetrics/Geodesics/localisation_forte_discrete}
    \norme{\gamma} \le \frac{b}{a}\norme{x-y} \le \frac {bd n}a.
    \end{equation}
\end{Lemma}

\subsection{The highway method}
The main result of this subsection is Proposition~\ref{prop : GeoMetrics/HW/HW_cvg}, which states that a pseudometric $D\in \GeoMetrics$ may be approximated by considering the uniform metric associated with $b\norme \cdot$ and then ``activating'' a large number of $D$-geodesics. Moreover, those geodesics may be chosen pairwise disjoint.
\begin{Definition}
    \label{def : GeoMetrics/HW/HW}
    Let $D\in \GeoMetrics$. Consider a sequence $\p{ \sigma_k : \intervalleff0{T_k} \rightarrow \intervalleff01^d }_{k\ge 1}$ of $D$-geodesics. We recursively define the sequence of functions $\p{\HW{D}{ (\sigma_k)_{1\le k \le K} } }_{K\ge 0}$ as follows: $\HW{D}{\emptyset}= d_{b\norme\cdot}$ and for all $K\ge 0$ and $x,y\in\intervalleff01^d$,
    \begin{equation}
        \label{eqn : GeoMetrics/HW/HW}
    \begin{split}
        \HW{D}{ (\sigma_k)_{1\le k \le K +1} }(x,y) \dpe &\HW{D}{ (\sigma_k)_{1\le k \le K} }(x,y) \wedge  \min_{0 \le s,t \le T_{K+1} } \Bigl( \HW{D}{ (\sigma_k)_{1\le k \le K} }(x,\sigma_{K+1}(s) )\\
         &\; + D(\sigma_{K+1}(s) , \sigma_{K+1}(t) ) + \HW{D}{ (\sigma_k)_{1\le k \le K} }( \sigma_{K+1}(t) , y) \Bigr). 
    \end{split}
    \end{equation}
\end{Definition}
\begin{Lemma}
    \label{lem : GeoMetrics/HW/HW_simple_ptys}
    We adopt the same notations as in Definition~\ref{def : GeoMetrics/HW/HW}. 
    \begin{enumerate}[(i)]
        \item For all $K\ge 0$, $\HW{D}{ (\sigma_k)_{1\le k \le K} } \in \Metrics$.
        \item For all $K\ge 0$,
        \begin{equation}
            \label{eqn : GeoMetrics/HW/HW_simple_ptys/noninc}
            \HW{D}{ (\sigma_k)_{1\le k \le K} } \ge \HW{D}{ (\sigma_k)_{1\le k \le K +1} } \ge D.
        \end{equation}
        \item If the sequence $\bigl( ( \sigma_k(0) , \sigma_k(T_k) ) \bigr)_{k\ge 1}$ is dense in $(\intervalleff01^d)^2$ for the usual topology, then
        \begin{equation}
            \label{eqn : GeoMetrics/HW/HW_simple_ptys/Convergence}
            D= \lim_{K\to \infty} \HW{D}{ (\sigma_k)_{1\le k \le K} }.
        \end{equation}
    \end{enumerate}
\end{Lemma}
\begin{proof}
    For all $K\ge 0$, define $D_K \dpe \HW{D}{ (\sigma_k)_{1\le k \le K} }$. It is straightforward to see that for all $K\ge0$, $b\norme{\cdot}\ge D_{K}\ge D_{K+1} \ge D$.

	We show by induction that for all $K\ge 0$, $D_K\in\Metrics$. Let $K\ge0$. Assume that $D_K\in\Metrics$. Let $x,y,z\in \intervalleff{0}{1}^d$. The hardest case to consider when proving the triangle inequality for $D_{K+1}$ between $x,y$ and $z$ is when there exist $x',y',y'',z'' \in \sigma_{K+1}$ such that
	\begin{equation}
		D_{K+1}(x,y) = D_K(x,x')+D(x',y')+D_K(y',y)  \text{ and } D_{K+1}(y,z) = D_K(y,y'')+D(y'',z'')+D_K(z'',z).
	\end{equation}
	Thanks to the definition of $D_{K+1}$, the triangle inequality for $D$ and $D_K\ge D$,
  \begin{align*}
    D_{K+1}(x,z) &\le  D_K(x,x')+ D(x',z'') + D_K(z'',z) \eol
    	&\le D_K(x,x')+D(x',y')+D(y',y) + D(y,y'')+D(y'',z'') +D_K(z'',z) \eol
    	&\le D_K(x,x')+D(x',y')+D_K(y',y) + D_K(y,y'')+D(y'',z'')+D_K(z'',z) \eol
    	&=  D_{K+1}(x,y)+ D_{K+1}(y,z).
  \end{align*}

  As $\Metrics$ is compact (see Proposition~\ref{prop : GeoMetrics/compactness/compactness}), to prove the convergence, it is sufficient to show that $D$ is the unique accumulation point of $(D_K)_{K\ge 0}$. For all $k\ge 1$ and $K\ge k$, \[ D_K\p{\sigma_k(0), \sigma_k(T_k)} = D\p{\sigma_k(0), \sigma_k(T_k)}, \]  thus any accumulation point of $(D_K)_{K\ge 0}$ must coincide with $D$ on a dense subset of $(\intervalleff{0}{1}^d)^2$, therefore must be equal to $D$ by continuity.
\end{proof}

\begin{Definition}
    \label{def : GeoMetrics/HW/HW_Network}
    Let $D\in \GeoMetrics$. We say that a sequence $\p{\sigma_k : \intervalleff0{T_k} \rightarrow \intervalleff01^d}_{k\ge 1}$ of $D$-geodesics is a \emph{highway network} for $D$ if they are pairwise disjoint and~\eqref{eqn : GeoMetrics/HW/HW_simple_ptys/Convergence} holds.
\end{Definition}

\begin{Proposition}
    \label{prop : GeoMetrics/HW/HW_cvg}
    Let $D\in \GeoMetrics$. Then $D$ has a highway network.
\end{Proposition}
The general idea is to consider a family of geodesics $(\hat \sigma_k)_{k\ge 1}$ whose endpoints form a dense subset of $\p{\intervalleff01^d}^2$, then get rid of the intersections in order to create pairwise disjoint paths. We rely on Lemma~\ref{lem : GeoMetrics/HW/HW_cvg/KeyLemma} to cut paths.
\begin{Lemma}
    \label{lem : GeoMetrics/HW/HW_cvg/KeyLemma}
    Let $\gamma : \intervalleff0T \rightarrow \intervalleff01^d$ be a continuous path and $X_0\subseteq \intervalleff01^d$ be a compact set. There exists a countable family $\p{\intervalleff{s_p}{t_p}}_{p\ge 1}$ of pairwise disjoint segments of $\intervalleff0T$ such that
    \begin{enumerate}[(i)]
        \item For all $p\ge 1$, $\gamma\p{\intervalleff{s_p}{t_p}} \cap X_0 = \emptyset$.
        \item For almost every $t\in \intervalleff0T \setminus \p{ \bigcup_{p\ge1} \intervalleff{s_p}{t_p} }$, $\gamma(t)\in X_0$.
    \end{enumerate}
\end{Lemma}
\begin{proof}[Proof of Lemma~\ref{lem : GeoMetrics/HW/HW_cvg/KeyLemma}]
    Since $\gamma$ is continuous, $\gamma^{-1}\p{\intervalleff01^d\setminus X_0}$ is an open subset of $\intervalleff0T$, thus a countable union of disjoint open intervals. Each of those can be covered almost everywhere by a countable union of disjoint segments.
\end{proof}
\begin{proof}[Proof of Proposition~\ref{prop : GeoMetrics/HW/HW_cvg}]
    Let $\p{(x_k, y_k)}_{k\ge 1}$ be a dense sequence in $(\intervalleff01^d)^2$. For all $k\ge 1$, consider a $D$-geodesic $\hat \sigma_k : \intervalleff0{T_k}\rightarrow \intervalleff01^d$ from $x_k$ to $y_k$. By Lemma~\ref{lem : GeoMetrics/HW/HW_simple_ptys},
    \begin{equation}
        \label{eqn : GeoMetrics/HW/HW_cvg/input_cvg}
        D = \lim_{K\to\infty} \HW{D}{ (\hat \sigma_k)_{1\le k \le K} }.
    \end{equation}
    Let $k\ge1$. Applying Lemma~\ref{lem : GeoMetrics/HW/HW_cvg/KeyLemma} to $\gamma = \hat \sigma_k$ and $X_0 = \bigcup_{i=1}^{k-1} \hat \sigma_i$, we know that there exists a family of pairwise disjoint segments $\p{\intervalleff{s_{k,p}}{t_{k,p} } }_{p\ge 1}$ of $\intervalleff{0}{T_k}$ such that
    \begin{enumerate}[(i)]
        \item For all $p\ge1$, $\hat \sigma_k\p{ \intervalleff{s_{k,p}}{t_{k,p} } } \cap \p{\bigcup_{i=1}^{k-1} \hat \sigma_i } = \emptyset$.
        \item For almost every $t\in \intervalleff{0}{T_k} \setminus \p{\bigcup_{p\ge 1} \intervalleff{s_{k,p} }{t_{k,p} } }$, $\hat \sigma_k(t) \in \bigcup_{i=1}^{k-1} \hat \sigma_i$.
    \end{enumerate}
    We denote by \[\begin{split} \sigma_{k,p} : \intervalleff0{t_{k,p} - s_{k,p} } &\longrightarrow \intervalleff01^d \\ t &\longmapsto \hat\sigma_k(t+ s_{k,p}). \end{split}\] Note that it is a $D$-geodesic. Let $\p{\sigma_r}_{r\ge 1}$ be any enumeration of $\p{\sigma_{k,p}}_{k,p \ge 1}$. It is sufficient to show
    \begin{equation}
        \label{eqn : GeoMetrics/HW/HW_cvg/WinCon}
        D = \lim_{R\to\infty}\HW{D}{ ( \sigma_r)_{1\le r \le R} }.
    \end{equation}
    Let $\eps>0$. By~\eqref{eqn : GeoMetrics/HW/HW_cvg/input_cvg} there exists $K>0$ such that
    \begin{equation}
        \label{eqn : GeoMetrics/HW/HW_cvg/bigK}
        \UnifDistance\p{D , \HW{D}{ (\hat \sigma_k)_{1\le k \le K} }  } \le \eps.
    \end{equation}
    Besides, by the definition of the $\intervalleff{s_{k,p}}{t_{k,p}}$ and Lipschitz continuity, there exists $P>0$ such that
    \begin{equation}
        \label{eqn : GeoMetrics/HW/HW_cvg/controle_reste}
        \sum_{k=1}^K \sum_{p=P+1}^\infty \norme[1]{\sigma_{k,p}} \le \eps.
    \end{equation}
    Let $R>0$ be an integer large enough so that $\acc{\sigma_{k,p}}_{\substack{1\le k \le K \\ 1\le p \le P} } \subseteq \acc{\sigma_r}_{1\le r \le R}$, and $x,y \in \intervalleff01^d$. Any $\HW{D}{ (\hat \sigma_k)_{1\le k \le K} }$-geodesic between $x$ and $y$ is a concatenation of a finite number of line segments and subpaths of $\hat \sigma_k$, for $1\le k \le K$. Thanks to~\eqref{eqn : GeoMetrics/HW/HW_cvg/controle_reste}, the latter may be covered by subpaths of the $\sigma_{k,p}$, for $1\le k \le K$ and $1\le p \le P$, except along a $\norme[1]\cdot$-length smaller than $\eps$. Consequently, by the triangle inequality,
    \begin{equation}
        \label{eqn : GeoMetrics/HW/HW_cvg/bigI}
        \HW{D}{ ( \sigma_r)_{1\le r \le R} }(x,y) \le \HW{D}{ (\hat \sigma_k)_{1\le k \le K} }(x,y) + b \eps.
    \end{equation}
    Applying~\eqref{eqn : GeoMetrics/HW/HW_simple_ptys/noninc},~\eqref{eqn : GeoMetrics/HW/HW_cvg/bigK} and~\eqref{eqn : GeoMetrics/HW/HW_cvg/bigI}, we deduce that for large enough $R$, for all $x,y\in \intervalleff01^d$,
    \begin{equation}
        D(x,y) \le \HW{D}{ ( \sigma_r)_{1\le r \le R} }(x,y) \le D(x,y) + (b +1)\eps,
    \end{equation} 
    hence~\eqref{eqn : GeoMetrics/HW/HW_cvg/WinCon}.
\end{proof}


\subsection{Properties of the passage time}
  \begin{Proposition}
  \label{prop : GeoMetrics/Geodesics/BoxSPT_espace}
  	For all $n\ge 1$, $\SPT \in \GeoMetrics$ almost surely.
  \end{Proposition}
  \begin{proof}
  	Let $n\ge1$. Then $\SPT$ is clearly a metric that satisfies~\eqref{eqn : Intro/Main/LDP/GeoMetrics/Bounds}, so it suffices to prove that it satisfies~\eqref{eqn : GeoMetrics/compactness/approx_MPP}. Let $x,y\in \intervalleff 01^d$ and $\eps>0$. Without loss of generality, we can assume $x\neq y$.
  	Let ${\pi=(nx=x_0,}\dots, x_r=ny)$ be a finite sequence in $\intervalleff{0}{n}^d$ such that
  \begin{equation*}
    \PathPT{\pi} \le\CPT(nx,ny) +\eps.
  \end{equation*}
  There exists $r_0\in \intint{0}{r-1}$ such that 
  \begin{equation*}
    \sum_{j=0}^{r_0-1}\PathPT{x_j, x_{j+1}}\le \frac12\tau(\pi) \text{  and  } \sum_{j=r_0}^{r-1}\PathPT{x_j, x_{j+1}}> \frac12\tau(\pi).
  \end{equation*}
  A straightforward consequence of the definition of $\EPT$ on $\R^d\times \R^d$ (see~\eqref{eqn : Intro/main_thm/generalized_PathPT2}) is that for all $z,z'\in\intervalleff{x_{r_0}}{x_{r_0+1}}$,
  \begin{equation*}
    \PathPT{z ,z' }  \le \frac{\norme{z-z'}}{\norme{x_{r_0} -x_{r_0+1}}}\PathPT{x_{r_0}, x_{r_0+1}};
  \end{equation*}
  hence, denoting by $z$ the point on the segment $\intervalleff{x_{r_0}}{x_{r_0+1} }$ such that
  \begin{equation*}
      \sum_{j=0}^{r_0-1} \PathPT{x_j, x_{j+1}} + \frac{\norme{x_{r_0} - z} }{\norme{x_{r_0} - x_{r_0+1} } } \PathPT{x_{r_0}, x_{r_0+1}}%
      =\frac12 \PathPT{\pi},
  \end{equation*}
  we have
  \begin{equation*}
    \sum_{j=0}^{r_0-1}\PathPT{x_j, x_{j+1}} + \tau\p{x_{r_0}, z} \le \frac12\PathPT{\pi} \text{  and  } \PathPT{z,x_{r_0+1}} + \sum_{j=r_0+1}^{r-1}\PathPT{x_j, x_{j+1}}\le\frac12\tau(\pi).
  \end{equation*}
  Consequently,
  \begin{equation*}
    \max\p{ \SPT(x,z/n), \SPT(z/n,y) }\le \frac12 \SPT(x,y) + \frac\eps{2n},
  \end{equation*}
  which concludes the proof.
  \end{proof}

\section{Large deviation principle}
\label{sec : LDP}
In this section we assume that $0<a<b<\infty$ and prove Theorem \ref{thm : LDP}. To do so, we first define, for all $D\in \GeoMetrics$ and $\eps>0$, the event
\begin{equation}
  \LD_n(D,\eps) \dpe \acc{ \forall x,y\in \intint0n^d,\; \module{ \BoxPT(x,y)- n D\p{\frac xn, \frac yn} } \le n\eps  }.
\end{equation}
We then define
\begin{equation}
  \RFinf(D) \dpe \lim_{\eps \to 0} \liminf_{n\to\infty}-\frac{1}{n^d}\log \Pb{ \LD_n(D,\eps) } \text{  and  }\RFsup(D) \dpe \lim_{\eps \to 0} \limsup_{n\to\infty}-\frac{1}{n^d}\log \Pb{ \LD_n(D,\eps) }.
\end{equation}
Both $\RFinf$ and $\RFsup$ are rate functions by Lemma \ref{lem : LD/RateFunctions}. Proposition \ref{prop : LDP/wLDP} states that Theorem \ref{thm : LDP} still holds for those rate functions, on the space $\GeoMetrics$.
\begin{Proposition}
    \label{prop : LDP/wLDP}
    For all $D \in \GeoMetrics$,
    \begin{equation}
        \label{eqn : LDP/wLDP}
        \RFinf(D) = \RFsup(D) = \MonRF(D).
    \end{equation}
\end{Proposition}

From this result follows the large deviation principle for $\SPT$.
\begin{proof}[Proof of Theorem \ref{thm : LDP}]
    Almost surely, for all $n\ge1$, $\SPT$ belongs to the space $\GeoMetrics$, which is compact (see Propositions \ref{prop : GeoMetrics/compactness/compactness} and \ref{prop : GeoMetrics/Geodesics/BoxSPT_espace}) therefore by Lemma \ref{lem : Intro/LD/UpperLowerRF} it is sufficient to show that for all $D\in \GeoMetrics$, 
    \begin{equation}
        \label{eqn : LDP/WinCon}
         \lim_{\eps\to0} \liminf_{n\to\infty} -\frac{1}{n^d}\log \Pb{\UnifDistance\p{\SPT, D} \le \eps } = \lim_{\eps\to0} \limsup_{n\to\infty} -\frac{1}{n^d}\log \Pb{\UnifDistance\p{\SPT, D} \le \eps } = \MonRF(D).
    \end{equation}
    By Proposition \ref{prop : GeoMetrics/Geodesics/extension_tranquille} and the Lipschitz property for $\CPT$, for all $n\ge 1$ and $x,y\in \intervalleff01^d$,
    \[
        \module{ \BoxPT\p{ \floor{nx}, \floor{ny} } - \CPT\p{nx , ny} }\le 2bd.
    \]
    Consequently, for all $\eps>0$, for large enough $n$,
    \[
        \LD_n(D,\eps) \subseteq \acc{ \UnifDistance(\SPT, D) \le 2\eps } \subseteq \LD_n(D,3\eps).
    \]
    Applying \eqref{eqn : LDP/wLDP} thus yields \eqref{eqn : LDP/WinCon} and concludes the proof. 
\end{proof}


\subsection{Existence of the rate function: proof of Proposition \ref{prop : LDP/wLDP}}
\label{subsec : LDP/wLDP}
Fix $D\in \GeoMetrics$. Our strategy consists in modifying a small set $\FastE$ of edge weights in a realization of $\LD_n^+(D,\eps)$ in order to create a realization of $\LD_n(D,\eps)$. Lemma \ref{lem : LDP/wLDP/FastE}, whose proof is postponed to Section \ref{subsec : LDP/wLDP/ProofsPaths}, gives the existence of a suitable set $\FastE$.



\begin{Lemma}
\label{lem : LDP/wLDP/FastE}
	Consider a family of disjoint $D$-geodesics $(\sigma^1, \dots, \sigma^K)$ and $\delta >0$. For large enough $n$, there exists a set of edges $\FastE \subseteq \edges{\intint0n^d}$, with the following property: define \[ \tau_e^* \dpe a\ind{e \in \FastE} + b\ind{e\notin \FastE},\] and denote by $\BoxPT^*$ the associated pseudometric on $\intint0n^d$. Then
        \begin{equation}
            \label{eqn : LDP/wLDP/FastE}
            \forall x,y \in \intint0n^d, \quad \module{ \BoxPT^*(x,y) - n \HW{D}{ \p{\sigma^k}_{1\le k \le K} } \p{\frac xn, \frac yn}} \le n\delta.
        \end{equation}
    Moreover, for all $e\in \FastE$, \begin{equation}\label{eqn : LDP/wLDP/FastE/Localisation} \d\p{e, \bigcup_{k=1}^K n\sigma^k} \le d+1,\end{equation}
    and
    \begin{equation}
        \label{eqn : LDP/wLDP/UnifDistanceinal}
        \# \FastE \le \Cl{CardinalFastE} K n,
    \end{equation}
    where $\Cr{CardinalFastE}$ only depends on $a$, $b$ and $d$.
\end{Lemma}

\begin{proof}[Proof of Proposition \ref{prop : LDP/wLDP}]
    The inequality $\RFinf(D) \ge \MonRF(D)$ is a direct consequence of the inclusion $\LD_n(D,\eps) \subseteq \LD_n^+(D,\eps)$, for all $n\ge 1$ and $\eps>0$. It is thus sufficient to show that $\RFsup(D)\le \MonRF(D)$.

    Fix $\delta, \eta >0$. By Proposition \ref{prop : GeoMetrics/HW/HW_cvg}, there exists a finite family  $\p{ \sigma^k : \intervalleff0{T_k}  \rightarrow \intervalleff01^d }_{1\le k \le K}$ of disjoint $D$-geodesics such that
    \begin{equation}
        \label{eqn : LDP/wLDP/ApproxMetric}
        \UnifDistance\p{D, \HW D{(\sigma^k)_{1\le k \le K} } }  \le \eta.
    \end{equation}
    \stepcounter{tile}
    Let $n \ge 1$. Denote by $\FastE$ the set of edges given by Lemma \ref{lem : LDP/wLDP/FastE}, with the choice of parameter $\delta' = \delta^2$.

    Define the set of edges lying within distance $2n\delta$ of $\FastE$, i.e.\ \[ \Hard \dpe \set{e \in \edges{\intint0n^d} }{ \d\p{ e,\FastE } \le 2 n\delta }. \] By \eqref{eqn : LDP/wLDP/FastE/Localisation}, for large enough $n\ge 1$,
    \begin{equation}
        \label{eqn : LDP/wLDP/BoundCardHard}
        \#\Hard \le \Cl{Hard} \p{ K\delta + \sum_{k=1}^K \norme{\sigma^k}  }\delta^{d-1}  n^d,
    \end{equation}
    with $\Cr{Hard} = \Cr{Hard}(d)$, and
    \begin{equation}
        \label{eqn : LDP/wLDP/Protection}
         \d\p{ \FastE , \edges{\intint0n^d} \setminus \Hard  } \ge n\delta.
    \end{equation}
    Consider two families of i.i.d.\ weights\footnote{The superscripts respectively stand for \emph{slow} and \emph{fast}.} $\tau^\rS = (\tau_e^\rS)_{e \in \edges{\intint0n^d} }$ and $\tau^\rF = (\tau^\rF_e)_{e\in \FastE }$ with distribution $\nu$; assume that $\tau^\rS$ and $\tau^\rF$ are independent. We construct the family of edge weights $\tau = (\tau_e)_{e\in \edges{\intint0n^d} }$, where
    \begin{equation*}
        \tau_e \dpe \begin{cases}
            \tau^\rF_e \quad &\text{ if } e \in \FastE,\\
            \tau^\rS_e \quad &\text{ if } e \in \edges{\intint0n^d} \setminus \FastE.
        \end{cases}
    \end{equation*}
    The family $\tau$ is also i.i.d.\ with distribution $\nu$. We define the event
    \begin{equation}
        \begin{split}
            \cFav &\dpe \LD_n^+(D, \delta^2)[\tau^\rS] \cap \p{\bigcap_{e\in \Hard} \acc{\tau^\rS_e \ge b-\eta}}  \cap \p{\bigcap_{e\in \FastE} \acc{\tau^\rF_e \le a+\eta }}.
        \end{split}
    \end{equation}
    We claim that
    \begin{equation}
        \label{eqn : LDP/wLDP/PbFav}
        \limsup_{n\to\infty}-\frac1{n^d}\log \Pb{\cFav} \le \MonRF(D) -\Cr{Hard} \p{K \delta + \sum_{k=1}^K \norme{\sigma^k}  }\delta^{d-1}  \log \nu\intervalleff{b-\eta}{b},
    \end{equation}
    and for large enough $n$,
    \begin{equation}
        \label{eqn : LDP/wLDP/InclusionFav}
        \cFav \subseteq \LD_n\p{D, \frac{2bd}{a}(\eta + \delta)}.
    \end{equation}

    Indeed, by independence and by the FKG inequality,
    \begin{align*}
            \Pb{\cFav} &\ge \Pb{ \LD_n^+(D, \delta^2)} \cdot \nu\intervalleff{b-\eta}{b}^{\#\Hard}  \cdot \nu\intervalleff{a}{a+\eta} ^{\#\FastE}.
        \intertext{Taking the $\log$, multiplying by $-\frac1{n^d}$ and applying the inequalities \eqref{eqn : LDP/wLDP/UnifDistanceinal} and \eqref{eqn : LDP/wLDP/BoundCardHard} yields }
        -\frac1{n^d} \log\Pb{\cFav}%
            &\le -\frac1{n^d} \log \Pb{ \LD_n^+(D, \delta^2)} - \Cr{Hard} \p{ K\delta + \sum_{k=1}^K \norme{\sigma^k}  }\delta^{d-1} \log \
            \nu\intervalleff{b-\eta}{b} + \grando\p{\frac1{n^{d-1}} }.
        \end{align*}
    Taking the upper limit as $n\to\infty$ gives \eqref{eqn : LDP/wLDP/PbFav}.

    Assume that $\cFav$ occurs. By Lemma \ref{lem : LDP/wLDP/FastE}, if $n$ is large enough then
    \begin{equation}
            \label{eqn : LDP/wLDP/InclusionFav/FastE}
            \forall x,y \in \intint0n^d, \quad \module{ \BoxPT^*(x,y) - n \HW{D}{ \p{\sigma^k}_{1\le k \le K} } \p{\frac xn, \frac yn}} \le n\delta^2.
        \end{equation}
    Fix $x,y \in \intint0n^d$. By \eqref{eqn : LDP/wLDP/ApproxMetric} and \eqref{eqn : LDP/wLDP/InclusionFav/FastE},
    \begin{align}
        \BoxPT^*(x,y)%
            &\le n D\p{ \frac xn  , \frac yn } + (\eta + \delta^2) n.\nonumber
        \intertext{Moreover, discrete geodesics from $x$ to $y$ for $\BoxPT^*$ has length at most $\frac{bdn}{a}$, and for all $e\in \edges{\intint0n^d} $, $\tau_e \le \tau_e^* + \eta$, thus }
        \BoxPT(x,y) &\le \BoxPT^*(x,y) + \frac{bd \eta n}{a}.\nonumber
        \intertext{Combining the last two displays gives}
        \BoxPT(x,y)%
            &\le n D\p{ \frac xn  , \frac yn } + \cro{ \p{1  + \frac {bd}a}\eta + \delta^2 } n.
        \label{eqn : LDP/wLDP/InclusionFav/UB}
    \end{align}
    We now give a lower bound for $\BoxPT(x,y)$. Let $\gamma$ be a discrete geodesic from $x$ to $y$ for $\BoxPT$. Decompose $\gamma$ as
    \[
        x= x_0 \Path{\gamma_0} x_0' \Path{\gamma_0'} x_1 \Path{\gamma_1} x_1' \Path{\gamma_1'} \dots \Path{\gamma_{R-1}'} x_R \Path{\gamma_R} x_R' =y,
    \]
    where
    \begin{enumerate}[(i)]
        \item for all $i\in \intint0R$, $\gamma_i$ either avoids  $\FastE$ or is included in $ \Hard$.
        \item for all $i \in \intint{0}{R-1}$, $\gamma_i'$ has one endpoint in $\FastE$, one endpoint in $\edges{\intint0n^d} \setminus \Hard$ and all its other vertices are in $\Hard \setminus \FastE$.
    \end{enumerate}

    Let $i\in\intint0{R-1}$. If $\gamma_i$ avoids $\FastE$, then 
    \begin{align}
        \tau(\gamma_i) + \tau(\gamma_i')
            &= \tau^\rS(\gamma_i) + \tau^\rS(\gamma_i') \nonumber \\
            &\ge \BoxPT^\rS(x_i, x_{i+1}),  \nonumber
        \intertext{and since $\LD_n^+(D, \delta^2)[\tau^\rS]$ occurs,}
        \label{eqn : LDP/wLDP/InclusionFav/LB/Case1}
        \tau(\gamma_i) + \tau(\gamma_i')
            &\ge nD\p{\frac{ x_{i} }{n } , \frac{ x_{i+1} }{n }} - \delta^2 n.
    \end{align}
    {Otherwise, $\gamma_i\subseteq  \Hard$, therefore }
    \begin{align}
        \tau(\gamma_i) + \tau(\gamma_i') 
            &\ge \tau^*(\gamma_i) + \tau^*(\gamma_i') - \p{ \norme{\gamma_i} + \norme{\gamma_i'} }\eta \nonumber \\
            &\ge \BoxPT^*(x_i, x_{i+1}) - \p{ \norme{\gamma_i} + \norme{\gamma_i'} }\eta  \nonumber.
        \intertext{Plugging in \eqref{eqn : LDP/wLDP/InclusionFav/FastE} gives}
            \tau(\gamma_i) + \tau(\gamma_i') 
            &\ge n \HW D{ \p{\sigma^k}_{1\le k \le K} } \p{\frac{ x_{i} }{n } , \frac{ x_{i+1} }{n }} -  (  \norme{\gamma_i} + \norme{\gamma_i'} )\eta - \delta^2n \\
            &\ge n  D \p{\frac{ x_{i} }{n } , \frac{ x_{i+1} }{n }} - (  \norme{\gamma_i} + \norme{\gamma_i'} )\eta - \delta^2n.
        \label{eqn : LDP/wLDP/InclusionFav/LB/Case2}
        \intertext{By \eqref{eqn : LDP/wLDP/InclusionFav/LB/Case1} and \eqref{eqn : LDP/wLDP/InclusionFav/LB/Case2}, for all $i \in \intint{0}{R-1}$,}
         \tau(\gamma_i) + \tau(\gamma_i')
            &\ge nD\p{\frac{ x_{i} }{n } , \frac{ x_{i+1} }{n }} - (  \norme{\gamma_i} + \norme{\gamma_i'} )\eta - \delta^2n.
            \label{eqn : LDP/wLDP/InclusionFav/OneBit}
         \intertext{Similarly,}
         \tau(\gamma_R)
            &\ge nD\p{\frac{ x_{R} }{n } , \frac{ x_{R}' }{n }} -  \norme{\gamma_R}  \eta - \delta^2n.
            \label{eqn : LDP/wLDP/InclusionFav/OneBitBis}
    \end{align}
    Summing \eqref{eqn : LDP/wLDP/InclusionFav/OneBit} over all $i\in \intint0{R-1}$ and \eqref{eqn : LDP/wLDP/InclusionFav/OneBitBis} yields, by the triangle inequality,
    \begin{align*}
        \tau(\gamma)%
            &\ge nD\p{\frac xn, \frac yn} -   \norme\gamma \eta  -(R+1)\delta^2 n.
    \end{align*}
    But \eqref{eqn : LDP/wLDP/Protection} implies that for all $i \in \intint0{R-1}$, $\norme{\gamma_i'} \ge n\delta$, and by Lemma \ref{lem : GeoMetrics/Geodesics/localisation_discrete}, $\norme\gamma \le \frac{bdn}{a}$. Hence $    R\le \frac{bd}{a \delta}$. We finally get
    \begin{equation}
        \label{eqn : LDP/wLDP/InclusionFav/LB}
        \BoxPT(x,y)%
            \ge nD\p{\frac xn, \frac yn} -   \frac{bd}{a}(\eta + 2\delta)n.
    \end{equation}
    The inequalities \eqref{eqn : LDP/wLDP/InclusionFav/UB} and \eqref{eqn : LDP/wLDP/InclusionFav/LB} imply that, for sufficiently large $n$, \eqref{eqn : LDP/wLDP/InclusionFav} holds.

    Equations \eqref{eqn : LDP/wLDP/PbFav} and \eqref{eqn : LDP/wLDP/InclusionFav} imply
    \begin{equation}
        \label{eqn : LDP/wLDP/WinCon}
        \limsup_{n\to\infty}-\frac1{n^d}\log \Pb{\LD_n\p{D, \frac{2bd}{a}(\eta + \delta)}} \le \MonRF(D) -\Cr{Hard} \p{ K\delta + \sum_{k=1}^K \norme{\sigma^k}  }\delta^{d-1}  \log \nu\intervalleff{b-\eta}{b}.
    \end{equation}
    Letting $\delta\to0$ then $\eta \to 0$ yields $\RFsup(D) \le \MonRF(D)$, which concludes the proof.
\end{proof}

\subsection{Constructing the modification set}
\label{subsec : LDP/wLDP/ProofsPaths}
We first check in Lemma \ref{lem : LDP/wLDP/approx_path} that any $D$-geodesic can be approximated by a discrete path with controlled length, then we construct in Lemma \ref{lem : LDP/wLDP/FastESimple} the set $\FastE$ in the special case where only one $D$-geodesic is involved, and finally we proceed with the proof of Lemma \ref{lem : LDP/wLDP/FastE}.
\begin{Lemma}
\label{lem : LDP/wLDP/approx_path}
	For all $D$-geodesics $x\Path{\sigma}y$ and all $n \ge 1$, there exists a discrete path $\alpha = (\alpha_j)_{j=0}^p$ in $\intint0n^d$ and a surjective, nondecreasing, right-continuous map ${\bj : \intervalleff{0}{D(x,y)}\rightarrow \intint0p }$ such that
	\begin{align}
     	\label{eqn : LDP/wLDP/approx_path1}\forall 0 \le t \le D(x,y), \quad \norme{\alpha_{\bj(t)} - n\sigma(t)} &\le d+1 \\
     	\label{eqn : LDP/wLDP/approx_path2}\forall 0 \le t_1<t_2 \le D(x,y), \quad  \bj(t_2)-\bj(t_1) &\le  n\int_{t_1}^{t_2} \norme{\sigma'(u)}\d u +d.
   \end{align}
\end{Lemma}
Lemma~\ref{lem : Gradient/Rademacher_kindof} guarantees that the integral in~\eqref{eqn : LDP/wLDP/approx_path2} is well-defined. The inequality~\eqref{eqn : LDP/wLDP/approx_path2}, along with Lemma~\ref{lem : GeoMetrics/Geodesics/localisation}, implies
\begin{align}
	\label{eqn : LDP/wLDP/approx_path_csq1}
	\bj(t_2)-\bj(t_1) &\le  \frac{ n }{a}(t_2-t_1) +d.
	\intertext{In particular,}
	\label{eqn : LDP/wLDP/approx_path_csq2}
	p &\le \frac{ n b d}{a}+d.
\end{align}

\begin{proof}[Proof of Lemma \ref{lem : LDP/wLDP/approx_path}]
	For all $0\le t \le D(x,y)$, we write ${\sigma(t)\dpe \p{\sigma_1(t), \sigma_2(t),\dots, \sigma_d(t)} }$. The idea of the approximation is to replace $n \sigma_i(t)$ by the last visited integer. At times $t$ such that several coordinates of $n\sigma(t)$ are integers, this approximation may have a jump between two non-adjacent vertices of $\Z^d$. Since we aim to build a discrete path, the existence of such times causes a minor problem. Applying a small translation neutralizes this obstacle.

	We claim that there exists $z\in \ball{0,1}\cap \intervallefo{0}{\infty}^d$ such that for all $0\le t \le D(x,y)$,
	\begin{equation}
	\label{eqn : LDP/wLDP/approx_path/translation}
		\#\acc{i\in\intint1d \Bigm| n\sigma_i(t)+z_i\in \Z  }\le 1.
	\end{equation}
	Indeed, let $1\le i_1 < i_2 \le d$ and $v \in \Z^2$. Consider
	\begin{align*}
		A(i_1, i_2,v) &\dpe \acc{z\in \ball{0,1} \Bigm| \exists 0\le t \le D(x,y), \quad n\sigma_{i_1}(t)+z_{i_1}=v_1\text{ and } n\sigma_{i_2}(t)+z_{i_2}= v_2 }\\
		&\subseteq \set{z = (z_1,\dotsc, z_d)\in \R^d}{(z_{i_1}, z_{i_2})\in\acc{v - \p{ n \sigma_{i_1}(t), n\sigma_{i_2}(t)}, \quad 0\le t \le D(x,y)} }.		
	\end{align*}
	Besides, $t\mapsto \p{ n \sigma_{i_1}(t), n \sigma_{i_2}(t)}$ is Lipschitz therefore its image has Lebesgue measure zero (by the Lipschitz property it may be covered by a family of $K$ balls of radius $\grando(1/K)$). Hence the set $\displaystyle \bigcup_{\substack{1\le i_1<i_2\le d \\ v\in \Z^2} } A(i_1, i_2,v)$ of points $z\in \ball{0,1}$ such that \eqref{eqn : LDP/wLDP/approx_path/translation} fails has measure zero.

	Consider the map
	\begin{align*}
		\hat \sigma : \intervalleff{0}{D(x,y)} &\longrightarrow \Z^d\\
		t &\longmapsto \p{\hat \sigma_1(t), \dots, \hat \sigma_d(t)},
	\end{align*}
	where
	\begin{equation*}
		\hat\sigma_i(t) =%
			\begin{cases}
				\floor{n \sigma_i(0) + z_i} %
					&\text{ if } \p{n\sigma_i\p{\intervalleff{0}{t}} + z_i } \cap \Z = \emptyset,\\
				n \sigma_i\p{\sup\acc{s \in \intervalleff0t \mid n \sigma_i(s) + z_i \in \Z } } +z_i%
					&\text{ otherwise.}
			\end{cases}
	\end{equation*}
	In other words, up to a translation, $\hat\sigma_i(t)$ is an approximation of $n \sigma_i(t)$ by the last visited integer, except before the first visit. The map $\hat\sigma \dpe (\hat\sigma_1,\dots, \hat\sigma_d)$ is right-continuous. Moreover, \eqref{eqn : LDP/wLDP/approx_path/translation} implies that its jumps have norm $1$. Let $p$ be the total number of jumps of $\hat\sigma$ and  $\bj(t)$ be the number of jumps up to time $t$. It is straightforward to see that $\bj$ is nondecreasing, right-continuous and surjective on $\intint0p$. For all $j\in \intint0p$, define
	\begin{equation*}
	 	\alpha_j \dpe \hat \sigma(t_j),
	\end{equation*}
	where $t_j$ is any element of $\bj^{-1}\p{j}$ (e.g. the time of the $j$\textsuperscript{th} jump). The inequality~\eqref{eqn : LDP/wLDP/approx_path1} is clear.

	Let $1\le i \le d$. For all pairs $u_1 <u_2$ of consecutive jump instants of $\hat\sigma_i$, Lemma~\ref{lem : Gradient/Rademacher_kindof} yields
	\begin{align*}
      1  &= \module{\hat \sigma_i(u_1) - \hat \sigma_i(u_2) } \\
      &= n \module{\sigma_i(u_1)  - \sigma_i(u_2)} \\
      &= n \module{ \int_{u_1}^{u_2}\sigma_i'(u) \d u }\\
      &\le n \int_{u_1}^{u_2}\module{\sigma_i'(u)} \d u.
    \end{align*}
    Thus, for all $0< t_1 < t_2 \le D(x,y)$, by denoting $t_1 \le u_1 < \dots < u_r \le t_2$ the jump instants of $\hat \sigma_i$ in $\intervalleff{t_1}{t_2}$, we have
    \begin{align*}
      \# \acc{ u \in\intervalleff{t_1}{t_2} \mid \hat \sigma_i(u^-)\neq \hat \sigma_i(u)   } &= r-1 +1 \\%
      	&\le n \int_{u_1}^{u_r} \module{\sigma_{i}'(u)} \d u +1\\
      	&\le n\int_{t_1}^{t_2} \module{\sigma_{i}'(u)} \d u +1.
    \end{align*}
    Summing over $i$ we obtain~\eqref{eqn : LDP/wLDP/approx_path2}.
\end{proof}
\begin{Lemma}
    \label{lem : LDP/wLDP/FastESimple}
    Consider a $D$-geodesic $\sigma$ and $\delta>0$. Then for large enough $n$, there exists a set of edges $\FastE(\sigma) \subseteq \edges{\intint0n^d}$, with the following property: defining \[ \tau_e^* \dpe a\ind{e \in \FastE(\sigma)} + b\ind{e\notin \FastE(\sigma)},\] and denoting by $\BoxPT^*$ the associated pseudometric on $\intint0n^d$, we have
        \begin{equation}
            \label{eqn : LDP/wLDP/FastESimple}
            \forall x,y \in \intint0n^d, \quad \module{ \BoxPT^*(x,y) - n \HW{D}{ \sigma } \p{\frac xn, \frac yn}} \le n\delta.
        \end{equation}
    Moreover, for all $e\in \FastE(\sigma)$, \begin{equation}\label{eqn : LDP/wLDP/FastESimple/Localisation} \d\p{e, n\sigma} \le d+1,\end{equation}
    \begin{equation}
        \label{eqn : LDP/wLDP/FastESimple/Cardinal}
        \# \FastE(\sigma) \le \Cr{CardinalFastE}  n,
    \end{equation}
    where $\Cr{CardinalFastE}$ only depends on $a$, $b$ and $d$.
\end{Lemma}

\begin{proof}[Proof of Lemma~\ref{lem : LDP/wLDP/FastESimple}]
	 The general idea behind the construction of $\FastE(\sigma)$ is to set the passage time of all edges along $\alpha$ between $\alpha_{\bj(0)}$ and $\alpha_{\bj(\delta^2)}$ to $a$, choose $t_1\ge \delta^2$ adequately, set the passage time of all edges along $\alpha$ between $\alpha_{\bj(t_1)}$ to $\alpha_{\bj(t_1+\delta^2)}$ to $a$ and so on. The $t_k$ are chosen as small as possible, but such that the metric never gets too far below the target time (see Figure~\ref{fig : monotonicity/HW_is_free/modification/GeneralPlan}). Fix $0<\delta< 1$.
	\begin{figure}
	\center
	\def\svgwidth{0.6\textwidth}
	\begingroup%
	  \makeatletter%
	  \providecommand\color[2][]{%
	    \errmessage{(Inkscape) Color is used for the text in Inkscape, but the package 'color.sty' is not loaded}%
	    \renewcommand\color[2][]{}%
	  }%
	  \providecommand\transparent[1]{%
	    \errmessage{(Inkscape) Transparency is used (non-zero) for the text in Inkscape, but the package 'transparent.sty' is not loaded}%
	    \renewcommand\transparent[1]{}%
	  }%
	  \providecommand\rotatebox[2]{#2}%
	  \newcommand*\fsize{\dimexpr\f@size pt\relax}%
	  \newcommand*\lineheight[1]{\fontsize{\fsize}{#1\fsize}\selectfont}%
	  \ifx\svgwidth\undefined%
	    \setlength{\unitlength}{426.61417323bp}%
	    \ifx\svgscale\undefined%
	      \relax%
	    \else%
	      \setlength{\unitlength}{\unitlength * \real{\svgscale}}%
	    \fi%
	  \else%
	    \setlength{\unitlength}{\svgwidth}%
	  \fi%
	  \global\let\svgwidth\undefined%
	  \global\let\svgscale\undefined%
	  \makeatother%
	  \begin{picture}(1,0.53488377)%
	    \lineheight{1}%
	    \setlength\tabcolsep{0pt}%
	    \put(0,0){\includegraphics[width=\unitlength,page=1]{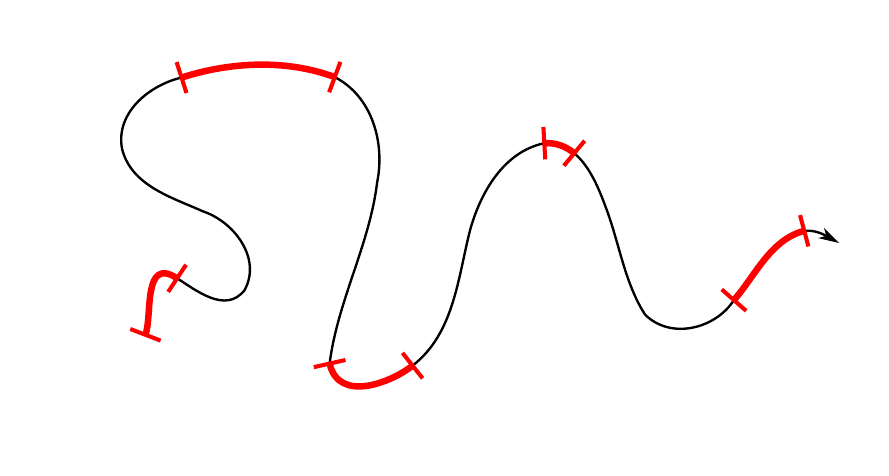}}%
	    \put(0.17994117,0.12802573){\color[rgb]{1,0,0}\makebox(0,0)[lt]{\lineheight{1.25}\smash{\begin{tabular}[t]{l}$\alpha_{\bj(t_0)}$\end{tabular}}}}%
	    \put(0.1213999,0.25634968){\color[rgb]{1,0,0}\makebox(0,0)[lt]{\lineheight{1.25}\smash{\begin{tabular}[t]{l}$\alpha_{\bj(t_0+ \delta^2)}$\end{tabular}}}}%
	    \put(0.17994117,0.41162911){\color[rgb]{1,0,0}\makebox(0,0)[lt]{\lineheight{1.25}\smash{\begin{tabular}[t]{l}$\alpha_{\bj(t_1)}$\end{tabular}}}}%
	    \put(0.36689199,0.47803874){\color[rgb]{1,0,0}\makebox(0,0)[lt]{\lineheight{1.25}\smash{\begin{tabular}[t]{l}$\alpha_{\bj(t_1+\delta^2)}$\end{tabular}}}}%
	    \put(0.32149587,0.0695622){\color[rgb]{1,0,0}\makebox(0,0)[lt]{\lineheight{1.25}\smash{\begin{tabular}[t]{l}$\alpha_{\bj(t_2)}$\end{tabular}}}}%
	    \put(0.48560335,0.11959345){\color[rgb]{1,0,0}\makebox(0,0)[lt]{\lineheight{1.25}\smash{\begin{tabular}[t]{l}$\alpha_{\bj(t_2+\delta^2)}$\end{tabular}}}}%
	    \put(0.54990278,0.40991467){\color[rgb]{1,0,0}\makebox(0,0)[lt]{\lineheight{1.25}\smash{\begin{tabular}[t]{l}$\alpha_{\bj(t_3)}$\end{tabular}}}}%
	    \put(0.67370316,0.35807713){\color[rgb]{1,0,0}\makebox(0,0)[lt]{\lineheight{1.25}\smash{\begin{tabular}[t]{l}$\alpha_{\bj(t_3+\delta^2)}$\end{tabular}}}}%
	    \put(0.81890373,0.15903731){\color[rgb]{1,0,0}\makebox(0,0)[lt]{\lineheight{1.25}\smash{\begin{tabular}[t]{l}$\alpha_{\bj(t_4)}$\end{tabular}}}}%
	    \put(0.88915418,0.23907521){\color[rgb]{1,0,0}\makebox(0,0)[lt]{\lineheight{1.25}\smash{\begin{tabular}[t]{l}$\alpha_{\bj(t_4 + \delta^2)}$\end{tabular}}}}%
	    \put(0.09670034,0.37594568){\color[rgb]{0,0,0}\makebox(0,0)[lt]{\lineheight{1.25}\smash{\begin{tabular}[t]{l}$\alpha$\end{tabular}}}}%
	    \put(0.73802818,0.05402654){\color[rgb]{0,0,0}\makebox(0,0)[lt]{\lineheight{1.25}\smash{\begin{tabular}[t]{l}$nX$\end{tabular}}}}%
	  \end{picture}%
	\endgroup%
	\caption{Edges in $\FastE(\sigma)$ are represented by the thick red lines.}
	\label{fig : monotonicity/HW_is_free/modification/GeneralPlan}
	\end{figure}

	\paragraph{Step 1: Building $\FastE$.} We claim that there exists a finite sequence $(0=t_0<t_1 <\dots <t_K\le D(\sigma))$ such that $D(\sigma) - t_K \le \frac{2b\delta}a$, and 
	\begin{alignat}{2}
		\label{eqn : LDP/wLDP/rec1}
		  &\forall 1 \le k \le K,\qquad  && %
		  \delta^2 \le t_{k}-t_{k-1} \le \frac{2b\delta}a, \\
		\label{eqn : LDP/wLDP/rec2}
		  &\forall 1 \le k \le K,\quad   \forall t\ge  t_{k},\qquad  &&%
		   \BoxPT^{(k)}\p{ {\alpha_{\bj(t_{k-1})}}, {\alpha_{\bj(t)}} } \ge n(t-t_{k-1})(1-\delta),\\
		\label{eqn : LDP/wLDP/rec3}
		  &\forall 1 \le k \le K,\qquad && %
		  \BoxPT^{(k)}\p{ {\alpha_{\bj(t_{k-1})}}, {\alpha_{\bj(t_k)}} } \le n(t_{k} - t_{k-1})(1+\delta),
    \end{alignat}
	where
  	\begin{equation}
  		\EPT[e]^{(k)} \dpe 
	    \begin{cases}
	        a &\text{ if } e \in \displaystyle \bigcup_{l=0}^{k-1}\Pathedges{\restriction\alpha{\intervalleff{\bj(t_{l})}{\bj(t_{l}+\delta^2) } } } \\
	        b &\text{ otherwise.}    
	    \end{cases}
  	\end{equation}
  	Indeed let us assume that $t_0,\dots, t_q$ have already been defined and satisfy~\eqref{eqn : LDP/wLDP/rec1}, \eqref{eqn : LDP/wLDP/rec2} and \eqref{eqn : LDP/wLDP/rec3} for all $1\le k \le q$. Notice that if $q=0$, we do not assume anything. If $D(\sigma)-t_q \le \frac{2b\delta}a$ then the construction is complete with $K\dpe q$. Otherwise two cases need to be considered: either the latest modification has made the passage time far below the target time, and we need to skip a suitable number of edges along $\alpha$ before the next modification, or the passage time is not far below the target time, and we resume modifying the edge passage times without any skip. 

  	\emph{Case 1: there exists $t\ge t_q + \delta^2$ such that}
  	\begin{equation}
  	\label{eqn : LDP/wLDP/condition_choix_jalon}
  		\BoxPT^{(q+1)}\p{{\alpha_{\bj(t_q)}} , {\alpha_{\bj(t)}}} < n(t-t_q)(1-\delta).
  	\end{equation}

  	Let
  	\begin{equation*}
  		t_{q+1} \dpe \sup \acc{t\ge t_q + \delta^2 \Biggm| \BoxPT^{(q+1)}\p{{\alpha_{\bj(t_q)}} , {\alpha_{\bj(t)}}} < n(t-t_q)(1-\delta) }.
  	\end{equation*}
  	We first show that~\eqref{eqn : LDP/wLDP/rec1} holds for $k=q+1$, i.e. $t_{q+1} \le t_q + \frac{2b\delta}a < D(\sigma)$. Indeed let $t>t_q + \frac{2b\delta}a$. Consider a $\BoxPT^{(q+1)}$-geodesic ${\alpha_{\bj(t_q)}} \Path{\gamma}{\alpha_{\bj(t)}}$. Let $z$ be its last point belonging to $\restriction\alpha{ \intervalleff{0}{\bj(t_q)} }$ (see Figure~\ref{fig : LDP/wLDP/BuildingEfastCase1}). There exists $0\le s \le t_q$ such that $z={\alpha_{\bj(s)}}$.
  	\begin{figure}
  	\def\svgwidth{0.95\textwidth}
	\begingroup%
	  \makeatletter%
	  \providecommand\color[2][]{%
	    \errmessage{(Inkscape) Color is used for the text in Inkscape, but the package 'color.sty' is not loaded}%
	    \renewcommand\color[2][]{}%
	  }%
	  \providecommand\transparent[1]{%
	    \errmessage{(Inkscape) Transparency is used (non-zero) for the text in Inkscape, but the package 'transparent.sty' is not loaded}%
	    \renewcommand\transparent[1]{}%
	  }%
	  \providecommand\rotatebox[2]{#2}%
	  \newcommand*\fsize{\dimexpr\f@size pt\relax}%
	  \newcommand*\lineheight[1]{\fontsize{\fsize}{#1\fsize}\selectfont}%
	  \ifx\svgwidth\undefined%
	    \setlength{\unitlength}{341.95210134bp}%
	    \ifx\svgscale\undefined%
	      \relax%
	    \else%
	      \setlength{\unitlength}{\unitlength * \real{\svgscale}}%
	    \fi%
	  \else%
	    \setlength{\unitlength}{\svgwidth}%
	  \fi%
	  \global\let\svgwidth\undefined%
	  \global\let\svgscale\undefined%
	  \makeatother%
	  \begin{picture}(1,0.20084103)%
	    \lineheight{1}%
	    \setlength\tabcolsep{0pt}%
	    \put(0,0){\includegraphics[width=\unitlength,page=1]{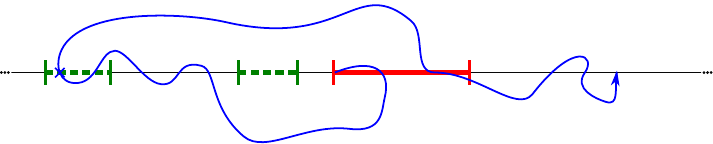}}%
	    \put(0.09252908,0.11383934){\color[rgb]{0,0,1}\makebox(0,0)[lt]{\lineheight{1.25}\smash{\begin{tabular}[t]{l}$z$\end{tabular}}}}%
	    \put(0.35503426,0.19169338){\color[rgb]{0,0,1}\makebox(0,0)[lt]{\lineheight{1.25}\smash{\begin{tabular}[t]{l}$\gamma$\\\end{tabular}}}}%
	    \put(0.95336526,0.07291732){\color[rgb]{0,0,1}\makebox(0,0)[lt]{\lineheight{1.25}\smash{\begin{tabular}[t]{l}$\alpha$\end{tabular}}}}%
	    \put(0.87038171,0.12989947){\color[rgb]{0,0,1}\makebox(0,0)[lt]{\lineheight{1.25}\smash{\begin{tabular}[t]{l}\textit{$\alpha_{\bj(t) }$}\end{tabular}}}}%
	    \put(0.14496248,0.0607555){\color[rgb]{0,0.50196078,0}\makebox(0,0)[lt]{\lineheight{1.25}\smash{\begin{tabular}[t]{l}$\alpha_{\bj(t_{q-2}+\delta^2) }$\end{tabular}}}}%
	    \put(0.32276124,0.0607555){\color[rgb]{0,0.50196078,0}\makebox(0,0)[lt]{\lineheight{1.25}\smash{\begin{tabular}[t]{l}$\alpha_{\bj(t_{q-1}) }$\end{tabular}}}}%
	    \put(0.40507549,0.0607555){\color[rgb]{0,0.50196078,0}\makebox(0,0)[lt]{\lineheight{1.25}\smash{\begin{tabular}[t]{l}$\alpha_{\bj(t_{q-1} + \delta^2 ) }$\end{tabular}}}}%
	    \put(0.05217187,0.0607555){\color[rgb]{0,0.50196078,0}\makebox(0,0)[lt]{\lineheight{1.25}\smash{\begin{tabular}[t]{l}$\alpha_{\bj(t_{q-2}) }$\end{tabular}}}}%
	    \put(0.45835525,0.13678392){\color[rgb]{1,0,0}\makebox(0,0)[lt]{\lineheight{1.25}\smash{\begin{tabular}[t]{l}$\alpha_{\bj(t_{q}) }$\end{tabular}}}}%
	    \put(0.64962362,0.13678392){\color[rgb]{1,0,0}\makebox(0,0)[lt]{\lineheight{1.25}\smash{\begin{tabular}[t]{l}$\alpha_{\bj(t_{q}+\delta^2) }$\end{tabular}}}}%
	  \end{picture}%
	\endgroup%
  	\caption{For the sake of simplicity, $\alpha$ is represented as a straight line. The newly modified edges are represented by the thick, red, solid line while the other modified edges are represented by the thick, green, dashed lines. Edges used by the portion of $\gamma$ (blue curved line) lying between $z$ and $\alpha_{\bj(t)}$ are either unmodified or newly modified.}
  	\label{fig : LDP/wLDP/BuildingEfastCase1}	
  	\end{figure}
  	Inequality~\eqref{eqn : LDP/wLDP/approx_path_csq1} implies that for large enough $n$,
  	\begin{equation}
  		\bj(t_q + \delta^2) - \bj(t_q) \le \frac{n\delta^2}{a} + d,
  	\end{equation}
  	therefore the subpath of $\gamma$ whose endpoints are $z = \alpha_{\bj(s)}$ and $\alpha_{\bj(t)}$ uses at most $\frac{n\delta^2}{a} + d$ edges $e$ such that $\EPT[e]^{(q+1)} < b$. Consequently, for large enough $n$,
  	\begin{align}
  		\BoxPT^{(q+1)}\p{\alpha_{\bj(t_q)} , \alpha_{\bj(t)}} %
  			&\ge \BoxPT^{(q+1)}\p{\alpha_{\bj(s)} , \alpha_{\bj(t)}}\eol
  			&\ge b\norme{ \alpha_{\bj(s)} - \alpha_{\bj(t)}} - \p{ \frac{\delta^2n}{a} +d}b \\
  		\label{eqn : LDP/wLDP/borne_tp+1_1}
            &\ge nD\p{ \frac{\alpha_{\bj(s)} }n  ,  \frac{\alpha_{\bj(t)} }n } - \frac{2b\delta^2n}{a}. \nonumber
        \intertext{Applying \eqref{eqn : LDP/wLDP/approx_path1} and using the triangle inequality for $D$, we get}
        \BoxPT^{(q+1)}\p{\alpha_{\bj(t_q)} , \alpha_{\bj(t)}} %
            &\ge nD\p{ \sigma(s)   ,  \sigma(t)  } - \p{ \frac{\delta^2n}{a} + 3d  +2}b.  \nonumber
        \intertext{Since $\sigma$ is a $D$-geodesic, we get for large enough $n$,}
        \BoxPT^{(q+1)}\p{\alpha_{\bj(t_q)} , \alpha_{\bj(t)}} %
            &\ge n(t-s) -  \frac{2b\delta^2n}{a} \eol
            &\ge n(t-t_q) -  \frac{2b\delta^2n}{a}.
    \end{align}
  	Consequently, for all $t\ge t_q + \frac{2b\delta}a$, inequality~\eqref{eqn : LDP/wLDP/condition_choix_jalon} fails, thus~\eqref{eqn : LDP/wLDP/rec1} holds.

  	Inequality~\eqref{eqn : LDP/wLDP/rec2} for $k=q+1$ is a direct consequence of $\bj$'s right-continuity and the fact that for all $t\in\intervalleof{t_{q+1}}{D(\sigma)}$,
  	\begin{equation*}
  		\BoxPT^{(q+1)}\p{{\alpha_{\bj(t_q)}} , {\alpha_{\bj(t)}}} \ge n(t-t_q)(1-\delta).
  	\end{equation*} To show inequality~\eqref{eqn : LDP/wLDP/rec3} for $k=q+1$, it suffices to notice that
  	\begin{align*}
  	 	\BoxPT^{(q+1)}\p{{\alpha_{\bj(t_{q}) }} , {\alpha_{\bj(t_{q+1}) }} } %
  	 		&\le \BoxPT^{(q+1)}\p{{\alpha_{\bj(t_{q}) }} , {\alpha_{\bj(t_{q+1}^-) }}} + b\\
  	 		&\le n(t_{q+1} - t_q)(1-\delta) + b \\
  	 		&\le n(t_{q+1}-t_q)
  	\end{align*}
  	for large enough $n$.

  	\emph{Case 2: for all $t\ge t_q + \delta^2$, inequality~\eqref{eqn : LDP/wLDP/condition_choix_jalon} fails.} We define ${t_{q+1}\dpe t_q+ \delta^2}$. Only inequality~\eqref{eqn : LDP/wLDP/rec3} is non trivial. Edges along $\alpha\intervalleff{\bj(t_{q})}{\bj(t_{q+1})}$ have passage time $a$ in the configuration $\EPT^{(q+1)}$, therefore
  	\begin{align*}
  		\BoxPT^{(q+1)}\p{ {\alpha_{\bj(t_{q})}}, {\alpha_{\bj(t_{q+1})}} }%
  			&\le a\p{\bj(t_{q+1}) - \bj(t_{q})}.
  		\intertext{Applying~\eqref{eqn : LDP/wLDP/approx_path_csq1} we get}
  		\BoxPT^{(q+1)}\p{ {\alpha_{\bj(t_{q})}}, {\alpha_{\bj(t_{q+1})}} }%
  			&\le a \cdot \p{ \frac{n}a(t_{q+1} - t_q) +d } \\
  			&= n(t_{q+1} - t_q) + {ad}\\
  			&\le n(t_{q+1}-t_q)(1+\delta)
  	\end{align*}
  	for large enough $n$, thus~\eqref{eqn : LDP/wLDP/rec3} is proved for $k=q+1$.

  	We define
  	\begin{equation}
  	\label{eqn : LDP/wLDP/def_FastE}
  		\FastE(\sigma) \dpe \bigcup_{k=0}^{K-1} \Pathedges{\restriction\alpha{\intervalleff{\bj(t_k)}{\bj(t_k + \delta^2)} } }.
  	\end{equation}
  	The inequality \eqref{eqn : LDP/wLDP/FastESimple/Cardinal} is a consequence of \eqref{eqn : LDP/wLDP/approx_path_csq2}. The integer $K$ is the number of steps in the construction. By the lower bound in~\eqref{eqn : LDP/wLDP/rec1},
  	\begin{equation}
  		\label{eqn : LDP/wLDP/UB_K}
  		K \le \frac{D(\sigma)}{\delta^2} \le \frac{bd}{\delta^2} < \infty.
  	\end{equation}
    
  	\paragraph{Step 2: Estimating the passage time between two milestones.}

  	We claim that for all $0\le k_1 < k_2 \le K$,
  	\begin{equation}
  	\label{eqn : LDP/wLDP/encadrement_jalons_final}
  		n(t_{k_2} - t_{k_1}) - \Cl{LB_jalons}\delta n%
  			\le \BoxPT^{(K)}\p{ \alpha_{\bj(t_{k_1})}, \alpha_{\bj(t_{k_2})} }%
  			\le n(t_{k_2} - t_{k_1}) + \Cr{LB_jalons}\delta n,
  	\end{equation}
  	where $\Cr{LB_jalons}$ is a constant depending only on $a,b$ and $d$.

  	Let $0\le k_1 < k_2 \le K$. The triangle inequality for $\BoxPT^{(K)}$ and the inequality $\BoxPT^{(K)} \le \BoxPT^{(k)}$ for all $0\le k \le K$ yield
  	\begin{align}
  		\BoxPT^{(K)}\p{ \alpha_{\bj(t_{k_1})}, \alpha_{\bj(t_{k_2})} }%
  			&\le \sum_{k=k_1 +1}^{k_2} \BoxPT^{(K)}\p{ \alpha_{\bj(t_{k-1})}, \alpha_{\bj(t_{k})} }\eol
  			&\le \sum_{k=k_1 +1}^{k_2} \BoxPT^{(k)}\p{ \alpha_{\bj(t_{k-1})}, \alpha_{\bj(t_{k})} }.\notag
        \intertext{Applying inequality~\eqref{eqn : LDP/wLDP/rec3} gives}
        \BoxPT^{(K)}\p{ \alpha_{\bj(t_{k_1})}, \alpha_{\bj(t_{k_2})} }%
            &\le (1+\delta) n (t_{k_2} - t_{k_1}).
  	\end{align}
  	Since $t_{k_2} - t_{k_1} \le D(\sigma) \le bd$, this gives the upper bound in~\eqref{eqn : LDP/wLDP/encadrement_jalons_final}.

  	We now turn to the lower bound. Consider a $\BoxPT^{(K)}$-geodesic $\alpha_{\bj{(t_{k_1})} }  \Path{\gamma} \alpha_{\bj{(t_{k_2})} } $. Let $ \alpha_{\bj(t)}$ denote the first point of $\gamma$ belonging to %
  	\[\bigcup_{l=k_1+1}^{K} {\restriction\alpha{\intervalleff{\bj(t_l)}{\bj(t_l + \delta^2)} }}. \]%
  	Note that $ \alpha_{\bj(t)}$ always exists, since $\gamma$ takes the value $\alpha_{\bj{(t_{k_2})} } $. The edges used by $\gamma$ before reaching $ \alpha_{\bj(t)}$ have the same passage time in configurations $\EPT^{(k_1+1)}$ and $\EPT^{(K)}$, therefore
  	\begin{align}
  		\BoxPT^{(K)}\p{\alpha_{\bj{(t_{k_1})} }  ,  \alpha_{\bj(t)} } %
  			&= \BoxPT^{(k_1+1)}\p{\alpha_{\bj{(t_{k_1})} }  ,  \alpha_{\bj(t)} }.\notag
  		\intertext{Applying~\eqref{eqn : LDP/wLDP/rec2}, we get}
  		\label{eqn : LDP/wLDP/lower_bound_k1_z}
  		\BoxPT^{(K)}\p{ \alpha_{\bj{(t_{k_1})} }  , \alpha_{\bj(t)} }%
  			&\ge n(t-t_{k_1})(1-\delta).
  	\end{align}
  	In the easier case where $t\ge t_{k_2}$, inequality~\eqref{eqn : LDP/wLDP/lower_bound_k1_z} implies
  	\begin{equation}
  	\label{eqn : LDP/wLDP/LB_jalons_cas_facile}
  		\BoxPT^{(K)}\p{ \alpha_{\bj(t_{k_1})},  \alpha_{\bj(t_{k_2})} }%
  			\ge n( t_{k_2} - t_{k_1})(1-\delta),
  	\end{equation}
  	which is stronger than the desired bound. Otherwise there exists $k_1 +1 \le k \le k_2-1$ such that $t_k \le t \le t_k + \delta^2$. Every edge along $\restriction\alpha{\intervalleff{\bj(t_k)}{\bj(t)} }$ has a passage time $a$ in configuration $\EPT^{(K)}$, thus inequality~\eqref{eqn : LDP/wLDP/approx_path_csq1} yields
  	\begin{align}
  		\BoxPT^{(K)}\p{ \alpha_{\bj(t_k)}   ,  \alpha_{\bj(t)}  }%
  			&\le a \p{\bj(t) - \bj(t_k)}\eol
  			&\le n(t-t_k) + ad.
  		\label{eqn : LDP/wLDP/LB_jalons_cas_dur1}
  	\end{align}
  	for large enough $n$. Since $\gamma$ is a $\BoxPT^{(K)}$-geodesic,
  	\begin{align}
  		\BoxPT^{(K)}\p{\alpha_{\bj(t_{k_1})} ,  \alpha_{\bj(t_{k_2})}  }%
  			&= \BoxPT^{(K)}\p{ \alpha_{\bj(t_{k_1})}   , \alpha_{\bj(t)}} + \BoxPT^{(K)}\p{\alpha_{\bj(t)} , \alpha_{\bj(t_{k_2})}  }.
  		\intertext{Applying the triangle inequality yields}
  		\label{eqn : LDP/wLDP/LB_jalons_cas_dur2}
  		\BoxPT^{(K)}\p{\alpha_{\bj(t_{k_1})}  , {\alpha_{\bj(t_{k_2})} } }%
  			&\ge \BoxPT^{(K)}\p{\alpha_{\bj(t_{k_1})}  , {\alpha_{\bj(t)} } }%
  				+\BoxPT^{(K)}\p{\alpha_{\bj(t_k)}  , {\alpha_{\bj(t_{k_2})} }} %
        - \BoxPT^{(K)}\p{{\alpha_{\bj(t_k)}  , {\alpha_{\bj(t)} } }}.
  		\intertext{In the right-hand side of~\eqref{eqn : LDP/wLDP/LB_jalons_cas_dur2}, lower bounding the first term with~\eqref{eqn : LDP/wLDP/rec2} and the third with~\eqref{eqn : LDP/wLDP/LB_jalons_cas_dur1} gives}
  		\BoxPT^{(K)}\p{ \alpha_{\bj(t_{k_1})}  , \alpha_{\bj(t_{k_2})} }%
  			&\ge (1-\delta)n(t-t_{k_1})%
  				+\BoxPT^{(K)}\p{ \alpha_{\bj(t_k)}  , \alpha_{\bj(t_{k_2})} }%
  				-n(t-t_k) - ad\eol
  			&= \BoxPT^{(K)}\p{ \alpha_{\bj(t_k)}  , \alpha_{\bj(t_{k_2})} } +%
  				(1-\delta)n(t_k - t_{k_1}) - \delta n(t-t_k) - ad.\notag
  		\intertext{As $t-t_k \le \delta^2$, for large enough $n$,}
  		\BoxPT^{(K)}\p{ \alpha_{\bj(t_{k_1})}  , \alpha_{\bj(t_{k_2})} }%
  			&\ge\BoxPT^{(K)}\p{ \alpha_{\bj(t_k)}  , \alpha_{\bj(t_{k_2})} } +%
  				(1-\delta)n(t_k - t_{k_1}) - 2n\delta^3.
  		\label{eqn : LDP/wLDP/LB_jalons_cas_dur3}
  	\end{align}
  	It remains to get a lower bound on $\BoxPT^{(K)}\p{ \alpha_{\bj(t_k)}  , \alpha_{\bj(t_{k_2})} }$. Using recursively the same arguments one gets a lower bound analogous to~\eqref{eqn : LDP/wLDP/LB_jalons_cas_dur3} or (in at most $K$ steps) finally analogous to~\eqref{eqn : LDP/wLDP/LB_jalons_cas_facile}. This leads to
  	\begin{equation*}
  		\BoxPT^{(K)}\p{ \alpha_{\bj(t_{k_1})}  , \alpha_{\bj(t_{k_2})} }%
  			\ge (1-\delta)n(t_{k_2} - t_{k_1}) - 2n K\delta^3.
  	\end{equation*}
  	Applying~\eqref{eqn : LDP/wLDP/UB_K}, $t_{k_2} - t_{k_1} \le D(\sigma) \le bd$, we get the lower bound in~\eqref{eqn : LDP/wLDP/encadrement_jalons_final}.

  	\paragraph{Step 3: Extending~\eqref{eqn : LDP/wLDP/encadrement_jalons_final} beyond milestones.} We show that for all $x,y \in \intint0n^d$,
	\begin{equation}
	\label{eqn : LDP/wLDP/encadrement_final}
    n \HW D\sigma \p{\frac xn,\frac yn} - \Cl{HW_final}n\delta \le \BoxPT^{(K)}(x,y) \le  n \HW D\sigma \p{\frac xn,\frac yn}+ \Cr{HW_final}n \delta,
	\end{equation}
	where $\Cr{HW_final}$ may only depend on $a,b$ and $d$.
	
	Let $x,y \in \intint0n^d$. We start with the lower bound. Let $x\Path{\gamma}y$ be a discrete $\BoxPT^{(K)}$-geodesic. If $\gamma$ does not intersect $\FastE$, then $\tau^{(K)}(\gamma) = b\norme{\gamma}$ thus the lower bound in~\eqref{eqn : LDP/wLDP/encadrement_final} is proven. Otherwise $\gamma$ can be decomposed as
	\begin{equation*}
		x \Path{\Gag} \alpha_{ \bj(s_\rmg) } \Path{\Gam} \alpha_{ \bj(s_\rd) } \Path{\Gad} y,
	\end{equation*}
	where $\alpha_{ \bj(s_\rmg) }$ and $\alpha_{ \bj(s_\rd) }$ are respectively the first and last visit of $\gamma$ in $\FastE$. We have
	\begin{align}
		\tau^{(K)}(\gamma) &= \tau^{(K)}(\Gag)+\tau^{(K)}(\Gam)+\tau^{(K)}(\Gad) \eol
			&= b \norme\Gag+\tau^{(K)}(\Gam)+b \norme\Gad,\notag
		\intertext{thus}
		\label{eqn : LDP/wLDP/LB_gamma_decompose}
		\BoxPT^{(K)}(x,y) &\ge  n \HW D\sigma\p{\frac{x}{n}, \frac{\alpha_{ \bj(s_\rmg) }}{n} } + \BoxPT^{(K)}(\alpha_{ \bj(s_\rmg) }, \alpha_{ \bj(s_\rd) }) + n \HW D\sigma\p{\frac{\alpha_{ \bj(s_\rd) }}{n}, \frac{y}{n} }.
	\end{align}
	By the definition of $s_\rmg$ there exists $t_\rmg\in \acc{t_k, 0\le k \le K-1 }$ such that $t_\rmg \le s_\rmg \le t_\rmg +\delta^2$. Therefore~\eqref{eqn : LDP/wLDP/approx_path_csq1} yields
	\begin{align}
	\label{eqn : LDP/wLDP/zg_VS_tg1}
		\norme{\alpha_{ \bj(s_\rmg) } - \alpha_{\bj(t_\rmg)}  } &\le \frac{\delta^2 n}{a} +d.
	\intertext{Applying~\eqref{eqn : LDP/wLDP/approx_path1} gives }
	\label{eqn : LDP/wLDP/zg_VS_tg2}
		\norme{\alpha_{ \bj(s_\rmg) } - n\sigma(t_\rmg) } &\le \frac{\delta^2 n}{a} + 2d+ 1.
	\end{align}
	We define $t_\rd$ similarly and $\alpha_{ \bj(s_\rd) }$ satisfies analogous inequalities. Triangle inequality and~\eqref{eqn : LDP/wLDP/zg_VS_tg1} yield
	\begin{align}
		\BoxPT^{(K)}(\alpha_{ \bj(s_\rmg) }, \alpha_{ \bj(s_\rd) }) %
			&\ge \BoxPT^{(K)}\p{ \alpha_{\bj(t_\rmg)}  ,\alpha_{\bj(t_\rd)}  }%
					-2b\p{ \frac{\delta^2 n}{a} +d }.\notag
		\intertext{Applying~\eqref{eqn : LDP/wLDP/encadrement_jalons_final}, we get}
		\label{eqn : LDP/wLDP/LB_gamma_2.1}
		\BoxPT^{(K)}(\alpha_{ \bj(s_\rmg) }, \alpha_{ \bj(s_\rd) }) %
			&\ge n\module{t_\rmg  - t_\rd} -2b\p{ \frac{\delta^2 n}{a} +d } - n\Cr{LB_jalons}\delta.
		\intertext{Besides, the triangle inequality and \eqref{eqn : LDP/wLDP/zg_VS_tg2} yield}
		\HW D\sigma\p{ \frac{\alpha_{ \bj(s_\rmg) }}n, \frac{\alpha_{ \bj(s_\rd) }}n } &\le \HW D\sigma\p{\sigma(t_\rmg) , \sigma(t_\rd)} + 2b\p{\frac{\delta^2 }{a} + \frac{2d+ 1}n }.\notag
		\intertext{By the definition of $\HW D\sigma$ (see \eqref{eqn : GeoMetrics/HW/HW}), $\HW D\sigma\p{\sigma(t_\rmg) , \sigma(t_\rd)} = D\p{\sigma(t_\rmg) , \sigma(t_\rd)}$. Moreover, $\sigma$ is a $D$-geodesic, therefore}
		\label{eqn : LDP/wLDP/LB_gamma_2.2}
		\HW D\sigma\p{ \frac{\alpha_{ \bj(s_\rmg) }}n, \frac{\alpha_{ \bj(s_\rd) }}n } &\le \module{t_\rmg  - t_\rd} + 2b\p{\frac{\delta^2 }{a} + \frac{2d+ 1}n }.
		\intertext{Combining~\eqref{eqn : LDP/wLDP/LB_gamma_2.1} and~\eqref{eqn : LDP/wLDP/LB_gamma_2.2}, we get}
		\label{eqn : LDP/wLDP/LB_gamma_2}
		\BoxPT^{(K)}(\alpha_{ \bj(s_\rmg) }, \alpha_{ \bj(s_\rd) }) %
			&\ge n \HW D\sigma\p{ \frac{\alpha_{ \bj(s_\rmg) }}n, \frac{\alpha_{ \bj(s_\rd) }}n }  - 4b\p{\frac{\delta^2 n}{a} + 2d+ 1 } - \Cr{LB_jalons}n\delta.
	\end{align}
	Inequality~\eqref{eqn : LDP/wLDP/LB_gamma_2} gives a lower bound for the second term in~\eqref{eqn : LDP/wLDP/LB_gamma_decompose}, leading to the lower bound in~\eqref{eqn : LDP/wLDP/encadrement_final}.

	We now turn to the upper bound in~\eqref{eqn : LDP/wLDP/encadrement_final}. Recall the definition of $\HW D\sigma$ given in~\eqref{eqn : GeoMetrics/HW/HW}. If $\HW D\sigma\p{ \frac xn, \frac yn}= \frac bn \norme{x-y}$ then the claim is straightforward. Otherwise, there exist $0 \le s_\rmg, s_\rd \le D(\sigma)$ such that \[ \HW D\sigma\p{ \frac xn, \frac yn} = b\norme{\frac xn- \sigma(s_\rmg)} + D\p{ \sigma(s_\rmg), \sigma(s_\rd) } + b\norme{ \sigma(s_\rd) - \frac yn}. \] We have
	\begin{align}
		\BoxPT^{(K)}(x,y) &\le \BoxPT^{(K)}\p{x, \alpha_{\bj(s_\rmg)} }+\BoxPT^{(K)}\p{ \alpha_{\bj(s_\rmg)} ,  \alpha_{\bj(s_\rd)} } +\BoxPT^{(K)}\p{ \alpha_{\bj(s_\rd)}, y } \eol
			&\le b\norme{x - \alpha_{\bj(s_\rmg)}}+\BoxPT^{(K)}\p{ \alpha_{\bj(s_\rmg)}, \alpha_{\bj(s_\rd)} } +b\norme{ \alpha_{\bj(s_\rd)} - y } \nonumber.
        \intertext{Using \eqref{eqn : LDP/wLDP/approx_path1} and the triangle inequality, we get}
		\label{eqn : LDP/wLDP/UB_decompose}
		\BoxPT^{(K)}(x,y) &\le  nD \p{ \frac xn , \frac{ \alpha_{\bj(s_\rmg)}}n } + \BoxPT^{(K)}\p{ \alpha_{\bj(s_\rmg)}, \alpha_{\bj(s_\rd)} } + nD \p{\frac{ \alpha_{\bj(s_\rd)}}n , \frac yn } + \grando(1).
	\end{align}
	The upper bound in~\eqref{eqn : LDP/wLDP/rec1}, along with $D(\sigma) - t_K \le \frac{2b\delta}a$, implies the existence of $t_\rmg, t_\rd \in \acc{t_k, k\in\intint0K} $ such that $\module{t_\rmg - s_\rmg}, \module{t_\rd - s_\rd} \le \frac{2b\delta}a$. Proceeding as for~\eqref{eqn : LDP/wLDP/LB_gamma_2}, we show that the second term in~\eqref{eqn : LDP/wLDP/UB_decompose} is upper bounded by $n D\p{ \frac{\alpha_{\bj(s_\rmg)}}n , \frac{\alpha_{\bj(s_\rd)}}n } + \grando(n\delta) = n \HW D\sigma\p{ \frac{\alpha_{\bj(s_\rmg)}}n , \frac{\alpha_{\bj(s_\rd)}}n } + \grando(n\delta)$, hence the upper bound in~\eqref{eqn : LDP/wLDP/encadrement_final} is proven.
\end{proof}

\begin{proof}[Proof of Lemma \ref{lem : LDP/wLDP/FastE}]
    There exists $\eta>0$ such that for all $1\le k_1 < k_2 \le K$,
    \begin{equation}
        \d( \sigma_{k_1}, \sigma_{k_2} ) \ge 2\eta.
    \end{equation}
    Fix $n \ge 1$ large enough so that for all $k\in \intint1K$, there exists a set of edges $\FastE(\sigma^k)$ satisfying the conclusion of Lemma \ref{lem : LDP/wLDP/FastESimple}. Define
    \begin{equation}
        \FastE \dpe \bigcup_{k=1}^K \FastE(\sigma^k).
    \end{equation}
    The inequality \eqref{eqn : LDP/wLDP/FastE/Localisation} is implied by the inequalities \eqref{eqn : LDP/wLDP/FastESimple/Localisation}, for the geodesics $\sigma^k$. Moreover, note that for large enough $n$, for all $1\le k_1 < k_2 \le K$,
    \begin{equation}
        \label{eqn : LDP/wLDP/FastE/SetFarFromEachOther}
        \d \p{ \FastE(\sigma_{k_1} ) , \FastE(\sigma_{k_2} )} \ge \eta n.
    \end{equation}

    For all $k \in \intint1K$ and $e\in \edges{\intint0n^d}$, we define
        \begin{equation}
            \tau_e^{[k]} \dpe a \ind{e\in \FastE(\sigma^k)} + b \ind{e\notin \FastE(\sigma^k)},
        \end{equation}
    and we denote by $\BoxPT^{[k]}$ the associated pseudometric on $\intint0n^d$. We have
    \begin{equation}
        \label{eqn : LDP/wLDP/FastE/InputSimple}
        \forall k\in\intint1K,\; \forall x,y\in\intint0n^d, \quad \module{\BoxPT^{[k]}(x,y) - n \HW D{\sigma^k}\p{ \frac xn,\frac yn } }\le n \delta.
    \end{equation}
    For all $e\in \edges{\intint0n^d}$, we define
        \begin{equation}
            \tau_e^{*} \dpe a \ind{e\in \FastE} + b \ind{e\notin \FastE},
        \end{equation}
    and we denote by $\BoxPT^{*}$ the associated pseudometric on $\intint0n^d$.
    
    Fix $x,y \in \intint0n^d$. We claim that
    \begin{equation}\label{eqn : LDP/wLDP/FastE/LowerBound}
        \BoxPT^*(x,y) \ge n \HW D{(\sigma^k)_{1\le k \le K}} \p{ \frac xn,\frac yn } - \p{\frac{bd}{a \eta} + 1} n\delta
    \end{equation}
    and
    \begin{equation}\label{eqn : LDP/wLDP/FastE/UpperBound}
        \BoxPT^*(x,y) \le n \HW D{(\sigma^k)_{1\le k \le K}} \p{ \frac xn,\frac yn } +  2K n\delta,
    \end{equation}
    which gives the lemma.
    
    Let us lower bound $\BoxPT^*(x,y)$. Consider a $\BoxPT^*$-geodesic $\gamma$ from $x$ to $y$; it can be decomposed as
    \begin{equation}
        x = x_0 \Path{\gamma_1} x_1 \Path{\gamma_2} x_2 \Path{\gamma_3} \dots \Path{\gamma_R} x_R = y,
    \end{equation}
    where for all $r \in \intint1R$, \[\Pathedges{\gamma_r} \subseteq \edges{\intint0n^d} \setminus \bigcup_{\substack{k=1 \\ k\neq k(r)}}^K  \FastE(\sigma^k), \] for some $k(r) \in \intint1K$. By \eqref{eqn : LDP/wLDP/FastE/SetFarFromEachOther}, the decomposition can be chosen so that for all $r\in \intint1{R-1}$, $\norme{\gamma_r} \ge n\eta$. Since $\norme\gamma \le \frac{bdn}{a}$ by Lemma \ref{lem : GeoMetrics/Geodesics/localisation_discrete}, we have
    \begin{equation}
        R \le \frac{bd}{a \eta} + 1.
    \end{equation}
    Let $r  \in \intint1R$. By \eqref{eqn : LDP/wLDP/FastE/InputSimple},
    \begin{align*}
        \tau^*(\gamma_r)%
            &= \tau^{[k(r)]}(\gamma_r) \\
            &\ge  n \HW D{\sigma_{k(r)}}\p{ \frac {x_{r-1}}n,\frac {x_r}n } - n\delta\\
            &\ge n \HW D{ \p{\sigma^k}_{1\le k \le K} }\p{ \frac {x_{r-1}}n,\frac {x_r}n } - n\delta.
    \end{align*}
    Summing over $r$ and applying the triangle inequality gives \eqref{eqn : LDP/wLDP/FastE/LowerBound}.

    We now turn to the upper bound. By the definition of $\HW D{ \p{\sigma^k}_{1\le k \le K}  }$, there exist an integer $S \le K$, a sequence $\p{ \frac xn = x_0, x_1, \dots, x_S = \frac yn }$ with values in $\intervalleff01^d$ and a sequence $\p{k(1),\dots, k(S) }$ with values in $\intint1K$ such that
    \[
        \HW D{ \p{\sigma^k}_{1\le k \le K}  }\p{\frac xn, \frac yn} = \sum_{s=1}^S \HW D{ \sigma_{k(s)} }\p{ x_{s-1}, x_s}.
    \]
    Let $s\in \intint1S$. By \eqref{eqn : LDP/wLDP/FastE/InputSimple},
    \begin{align*}
        \BoxPT^*\p{ \floor{n x_{s-1}}, \floor{nx_s} }%
            &\le \BoxPT^{[k(s)] }\p{ \floor{n x_{s-1}}, \floor{nx_s} }\\
            &\le n \HW D{\sigma_{k(s)}}\p{ \frac{ \floor{n x_{s-1} } }n,\frac{ \floor{n x_s} }n } + n \delta.
        \intertext{By the triangle inequality, for large enough $n$,}
        \BoxPT^*\p{ \floor{n x_{s-1}}, \floor{nx_s} }%
            &\le n \HW D{ \sigma_{k(s)} }\p{ x_{s-1}, x_s} + 2n \delta.
    \end{align*}
    Summing over $s$, we get \eqref{eqn : LDP/wLDP/FastE/UpperBound}.
\end{proof}

\section{Proof of the corollaries}
\label{sec : App}

Before proceeding with the applications, we give preliminary general tools in Sections \ref{subsec : App/OtherBox} and \ref{subsec : App/Contraction}. First, some applications naturally involve the random metric restricted to the box $\intint{-Cn}{Cn}^d$, for some integer $C\ge 1$, rather than $\intint0n^d$, so we need to adapt Theorems \ref{thm : MonLDP} and \ref{thm : LDP}. By stationarity, this is only a matter of notations. Second, we state the contraction principles suitable in our context.

\subsection{Translating the observation window}
\label{subsec : App/OtherBox}

For all integers $n,C \ge 1$, we denote by $\CPT[\intervalleff{-Cn}{Cn}^d]$ the pseudometric on $\intervalleff{-Cn}{Cn}^d$ defined as $\CPT$ (see \eqref{eqn : Intro/Framework/BoxPT}). We then define
\begin{equation}\begin{split}
  \label{eqn : Intro/Framework/SPT_OtherBox}
  \SPT[C,n] : \intervalleff{-C}C^d \times \intervalleff{-C}C^d &\longrightarrow \intervallefo0\infty \\
  (x,y) &\longmapsto \frac1n \CPT[\intervalleff{-Cn}{Cn}^d ](nx,ny).
\end{split}\end{equation}
Given $D\in \Metrics\p{\intervalleff{-C}{C}^d }$, we define the rate functions $\MonRFinf(D)$, $\MonRFsup(D)$, $\MonRF(D)$, $\RFinf(D)$ as for elements of $\Metrics$. By stationarity of the model, Theorems \ref{thm : MonLDP} and  \ref{thm : LDP} imply Corollaries \ref{cor : App/MonLDP_OtherBox} and \ref{cor : App/LDP_OtherBox}.

\begin{Corollary}
  \label{cor : App/MonLDP_OtherBox}
  Assume that $b<\infty$. Let $C\ge 1$ be an integer. Then for all $D\in\Metrics\p{\intervalleff{-C}{C}^d }$, $\MonRFinf(D) = \MonRFsup(D) \eqqcolon \MonRF(D)$. Moreover:
  \begin{enumerate}[(i)]
    \item For all $D\in\Metrics\p{\intervalleff{-C}{C}^d }$,
    \begin{equation}
        \label{eqn : App/MonLDP_OtherBox/Expression}
      \MonRF(D) = \int_{\intervalleff{-C}{C}^d} \MonRFelem\p{ (\grad D)_z } \d z.   
    \end{equation}
    \item For all $D\in\Metrics\p{\intervalleff{-C}{C}^d }$, $\MonRF(D) = 0$ if and only if for all $x,y\in \intervalleff{-C}{C}^d$, $D(x,y) \le \mu(x-y)$.
  \end{enumerate}
\end{Corollary}

\begin{Corollary}
  \label{cor : App/LDP_OtherBox}
  Assume that $0<a<b<\infty$. Let $C\ge 1$ be an integer. The process $\p{\SPT[C,n]}_{n \ge 1} $ follows the large deviation principle at speed $n^d$, with the rate function \[ \begin{split} \RF : \Metrics\p{ \intervalleff{-C}{C}^d  } &\longrightarrow \intervalleff0\infty \\
    D &\longmapsto \begin{cases}
      \MonRF(D) &\text{ if } D\in \GeoMetrics\p{\intervalleff{-C}{C}^d }, \\
      \infty &\text{ otherwise. } \end{cases} \end{split}\]
\end{Corollary}


\subsection{Contraction principles}
\label{subsec : App/Contraction}

Lemma~\ref{lem : App/Contraction} gives a LDP for any image of $(\SPT[C,n])_{n\ge1}$ through a continuous map. It is a special case of Theorem 4.2.1 in \cite{LDTA}.
\begin{Lemma}[Classic contraction principle]
\label{lem : App/Contraction}
	Assume that $0<a<b<\infty$. Let $C\ge1$ be an integer. Then for all Hausdorff topological spaces $\cY$ and continuous maps $\Phi:\Metrics\p{ \intervalleff{-C}{C}^d } \rightarrow \cY$, the process $\p{ \Phi(\SPT[C,n]) }_{n\ge 1}$ follows the large deviation principle at speed $n^d$ with the good rate function
	\begin{align}
		\cY &\longrightarrow \intervalleff0\infty \eol
		\label{eqn : App/Contraction}
        y &\longmapsto \inf_{\substack{D \in \Metrics\p{\intervalleff{-C}{C}^d} : \\ \Phi(D)=y} } \RF(D) = \inf_{\substack{D \in \GeoMetrics\p{\intervalleff{-C}{C}^d} : \\ \Phi(D)=y} } \int_{\intervalleff{-C}{C}^d} \MonRFelem\p{ \p{\grad D}_z} \d z.
	\end{align}
\end{Lemma}
We actually only use Lemma \ref{lem : App/Contraction} for Corollary \ref{cor : Intro/applications/ball}. The other applications being stated with milder assumptions on $\nu$, we use a somewhat different approach. First we obtain an estimate of $\Pb{ \Phi(\SPT) \gtrsim \zeta}$, with $\Phi$ depending on the studied process, thanks to Lemma \ref{lem : App/MonContraction}, which is essentially a one-sided version of Lemma \ref{lem : App/Contraction}.  Given a map \[\begin{split} \Phi:\Metrics &\longrightarrow \R^k \\ D &\longmapsto\p{ \Phi_1(D),\dots, \Phi_k(D) },\end {split} \] and $\zeta = (\zeta_1,\dots, \zeta_k)\in \R^k$, we define \begin{equation}\begin{split}
  \MonRFinfPhi(\zeta) &\dpe \lim_{\eps \to 0} \liminf_{n\to\infty}-\frac{1}{n^d}\log \Pb{ \forall i\in\intint1k,\quad \Phi_i(\SPT) \ge \zeta_i -\eps } \\ \text{  and  }%
  \MonRFsupPhi(\zeta) &\dpe \lim_{\eps \to 0} \limsup_{n\to\infty}-\frac{1}{n^d}\log \Pb{ \forall i\in\intint1k,\quad \Phi_i(\SPT) \ge \zeta_i -\eps }.
\end{split}\end{equation}
\begin{Lemma}[One-sided contraction principle]
\label{lem : App/MonContraction}
	Assume that $b<\infty$. Let $\Phi:\Metrics \longrightarrow \R^k$ be a continuous and nondecreasing function. Then for all $\zeta \in \R^k$,
    \begin{equation}
        \label{eqn : App/MonContraction}
        \begin{split}
        \MonRFPhi(\zeta) \dpe \MonRFinfPhi(\zeta) = \MonRFsupPhi(\zeta) &= \inf \set{\MonRF(D)}{D\in \Metrics \text{ s.t.\ } \Phi(D)\ge \zeta } \\
        &= \inf \set{ \int_{\intervalleff{0}{1}^d }\MonRFelem(g_z)\d z }{g\in \cG\p{\intervalleff{0}{1}^d} \text{ s.t.\ }\Phi(\mathbf D[g]) \ge \zeta },
        \end{split}
    \end{equation}
    and the infimum is attained if it spans on a nonempty set. Moreover, for all $\zeta <\max \Phi$ such that $\MonRFPhi$ is right-continuous at $\zeta$,
    \begin{equation}\label{eqn : App/MonContraction/ContinuousCase}
        \lim_{n\to\infty}-\frac{1}{n^d} \log \Pb{\Phi\p{\SPT[n]} \ge \zeta } = \MonRFPhi(\zeta).
    \end{equation}
\end{Lemma}
Of course a similar result holds for metrics restricted to $\intint{-Cn}{Cn}^d$.
\begin{Lemma}[One-sided contraction principle -- Translated box]
\label{lem : App/MonContractionOtherBox}
	Assume that $b<\infty$. Let $C\ge1$ be an integer. Consider a continuous and nondecreasing function $\Phi:\Metrics\p{\intervalleff{-C}{C}^d} \longrightarrow \R^k$. Then for all $\zeta \in \R^k$,
    \begin{equation}
        \label{eqn : App/MonContractionOtherBox}
        \begin{split}
        \MonRFinfPhi(\zeta) = \MonRFsupPhi(\zeta) &= \inf \set{\MonRF(D)}{D\in \Metrics\p{\intervalleff{-C}{C}^d} \text{ s.t.\ } \Phi(D)\ge \zeta } \\
        &= \inf \set{ \int_{\intervalleff{-C}{C}^d }\MonRFelem(g_z)\d z }{g\in \cG\p{\intervalleff{-C}{C}^d} \text{ s.t.\ }\Phi(\mathbf D[g]) \ge \zeta },
        \end{split}
    \end{equation}
    and the infimum is attained if it spans on a nonempty set. Moreover, for all $\zeta <\max \Phi$ such that $\MonRFPhi$ is right-continuous at $\zeta$,
    \begin{equation}\label{eqn : App/MonContractionOtherBox/ContinuousCase}
        \lim_{n\to\infty}-\frac{1}{n^d} \log \Pb{\Phi\p{\SPT[C,n]} \ge \zeta } = \MonRFPhi(\zeta).
    \end{equation}
\end{Lemma}
We give a proof of Lemma \ref{lem : App/MonContraction} in Appendix \ref{appsec : LD}. Note that by compactness of the sets $\Metrics$ and $\GeoMetrics$, infima in \eqref{eqn : App/Contraction}, \eqref{eqn : App/MonContraction} and \eqref{eqn : App/MonContractionOtherBox} are attained if they span on a nonempty set.

\paragraph{How we use the one-sided contraction principle.} The proofs of Corollaries \ref{cor : Intro/applications/point_point_Box}, \ref{cor : Intro/applications/point_point}, \ref{cor : Intro/applications/PtS} and \ref{cor : Intro/applications/crossing} all follow the same general pattern:
\begin{enumerate}
    \item Apply Lemma \ref{lem : App/MonContraction} or \ref{lem : App/MonContractionOtherBox} to get the existence of $\MonRFPhi$ for the appropriate choice of $\Phi$, and an expression as a variational problem over a set of pseudometrics.
    \item Check that modifying a linear number of edge weights in a realization of $\acc{\Phi(\SPT) \gtrsim \zeta}$ produces a realization of $\acc{\Phi(\SPT) \approx \zeta }$, leading to $\MonRFPhi = I_\Phi$. This step is similar to the proof of Proposition \ref{prop : LDP/wLDP}, except the choice of $\FastE$ is rather straightforward.
    \item Establish the continuity of $\MonRFPhi$ by considering a gradient family $g$ realizing the minimum in the variational formulation of $\MonRFPhi(\zeta)$, gently modify $g$ (i.e.\ increasing a bit $g_z$ everywhere and/or increasing it drastically on a set of small volume) into a candidate for $\MonRFPhi(\zeta+\eps)$.
\end{enumerate}


\subsection{Point-to-point passage time: proof of Corollaries \ref{cor : Intro/applications/point_point_Box} and \ref{cor : Intro/applications/point_point}}
\label{subsec : App/pp}
Fix $x \in \R^d \setminus \acc0$. We prove in Lemma \ref{lem : App/pp/ExistenceContinuity} the existence and continuity of the rate function for the point-to-point passage time between $0$ and $nx$, restricted to a box $\intint{-Cn}{Cn}^d$. We finally get Corollary \ref{cor : Intro/applications/point_point} by checking that for a large value of $C$, $\PT(0,nx) = \BoxPT[ \intervalleff{-Cn}{Cn}^d ](0,nx)$ with high probability, and that for any realization of the event $\acc{ \PT(0,nx) \gtrsim n\zeta}$, modifying a linear number of edges gives a realization of $\acc{ \PT(0,nx) \approx n\zeta}$.

For all $a \norme x \le \zeta \le b\norme x$, we define
\begin{equation}
\label{eqn : App/pp/DefMonRFppxC}
    \begin{split}
        \MonRFppxCinf(\zeta) &\dpe \lim_{\eps \to0} \liminf_{n\to\infty} -\frac1{n^d} \log \Pb{ \SPT[C,n](0,x) \ge \zeta -\eps} \\
        \text{ and }%
        \MonRFppxCsup(\zeta) &\dpe \lim_{\eps \to0} \limsup_{n\to\infty} -\frac1{n^d} \log \Pb{ \SPT[C,n](0,x) \ge \zeta -\eps}.
    \end{split}
\end{equation}
They are rate functions by Lemma \ref{lem : LD/RateFunctions}. We temporarily assume Lemma \ref{lem : App/pp/ExistenceContinuity} that states the equality and continuity of these functions.
\begin{Lemma}
    \label{lem : App/pp/ExistenceContinuity}
    Fix an integer $C> \norme x$. For all $a\norme x\le \zeta \le b\norme x$, the limit
    \begin{equation}
        \label{eqn : App/pp/Existence}
        \MonRFppxC(\zeta) \dpe \MonRFppxCinf(\zeta) = \MonRFppxCsup(\zeta)
    \end{equation}
    exists, and $\MonRFppxC$ is continuous on $\intervalleff{a\norme x}{b\norme x}$. Moreover, for all $a\norme x \le \zeta < b\norme x$,
    \begin{equation}
        \lim_{n\to\infty}-\frac1{n^d}\log \Pb{  \SPT[C,n](0,x) \ge \zeta} = \MonRFppxC(\zeta).
    \end{equation}
\end{Lemma}

\begin{proof}[Proof of Corollary \ref{cor : Intro/applications/point_point_Box}]
    For all $\zeta \in \R$, we define
    \begin{equation}\begin{split}
        \RFppxCinf(\zeta) &\dpe \lim_{\eps \to 0} \liminf_{n\to\infty} -\frac1{n^d} \log \Pb{ \module{ \frac1n\BoxPT[\intervalleff{-Cn}{Cn}^d](0,\floor{nx}) - \zeta} \le \eps } \\%
        \text{ and }%
        \RFppxCsup(\zeta) &\dpe \lim_{\eps \to 0} \limsup_{n\to\infty} -\frac1{n^d} \log \Pb{ \module{ \frac1n\BoxPT[\intervalleff{-Cn}{Cn}^d](0,\floor{nx}) - \zeta} \le \eps }.
    \end{split}\end{equation}
    They are rate functions by Lemma \ref{lem : LD/RateFunctions}. For all $\zeta\in \R \setminus \intervalleff{a\norme x}{b\norme x}$, for sufficiently large $n$, 
    \[
        \acc{ \module{ \frac1n\BoxPT[\intervalleff{-Cn}{Cn}^d](0,\floor{nx}) - \zeta} \le \eps} =\emptyset,
    \]
    thus 
    \begin{equation}\label{eqn : App/pp/Conclusion/RF_TrivialCase}
       \forall \zeta \in \R \setminus \intervalleff{a\norme x}{b\norme x}, \quad \RFppxCinf(\zeta) = \RFppxCsup(\zeta) = \infty.
    \end{equation}
    We claim that
    \begin{equation}
        \label{eqn : App/pp/Conclusion/RF}
        \forall \zeta \in \intervalleff{a\norme x}{b\norme x},\quad \RFppxCinf(\zeta) = \RFppxCsup(\zeta) = \MonRFppxC(\zeta).
    \end{equation}
    
    Indeed, fix $a\norme x \le \zeta \le b\norme x $. Let $\eps>0$. For sufficiently large $n$,
    \[
        \acc{ \module{ \frac1n\BoxPT[\intervalleff{-Cn}{Cn}^d](0,\floor{nx}) - \zeta} \le \eps  } \subseteq \acc{ \SPT[C,n](0,x) \ge \zeta-2\eps },
    \]
    therefore 
    \begin{equation*}
         \liminf_{n\to\infty} -\frac1{n^d} \log \Pb{ \module{ \frac1n\BoxPT[\intervalleff{-Cn}{Cn}^d](0,\floor{nx}) - \zeta} \le \eps  } \ge \liminf_{n\to\infty} -\frac1{n^d} \log \Pb{ \SPT[C,n](0,x) \ge \zeta-2\eps }.
    \end{equation*}
    By \eqref{eqn : App/pp/Existence}, letting $\eps \to 0$ yields
    \begin{equation}
        \label{eqn : App/pp/Conclusion/RF/LB}
        \RFppxCinf(\zeta) \ge \MonRFppxC(\zeta).
    \end{equation}
    
    Let us now prove the corresponding upper bound. Consider a deterministic discrete path $\gamma$ from $0$ to $\floor{nx}$, with $\norme\gamma = \norme{\floor{nx} }$. Let $\tau^{(1)} = \p{ \tau^{(1)}_e }_{e \in \bbE^d}$ and $\tau^{(2)} = \p{ \tau^{(2)}_e }_{e \in \Pathedges{\gamma}}$ be independent, i.i.d.\ weight families with distribution $\nu$. The family $\tau = \p{ \tau_e }_{e \in \bbE^d}$ defined by
    \[ \tau_e \dpe \begin{cases}
        \tau_e^{(1)} &\text{ if } e\notin \Pathedges\gamma,\\ \tau_e^{(2)} &\text{ if } e\in \Pathedges\gamma,
    \end{cases}\]
    is also an i.i.d.\ family with distribution $\nu$. Moreover, there exists a $\tau^{(1)}$-measurable subset $E\subseteq \Pathedges\gamma$ such that 
    \[ \begin{split}
        &\acc{ \frac1n\BoxPT[\intervalleff{-Cn}{Cn}^d]^{(1)}(0,\floor{nx}) \ge \zeta-2\eps } \cap \p{\bigcap_{e\in E} \acc{\tau_e^{(2)} \le a+\eps } } \cap \p{\bigcap_{e\in \Pathedges\gamma \setminus E} \acc{\tau_e^{(2)} \ge b-\eps } } \\
        &\quad \subseteq \acc{  \module{ \frac1n\BoxPT[\intervalleff{-Cn}{Cn}^d](0,\floor{nx}) - \zeta} \le (\norme x +2) \eps },
    \end{split} \]
    thus
    \[ \begin{split}
        & -\log \Pb{ \frac1n\BoxPT[\intervalleff{-Cn}{Cn}^d]^{(1)}(0,\floor{nx}) \ge \zeta-\eps } - \norme{nx }\log \p{ \nu\intervalleff a{a+\eps} \wedge \nu\intervalleff {b-\eps}{b} }  \\
        &\quad \ge - \log \Pb{  \module{ \frac1n\BoxPT[\intervalleff{-Cn}{Cn}^d](0,\floor{nx}) - \zeta} \le (\norme x +2) \eps }.
    \end{split} \]
    Taking the upper limit as $n\to\infty$, we get
    \begin{equation*}
        \limsup_{n\to\infty} -\frac1{n^d}\log \Pb{ \SPT[C,n](0,x) \ge \zeta -\eps } \ge \limsup_{n\to\infty} -\frac1{n^d} \log \Pb{  \module{ \frac1n\BoxPT[\intervalleff{-Cn}{Cn}^d](0,\floor{nx}) - \zeta} \le (\norme x +2) \eps }.
    \end{equation*}
    By \eqref{eqn : App/pp/Existence}, letting $\eps \to 0$ yields
    \begin{equation}
        \label{eqn : App/pp/Conclusion/RF/UB}
        \MonRFppxC(\zeta) \ge \RFppxCsup(\zeta).
    \end{equation}

    The inequalities \eqref{eqn : App/pp/Conclusion/RF/LB} and \eqref{eqn : App/pp/Conclusion/RF/UB} yield \eqref{eqn : App/pp/Conclusion/RF}. By \eqref{eqn : App/pp/Conclusion/RF_TrivialCase} and \eqref{eqn : App/pp/Conclusion/RF}, applying Lemma \ref{lem : Intro/LD/UpperLowerRF} and the remark following it completes the proof.
\end{proof}

\begin{proof}[Proof of Corollary \ref{cor : Intro/applications/point_point}]
    Since $\nu(\acc 0) < p_c(\Z^d)$, the time constant $\mu(\cdot)$ (see \eqref{eqn : Intro/Background/def_TimeConstant}) is a norm. Consequently, there exists an integer $C\ge 1$ such that for all $u \in \partial\p{ \intervalleff{-C}{C}^d }$, $\mu(u) \ge 2b \norme x$. Then by \eqref{eqn : Intro/Background/def_TimeConstant}, with high probability as $n\to \infty$, 
    \[
        \PT\p{0, \partial\p{\intervalleff{-Cn}{Cn}^d} } > b\norme{nx}.
    \]
    Consequently, by the FKG inequality, for all $\zeta \in \intervalleff{a\norme x}{b\norme x}$, 
    \[
        \Pb{\PT(0,\floor{nx}) \ge n\zeta } \underset{n\to\infty}\sim \Pb{\PT_{\intervalleff{-Cn}{Cn}^d}(0,\floor{nx}) \ge n\zeta }.
    \]
    From now on, the proof of Corollary \ref{cor : Intro/applications/point_point_Box} may be copied verbatim.
\end{proof}

\begin{proof}[Proof of Lemma~\ref{lem : App/pp/ExistenceContinuity}]

By Lemma \ref{lem : App/MonContractionOtherBox} applied with $\Phi : D \mapsto D(0,x)$, for all $a \norme x \le \zeta \le b\norme x$,
\begin{equation}\begin{split}
   \MonRFppxC(\zeta) &\dpe \MonRFppxCinf(\zeta)= \MonRFppxCsup(\zeta) \\ &= \min \set{ \int_{\intervalleff{-C}{C}^d }\MonRFelem(g_z)\d z }{g\in \cG\p{\intervalleff{-C}{C}^d} \text{ s.t.\ }\mathbf D[g](0,x) \ge \zeta }.
\end{split}\end{equation}

Since $\MonRFppxCinf$ and $\MonRFppxCsup$ are nondecreasing, it is sufficient to prove that the right-hand side is right-continuous. Let $a\norme x \le \zeta < b\norme x$. There exists $g\in \cG\p{\intervalleff{-C}{C}^d }$ such that  $\mathbf D[g](0,x) \ge \zeta$ and 
\begin{equation}
\label{eqn : App/pp/D_optimal}
	\MonRFppxC(\zeta) = \int_{\intervalleff{-C}{C}^d}\MonRFelem\p{g_z}\d z.
\end{equation}

Our strategy consists in showing the existence of a gradient family $\hat g$, with cost not much higher than $g$, such that $\mathbf D[\hat g](0,x)>\zeta$. The details of the construction of $\hat g$ differ depending on whether $\nu$ has an atom at $b$ or not.

\emph{Case 1: $\nu(\acc b)>0$.} Without loss of generality we assume that $x_i\ge 0$ $\forall i\in\intint1d$. We define the map
\begin{align*}
	f : \intervalleff{-C}{C}^d &\longrightarrow \intervalleff{-C}{C}^d \\
        (z_1,\dots, z_d) &\longmapsto \frac{1}{1+\eps} \p{z + \sum_{i=1}^d \eps x_i \ind{z_i \ge x_i} \base i}.
\end{align*}
Note that $f$ is injective. For all $z\in \intervalleff{-C}{C}^d$ and $u \in \R^d$, let
\begin{equation}
    \hat g_z(u) \dpe \begin{cases}
        g_{ f^{-1}(z) }(u) &\text{ if } z\in f\p{ \intervalleff{-C}{C}^d },\\
        b\norme u &\text{ otherwise}.
    \end{cases}
\end{equation}

We claim that
\begin{equation}\label{eqn : App/pp/ExistenceContinuity/Case1/Cost}
    \int_{\intervalleff{-C}{C}^d}\MonRFelem(\hat g_z) \d z
        \le \int_{\intervalleff{-C}{C}^d}\MonRFelem( g_z) \d z -d(2C)^d\p{ 1- \frac{1}{(1+\eps)^d }} \log \nu(\acc b)
\end{equation}
and
\begin{equation}
\label{eqn : App/pp/ExistenceContinuity/Case1/Time}
	\mathbf D[\hat g](0,x)\ge \frac{\zeta + \eps b\norme x}{1+\eps}>\zeta.
\end{equation}

We have
\begin{align*}
    \int_{\intervalleff{-C}{C}^d}\MonRFelem(\hat g_z) \d z
        &= \int_{f\p{ \intervalleff{-C}{C}^d } } \MonRFelem\p{ g_{f^{-1}(z)} } \d z + \int_{\intervalleff{-C}{C}^d \setminus f\p{ \intervalleff{-C}{C}^d } } \MonRFelem\p{ b\norme\cdot } \d z
    \intertext{Applying a change of variables to the first term and bounding the second one with \eqref{eqn : ERF/LimitCases/infty/UB}, we get}
      \int_{\intervalleff{-C}{C}^d}\MonRFelem(\hat g_z) \d z 
        &\le \frac{1}{(1+\eps)^d} \int_{ \intervalleff{-C}{C}^d  } \MonRFelem\p{ g_{z} } \d z  -d \Leb\p{ \intervalleff{-C}{C}^d \setminus f\p{ \intervalleff{-C}{C}^d } } \log \nu(\acc b),
\end{align*}
hence \eqref{eqn : App/pp/ExistenceContinuity/Case1/Cost}.

Let $\gamma : \intervalleff0T \rightarrow \intervalleff{-C}{C}^d$ be a Lipschitz path such that $\gamma(0)=0$ and $\gamma(T)=x$. Since $\hat g_z = b\norme\cdot$ outside the image of $f$, the integral of $\hat g_{\gamma(t)}\p{\gamma'(t)}\d t$ does not increase when replacing every excursion of $\gamma$ outside $f\p{\intervalleff{-C}{C}^d }$ by the concatenation of $d$ (possibly degenerate) segments, such that the $i$\textsuperscript{th} one is collinear to $\base i$. In particular we may assume that $\gamma$ is of the form
\begin{equation}
	0 \Path{\gamma^\mathrm{int}_1} y(1) \Path{\gamma^\mathrm{ext}_1} z(1) \Path{\gamma^\mathrm{int}_2} y(2) \Path{\gamma^\mathrm{ext}_2} \dots \Path{\gamma^\mathrm{ext}_r} z(r) \Path{\gamma^\mathrm{int}_{r+1}} x,
\end{equation}
where
\begin{enumerate}[(i)]
	\item For all $j\in\intint{1}{r+1}$ $\gamma^\mathrm{int}_j \subseteq f\p{ \intervalleff{-C}{C}^d }$.
	\item For all $j\in\intint1r$, $\gamma^\mathrm{ext}_j=\intervalleff{y(j)}{z(j)}$.
	\item For all $j\in\intint1r$, $y(j) - z(j)$ is collinear to some $\base{i(j)}$ and $\intervalleoo{y(j)}{z(j)}\subseteq \intervalleff{-C}{C}^d \setminus f\p{ \intervalleff{-C}{C}^d }$.
\end{enumerate}
Let $\hat \gamma : \intervalleff{0}{S} \rightarrow \intervalleff{-C}{C}^d$ denote the path obtained by concatenation of the $\p{ f^{-1}\p{ \gamma^\mathrm{int}_j  }}_{j\in\intint1{r+1}}$; it is still a Lipschitz path. Moreover, for almost every $t \in \intervalleff0{T}$, if $s\in \intervalleff0S$ is such that $\gamma(t) = f\p{ \hat \gamma(s) }$,
\[
    \gamma'(t) = \frac{ (\hat \gamma)'(s)}{1+\eps}. 
\]
Consequently,
\begin{align}
    \int_0^T \hat g_{\gamma(t)} \p{ \gamma'(t)} \ind{\gamma(t) \in f\p{\intervalleff{-C}{C}^d } } \d t
        &= \frac{1}{1+\eps}\int_0^S \hat g_{\gamma(s)} \p{ (\hat\gamma)'(s)} \d s \eol
        &= \frac{1}{1+\eps}\int_0^S g_{\hat \gamma(s)} \p{ (\hat\gamma)'(s)} \d s. \nonumber
    \intertext{By \eqref{eqn : Gradient/inverse_def}, we have}
    \int_0^T \hat g_{\gamma(t)} \p{ \gamma'(t)} \ind{\gamma(t) \in f\p{\intervalleff{-C}{C}^d } } \d t
        &\ge \frac{ \mathbf D[g](0,x) }{1+\eps} \ge \frac{\zeta}{1+\eps}
        \label{eqn : App/pp/ExistenceContinuity/Case1/Time/Term1}
\end{align}
Besides,
\begin{align}
    \int_0^T \hat g_{\gamma(t)} \p{ \gamma'(t)} \ind{\gamma(t) \notin f\p{\intervalleff{-C}{C}^d } } \d t
        &= \sum_{j=1}^r b\norme{y(j) - z(j)} = \frac{b \eps}{1+\eps} \sum_{j=1}^r x_{i(j)} \nonumber
    \intertext{By the intermediate value theorem, every coordinate of $x$ appears at least once in the sum above, hence}
    \int_0^T \hat g_{\gamma(t)} \p{ \gamma'(t)} \ind{\gamma(t) \notin f\p{\intervalleff{-C}{C}^d } } \d t
        &\ge \frac{b\eps}{1+\eps} \norme x.
        \label{eqn : App/pp/ExistenceContinuity/Case1/Time/Term2}
\end{align}
Summing \eqref{eqn : App/pp/ExistenceContinuity/Case1/Time/Term1} and \eqref{eqn : App/pp/ExistenceContinuity/Case1/Time/Term2} gives
\[
    \int_0^T\hat g_{\gamma(t)} \p{ \gamma'(t)} \d t \ge \frac{\zeta + b \eps \norme x}{1+\eps}.
\]
Taking the infimum over Lipschitz paths $\gamma$ gives \eqref{eqn : App/pp/ExistenceContinuity/Case1/Time}.

By \eqref{eqn : App/pp/ExistenceContinuity/Case1/Time},
\begin{equation*}
   \int_{\intervalleff{-C}{C}^d}\MonRFelem( g_z) \d z = \MonRFppxC(\zeta) \le \MonRFppxC(\zeta+) \le \int_{\intervalleff{-C}{C}^d}\MonRFelem(\hat g_z) \d z.
\end{equation*}
Plugging in \eqref{eqn : App/pp/ExistenceContinuity/Case1/Cost} yields
\begin{equation*}
    \int_{\intervalleff{-C}{C}^d}\MonRFelem( g_z) \d z = \MonRFppxC(\zeta) \le \MonRFppxC(\zeta+) \le \int_{\intervalleff{-C}{C}^d}\MonRFelem( g_z) \d z -d(2C)^d\p{ 1- \frac{1}{(1+\eps)^d }} \log \nu(\acc b).
\end{equation*}
Letting $\eps \to 0$ gives the desired right-continuity.
	

\emph{Case 2: $\nu(\acc b)=0$.} We follow the proof ideas of Proposition~1.4 in Basu, Ganguly and Sly \cite{Bas21}, in a deterministic setting: instead of slightly increasing every edge passage time in a configuration satisfying $\acc{\SPT[C,n](0,x)\ge \zeta}$, we slightly increase $g_z$. Let $0 < \eps \le \delta \le \frac{b\norme x - \zeta }{10\norme x}$. Let $\omega_\delta$ be a modulus of continuity for the restriction of $\MonRFelem$ to the set of seminorms $g$ such that $ g \le (b-\delta)\norme\cdot$, which is compact. Define the sets
\begin{equation}
	\label{eqn : App/pp/Continuity/cas_sans_atome/X1_X2}
	X_1(\delta) \dpe \acc{z\in \intervalleff{-C}{C}^d \Bigm| \normeHom{g_z}\le (b-2\delta) } \text{ and } X_2(\delta)\dpe \intervalleff{-C}{C}^d \setminus X_1(\delta).
\end{equation}
For all $z\in \intervalleff{-C}{C}^d$, define
\begin{equation*}
	\GB \dpe%
		\begin{cases}
			g_z + \eps\norme\cdot \qquad &\text{ if } z \in X_1(\delta),\\
			\cro{(b-8\delta)\norme\cdot}\vee g_z \qquad &\text{ if } z \in X_2(\delta).
		\end{cases}
\end{equation*}
We claim that
\begin{equation}
    \label{eqn : App/pp/ExistenceContinuity/Case2/Cost}
    \int_{\intervalleff{-C}{C}^d}\MonRFelem(\hat g_z) \d z \le \int_{\intervalleff{-C}{C}^d}\MonRFelem(g_z) \d z  + (2C)^d\omega_\delta(\eps) + (2d+1) \int_{X_2(\delta)} \MonRFelem\p{ g_z} \d z + 2d\Leb\p{X_2(\delta)} \log 2,
\end{equation}
and
\begin{equation}
    \label{eqn : App/pp/ExistenceContinuity/Case2/Time}
    \mathbf D[\hat g](0,x) \ge \zeta +\p{\delta \wedge \frac{\eps \delta}{b-8\delta}}\norme x.
\end{equation}

We have
\begin{equation}
\label{eqn : App/pp/Continuity/cas_sans_atome/FdT_D'}
	\int_{\intervalleff{-C}{C}^d}\MonRFelem(\hat g_z) \d z = \int_{X_1(\delta)} \MonRFelem\p{g_z + \eps\norme\cdot } \d z + \int_{X_2(\delta)} \MonRFelem\p{\cro{(b-8\delta)\norme\cdot}\vee g_z} \d z.
\end{equation}
The first term is upper bounded by 
\begin{align}
	\int_{X_1(\delta)} \MonRFelem\p{g_z + \eps\norme\cdot } \d z %
        &\le \int_{X_1(\delta)} \MonRFelem\p{g_z  }\d z  + \Leb(X_1(\delta))\omega_\delta(\eps) \eol
		& \le \int_{X_1(\delta)} \MonRFelem\p{g_z  }\d z  + (2C)^d\omega_\delta(\eps).\label{eqn : App/pp/Continuity/cas_sans_atome/UB_FdT_1}
\end{align}
Let us bound the second. Let $z\in X_2(\delta)$. First, note that by~\eqref{eqn : ERF/LimitCases/infty/UB} and~\eqref{eqn : ERF/LimitCases/infty/LB},
\begin{align}
	\MonRFelem\p{(b-8\delta)\norme\cdot} %
		&\le  - d\log \nu\p{\intervalleff{b-8\delta}{b} } \nonumber\\
		&\le  2d\MonRFelem\p{g_z  } + 2d\log 2.\nonumber
	\intertext{Moreover, by the FKG inequality,}
	\MonRFelem\p{\cro{(b-8\delta)\norme\cdot}\vee g_z} &\le \MonRFelem\p{(b-8\delta)\norme\cdot} + \MonRFelem\p{g_z  }\eol
		&\le (2d+1)\MonRFelem\p{g_z } + 2d\log 2.\notag
	\intertext{This implies}
	\label{eqn : App/pp/Continuity/cas_sans_atome/UB_FdT_2}
	\int_{X_2(\delta)} \MonRFelem\p{\cro{(b-8\delta)\norme\cdot}\vee g_z} \d z%
		&\le (2d+1) \int_{X_2(\delta)} \MonRFelem\p{ g_z} \d z + 2d \Leb\p{X_2(\delta)} \log 2.
	\end{align}
	Combining~\eqref{eqn : App/pp/Continuity/cas_sans_atome/UB_FdT_1} and~\eqref{eqn : App/pp/Continuity/cas_sans_atome/UB_FdT_2}, we obtain \eqref{eqn : App/pp/ExistenceContinuity/Case2/Cost}.

    Let $\gamma : \intervalleff0T \rightarrow \intervalleff{-C}{C}^d$ be a Lipschitz path such that $\gamma(0)=0$ and $\gamma(T)=x$. If
\begin{align}
	\int_0^T \norme{\gamma'(t)} \ind{X_2(\delta)}\p{\gamma(t)} \d t &\ge \frac{\zeta + \delta\norme x}{b-8\delta},\notag
	\intertext{then}
	\int_0^T \GB[\gamma(t)]\p{\gamma'(t)} \d t &\ge (b-8\delta)\int_0^T \norme{\gamma'(t)} \ind{X_2(\delta)}\p{\gamma(t)} \d t \eol
	\label{eqn : App/pp/Continuity/cas_sans_atome/LB_PT_1}
	&\ge \zeta + \delta\norme x.
\end{align}
Otherwise,
\begin{align}
	\int_0^T  \norme{\gamma'(t)} \ind{X_1(\delta)}\p{\gamma(t)} \d t\notag %
		&= \int_0^T \norme{\gamma'(t)}\d t - \int_0^T \norme{\gamma'(t)} \ind{X_2(\delta)}\p{\gamma(t)} \d t\notag\\
		&\ge \norme x - \frac{\zeta + \delta\norme x }{b-8\delta}\notag \\
		&= \frac{b\norme x - \zeta - 9\delta\norme x}{b-8\delta}\notag\\
		&\ge \frac{\delta \norme x}{b-8\delta}.\notag
\end{align}
Consequently,
\begin{align}
	\int_0^T \GB[\gamma(t)]\p{\gamma'(t)} \d t%
		&\ge \int_0^T g_{\gamma(t)}\p{\gamma'(t)} \d t + \int_0^T \p{ \GB[\gamma(t)]\p{\gamma'(t)} - g_{\gamma(t)}\p{\gamma'(t)}  } \ind{X_1(\delta)}\p{\gamma(t)} \d t\eol
		&\ge  \mathbf D[g](\gamma) + \eps\int_0^T \norme{\gamma'(t)} \ind{X_1(\delta)}\p{\gamma(t)}\d t\eol
	\label{eqn : App/pp/Continuity/cas_sans_atome/LB_PT_2}
		&\ge \zeta + \frac{\eps \delta}{b-8\delta}\norme x. 
\end{align}
Combining inequalities~\eqref{eqn : App/pp/Continuity/cas_sans_atome/LB_PT_1} and~\eqref{eqn : App/pp/Continuity/cas_sans_atome/LB_PT_2} with the definition of $\mathbf D[\hat g] $, we get \eqref{eqn : App/pp/ExistenceContinuity/Case2/Time}.

By \eqref{eqn : App/pp/ExistenceContinuity/Case2/Time},
\begin{equation*}\begin{split}
    \int_{\intervalleff{-C}{C}^d}\MonRFelem( g_z) \d z = \MonRFppxC(\zeta) \le \MonRFppxC(\zeta+) \le \int_{\intervalleff{-C}{C}^d}\MonRFelem(\hat g_z) \d z.
\end{split}\end{equation*}
Plugging in \eqref{eqn : App/pp/ExistenceContinuity/Case2/Cost} yields
\begin{equation*}\begin{split}
    \int_{\intervalleff{-C}{C}^d}\MonRFelem( g_z) \d z = \MonRFppxC(\zeta) &\le \MonRFppxC(\zeta+) \\
    &\le \int_{\intervalleff{-C}{C}^d}\MonRFelem( g_z) \d z  + (2C)^d\omega_\delta(\eps) \\
    &\qquad + (2d+1) \int_{X_2(\delta)} \MonRFelem\p{ g_z} \d z + 2d\Leb\p{X_2(\delta)} \log 2.
\end{split}\end{equation*}
Letting $\eps \to 0$ then $\delta\to0$ gives the desired right-continuity.
\end{proof}

\subsection{Point-to-sphere passage time: proof of Corollary~\ref{cor : Intro/applications/PtS}}
\label{subsec : App/PtS}
We prove in Lemma \ref{lem : App/PtS/ExistenceContinuity} the existence and continuity of the rate function for the point-to-sphere passage time. We finally get Corollary \ref{cor : Intro/applications/PtS} by checking that for any realization of the event $\acc{ \PtS(n) \gtrsim \zeta}$, modifying a linear number of edges gives a realization of $\acc{ \PtS(n) \approx \zeta}$.
For all $a\le \zeta \le b $, we define
\begin{equation}\begin{split}
        \MonRFPtSinf(\zeta) &\dpe \lim_{\eps \to 0} \liminf_{n\to\infty} -\frac1{n^d} \log \Pb{ \PtS(n) \ge \zeta - \eps }, \\%
        \MonRFPtSsup(\zeta) &\dpe \lim_{\eps \to 0} \limsup_{n\to\infty} -\frac1{n^d} \log \Pb{ \PtS(n) \ge \zeta - \eps }.
\end{split}\end{equation}
They are rate functions by Lemma \ref{lem : LD/RateFunctions}.
\begin{Lemma}
    \label{lem : App/PtS/ExistenceContinuity}
    For all $a\le \zeta \le b$,
    \begin{equation}
        \label{eqn : App/PtS/Existence}
        \MonRFPtS(\zeta) \dpe \MonRFPtSinf(\zeta) = \MonRFPtSsup(\zeta) 
    \end{equation}
   Moreover, $\MonRFPtS$ is continuous on $\intervalleff ab$.
\end{Lemma}

\begin{proof}[Proof of Corollary \ref{cor : Intro/applications/PtS}]
    For all $a \le \zeta \le b $, we define
    \begin{equation}\begin{split}
        \RFPtSinf(\zeta) &\dpe \lim_{\eps \to 0} \liminf_{n\to\infty} -\frac1{n^d} \log \Pb{ \module{ \PtS(n) - \zeta} \le \eps }, \\%
        \text{and }\RFPtSsup(\zeta) &\dpe \lim_{\eps \to 0} \limsup_{n\to\infty} -\frac1{n^d} \log \Pb{ \module{ \PtS(n) - \zeta} \le \eps}.
    \end{split}\end{equation}
    They are rate functions by Lemma \ref{lem : LD/RateFunctions}. Fix $a \le \zeta < b $. We claim that
    \begin{equation}
        \label{eqn : App/PtS/Conclusion/RF}
        \RFPtSinf(\zeta) = \RFPtSsup(\zeta) = \MonRFPtS(\zeta).
    \end{equation}
    Let $\eps>0$. For all $n\ge 1$,
    \[
        \acc{ \module{ \PtS(n) - \zeta} \le \eps  } \subseteq \acc{ \PtS(n) \ge \zeta-\eps },
    \]
    therefore by \eqref{eqn : App/PtS/Existence}, 
    \begin{equation*}
         \liminf_{n\to\infty} -\frac1{n^d} \log \Pb{ \PtS(n) \ge \zeta-\eps  } \ge \MonRFPtS(\zeta-\eps).
    \end{equation*}
    By continuity of $\MonRFPtS$, letting $\eps \to 0$ yields
    \begin{equation}
        \label{eqn : App/PtS/Conclusion/RF/LB}
        \RFPtSinf(\zeta) \ge \MonRFPtS(\zeta).
    \end{equation}
    
    Let us now prove the corresponding upper bound. Let $\eps>0$. Consider the straight discrete path $\gamma$ from $0$ to $\floor{n \base 1}$. Let $\tau^{(1)} = \p{ \tau^{(1)}_e }_{e \in \bbE^d}$ and $\tau^{(2)} = \p{ \tau^{(2)}_e }_{e \in \Pathedges{\gamma}}$ be independent, i.i.d.\ weight families with distribution $\nu$. The family $\tau = \p{ \tau_e }_{e \in \bbE^d}$ defined by
    \[ \tau_e \dpe \begin{cases}
        \tau_e^{(1)} &\text{ if } e\notin \Pathedges\gamma,\\ \tau_e^{(2)} &\text{ if } e\in \Pathedges\gamma,
    \end{cases}\]
    is also an i.i.d.\ family with distribution $\nu$. Moreover, there exists a $\tau^{(1)}$-measurable subset $E\subseteq \Pathedges\gamma$ such that 
    \[ \begin{split}
        &\acc{ \PtS^{(1)}(n) \ge \zeta-\eps } \cap \p{\bigcap_{e\in E} \acc{\tau_e^{(2)} \le a+\eps } } \cap \p{\bigcap_{e\in \Pathedges\gamma \setminus E} \acc{\tau_e^{(2)} \ge b-\eps } } \\
        &\quad \subseteq \acc{  \module{ \PtS(n) - \zeta} \le 2 \eps },
    \end{split} \]
    thus
    \[ \begin{split}
        & -\log \Pb{ \PtS(n) \ge \zeta-\eps } - n\log \p{ \nu\intervalleff a{a+\eps} \wedge \nu\intervalleff {b-\eps}{b} }  \\
        &\quad \ge - \log \Pb{  \module{ \PtS(n) - \zeta} \le 2 \eps }.
    \end{split} \]
    Consequently,
    \begin{equation*}
         \limsup_{n\to\infty} -\frac{1}{n^d} \log \Pb{ \PtS(n) \ge \zeta-\eps } \ge \limsup_{n\to\infty} -\frac1{n^d} \log \Pb{  \module{ \PtS(n) - \zeta} \le 2 \eps }.
    \end{equation*}
    By \eqref{eqn : App/PtS/Existence}, letting $\eps \to 0$ yields
    \begin{equation}
        \label{eqn : App/PtS/Conclusion/RF/UB}
        \MonRFPtS(\zeta) \ge \RFPtSsup(\zeta).
    \end{equation}

    The inequalities \eqref{eqn : App/PtS/Conclusion/RF/LB} and \eqref{eqn : App/PtS/Conclusion/RF/UB} yield \eqref{eqn : App/PtS/Conclusion/RF}. Moreover, the three functions involved are nondecreasing and lower semi-continuous, hence \eqref{eqn : App/PtS/Conclusion/RF} still holds for $\zeta = b$. Applying Lemma \ref{lem : Intro/LD/UpperLowerRF} and the remark following it, we deduce that $\p{\PtS(n)}_{n\ge 1}$ follows the LDP at speed $n^d$ for the rate function $\MonRFPtS$.
\end{proof}

\begin{proof}[Proof of Lemma~\ref{lem : App/PtS/ExistenceContinuity}]

By Lemma \ref{lem : App/MonContractionOtherBox} applied with $\Phi : D \mapsto D\p{0,\partial\p{\intervalleff{-1}{1}^d } }$, for all $a \le \zeta \le b$,
\begin{equation}\begin{split}
    \MonRFPtS(\zeta) &\dpe \MonRFPtSinf(\zeta) = \MonRFPtSsup(\zeta) \\
    &=  \min \set{ \int_{\intervalleff{-1}{1}^d }\MonRFelem(g_z)\d z }{g\in \cG\p{\intervalleff{-1}{1}^d} \text{ s.t.\ }\mathbf D[g]\p{0,\partial\p{\intervalleff{-1}{1}^d } } \ge \zeta }.
\end{split}\end{equation}

It remains to prove that $\MonRFPtS$ is continuous. Since it is nondecreasing and lower semi-continuous, only right-continuity needs to be checked. Let $a \le \zeta < b$. Fix $\eta>0$ such that $\zeta < b-\eta$. There exists $g \in \cG\p{\intervalleff{-1}{1}^d }$ such that  $\mathbf D[g]\p{0, \partial \p{\intervalleff{-1}{1}^d} } \ge \zeta$ and 
\begin{equation}
\label{eqn : App/PtS/D_optimal}
	\MonRFPtS(\zeta) = \int_{\intervalleff{-1}{1}^d }\MonRFelem(g_z)\d z.
\end{equation}

For all $z\in \intervalleff{-1}{1}^d$ and $u \in \R^d$, let
\begin{equation}
    \hat g_z(u) \dpe \begin{cases}
        g_{ (1-\eps)^{-1}z }(u) &\text{ if } z\in \intervalleff{-(1-\eps) }{1-\eps }^d ,\\
        (b-\eta)\norme u &\text{ otherwise}.
    \end{cases}
\end{equation}
We claim that
\begin{equation}\label{eqn : App/PtS/ExistenceContinuity/Cost}
     \int_{\intervalleff{-1}{1}^d }\MonRFelem(\hat g_z)\d z
        \le  \int_{\intervalleff{-1}{1}^d }\MonRFelem(g_z)\d z -d2^d\p{ 1- (1-\eps)^d } \log \nu\p{\intervalleff{b-\eta}{b} },
\end{equation}
and
\begin{equation}
\label{eqn : App/PtS/ExistenceContinuity/Time}
	\mathbf D[\hat g] \p{0, \partial \p{\intervalleff{-1}{1}^d} } \ge (1-\eps)\zeta + \eps(b-\eta) >\zeta.
\end{equation}

Indeed,
\begin{align*}
     \int_{\intervalleff{-1}{1}^d }\MonRFelem(\hat g_z)\d z 
        &\le \int_{ \intervalleff{-(1-\eps)}{1-\eps}^d  } \MonRFelem\p{ g_{(1-\eps)^{-1}z} } \d z + \int_{\intervalleff{-1}{1}^d \setminus \intervalleff{-(1-\eps) }{1-\eps }^d} \MonRFelem\p{ (b-\eta)\norme\cdot } \d z.
    \intertext{Applying a change of variables to the first term and bounding the second one with \eqref{eqn : ERF/LimitCases/infty/UB}, we get}
       \int_{\intervalleff{-1}{1}^d }\MonRFelem(\hat g_z)\d z 
        &\le (1-\eps)^d \int_{ \intervalleff{-1}{1}^d  } \MonRFelem\p{ g_{z} } \d z  \\
        &\qquad -d \Leb\p{ \intervalleff{-1}{1}^d \setminus  \intervalleff{-(1-\eps)}{1-\eps}^d }  \log \nu(\intervalleff{b-\eta}{b}),
\end{align*}
hence \eqref{eqn : App/PtS/ExistenceContinuity/Cost}.

Let $\gamma : \intervalleff0T \rightarrow \intervalleff{-1}{1}^d$ be a Lipschitz path such that $\gamma(0)=0$ and $\gamma(T) \in \partial\p{\intervalleff{-1}{1}^d}$. Let $t_1$ and $t_2$ respectively denote the first and the last visit times of $\gamma$ in $\partial \p{ \intervalleff{-(1-\eps)}{1-\eps}^d }$. We have
\begin{align}
    \int_0^{t_1} \hat g_{\gamma(t)}\p{ \gamma'(t) }\d t%
        &= \int_0^{t_1} g_{(1-\eps)^{-1}\gamma(t)}\p{ \gamma'(t) }\d t \eol
        &= (1-\eps) \int_0^{t_1} g_{\hat \gamma(t)}\p{ (\hat\gamma)'(t) }\d t, \nonumber
    \intertext{where $\hat \gamma(t) = (1-\eps)^{-1}\gamma(t)$. Hence, by Proposition \ref{prop : Gradient/courbe}, }
    \int_0^{t_1} \hat g_{\gamma(t)}\p{ \gamma'(t) }\d t%
        &\ge (1-\eps) \mathbf D[g]\p{0, \partial \p{\intervalleff{-1}{1}^d} } \ge (1-\eps)\zeta.
    \label{eqn : App/PtS/ExistenceContinuity/Time/Term1}
\end{align}
Besides, since $\hat g_z =(b-\eta)\norme \cdot$ outside $\intervalleff{-(1-\eps)}{1-\eps}^d$,
\begin{align}
    \int_{t_2}^{T} \hat g_{\gamma(t)}\p{ \gamma'(t) }\d t%
        &= (b-\eta) \int_{t_2}^{T} \norme{ \gamma'(t) }\d t. \nonumber
    \intertext{By Jensen's inequality and Lemma \ref{lem : Gradient/Rademacher_kindof},}
     \int_{t_2}^{T} \hat g_{\gamma(t)}\p{ \gamma'(t) }\d t%
        &\ge (b-\eta) \norme{\int_{t_2}^{T} \gamma'(t) \d t } \nonumber \\
        &= (b-\eta) \norme{\gamma(T ) - \gamma(t_2)}.
    \intertext{Consequently, }
    \int_{t_2}^{T} \hat g_{\gamma(t)}\p{ \gamma'(t) }\d t%
        &\ge \eps(b-\eta).
    \label{eqn : App/PtS/ExistenceContinuity/Time/Term2}
\end{align}
Summing \eqref{eqn : App/PtS/ExistenceContinuity/Time/Term1} and \eqref{eqn : App/PtS/ExistenceContinuity/Time/Term2} yields \eqref{eqn : App/PtS/ExistenceContinuity/Time}.

By \eqref{eqn : App/PtS/ExistenceContinuity/Time},
\begin{equation*}
     \int_{\intervalleff{-1}{1}^d }\MonRFelem(g_z)\d z = \MonRFPtS(\zeta) \le \MonRFPtS(\zeta+) \le \int_{\intervalleff{-1}{1}^d }\MonRFelem(\hat g_z)\d z.
\end{equation*}
Plugging in \eqref{eqn : App/PtS/ExistenceContinuity/Cost} yields
\begin{equation*}\begin{split}
     \int_{\intervalleff{-1}{1}^d }\MonRFelem(g_z)\d z &= \MonRFPtS(\zeta) \le \MonRFPtS(\zeta+) \\
     &\le  \int_{\intervalleff{-1}{1}^d }\MonRFelem(g_z)\d z -d2^d\p{ 1- (1-\eps)^d } \log \nu\p{\intervalleff{b-\eta}{b} }.
\end{split}\end{equation*}
Letting $\eps \to 0$ gives the desired right-continuity.
\end{proof}

\subsection{Crossing time: proof of Corollary~\ref{cor : Intro/applications/crossing}}
\label{subsec : App/crossing}
In this subsection, we only assume that $\nu$ has a bounded support, i.e. $0\le a < b < \infty$. We first prove in Lemma \ref{lem : App/CT/LDP} the LDP and give an expression of the rate function $\RFCT$ as the minimum of $\MonRF(D)$ over the relevant set of metrics. The properties of $\RFCT$ are given in Lemma \ref{lem : App/CT/Ptys_RF}.

\begin{Lemma}
    \label{lem : App/CT/LDP}
    The process $(\CT(n))_{n\ge 1}$ follows the large deviation principle at speed $n^d$, in $\intervalleff ab^d$, with the rate function
    \begin{equation}\label{eqn : App/CT/LDP}
        \begin{split}
            \RFCT : \intervalleff ab^d &\longrightarrow \intervalleff0\infty \\
            \p{\zeta_1,\dots, \zeta_d} &\longmapsto \min \set{ \MonRF(D) }{ D \in \Metrics, \; \forall i\in\intint1d, \; D(H_i, H_i') \ge \zeta_i}.
        \end{split}
    \end{equation}
\end{Lemma}
\begin{Lemma}
    \label{lem : App/CT/MagicNorm}
    For all $\zeta = (\zeta_1,\dots, \zeta_d) \in \intervalleff ab^d$,
    \begin{equation}\label{eqn : App/CT/MagicNorm}
    \RFCT(\zeta) = \MonRFelem\p{\mathfrak g^\zeta},
  \end{equation}
  where $\mathfrak g^\zeta$ is the seminorm defined by
  \begin{equation}
  		\mathfrak g^\zeta(u) \dpe \max_{1\le i \le d}\p{\zeta_i \module{u_i}}.
  \end{equation}
\end{Lemma}
\begin{Lemma}
\label{lem : App/CT/Ptys_RF}\leavevmode\vspace{-\baselineskip}
\begin{enumerate}[(i)]
	\item \label{item : App/CT/continuite}%
	$\RFCT$ is continuous on $\intervalleff{a}{b}^d$.
	\item \label{item : App/CT/croissance}%
	$\RFCT$ is nondecreasing on $\intervalleff{a}{b}^d$ for the componentwise order, i.e. for all $\zeta= (\zeta_1,\dots, \zeta_d)$ and $\zeta'=(\zeta_1',\dots, \zeta_d')$ in $\intervalleff{a}{b}^d$, if $\zeta_i \le \zeta_i'$ for all $i$, then $\RFCT(\zeta) \le \RFCT(\zeta')$.
	\item \label{item : App/CT/convexite}%
	$\RFCT$ is separately convex on $\intervalleff{a}{b}^d$, i.e. for all $1\le i \le d$ and $\zeta_1,\dots, \zeta_{i-1}, \zeta_{i+1},\dots, \zeta_d \in \intervalleff ab$, the function
	\begin{equation*}
		t \mapsto \RFCT\p{\zeta_1 ,\dots, \zeta_{i-1}, t, \zeta_{i+1}, \dots, \zeta_d }
	\end{equation*}
	is convex on $\intervalleff ab$.
	\item \label{item : App/CT/cas_infty}%
	$\RFCT(b,\dots,b)<\infty$ if and only if $\nu(\acc b)>0$. 
\end{enumerate}
\end{Lemma}
\begin{proof}[Proof of Lemma \ref{lem : App/CT/LDP}]
For all $\zeta \in \intervalleff{a}{b}^d$, we define
\begin{align}
	\label{eqn : App/CT/def_FdTCTsup}
	\RFCTsup(\zeta) &\dpe \inclim{\eps\to0} \limsup_{n\to\infty} -\frac{1}{n^d}\log \Pb{\norme[\infty]{ \CT(n) - \zeta} \le \eps },\\
	\label{eqn : App/CT/def_FdTCTinf}
	\RFCTinf(\zeta) &\dpe \inclim{\eps\to0} \liminf_{n\to\infty} -\frac{1}{n^d}\log \Pb{\norme[\infty]{ \CT(n) - \zeta} \le \eps },\\
	\label{eqn : App/CT/def_monFdTCTsup}
	\MonRFCTsup(\zeta) &\dpe \inclim{\eps\to0} \limsup_{n\to\infty} -\frac{1}{n^d}\log \Pb{\bigcap_{i=1}^d\acc{\BoxSPT[n,\intervalleff01^d](H_i, H_i') \ge \zeta_i - \eps} },\\
	\label{eqn : App/CT/def_monFdTCTinf}
	\MonRFCTinf(\zeta) &\dpe \inclim{\eps\to0} \liminf_{n\to\infty} -\frac{1}{n^d}\log \Pb{\bigcap_{i=1}^d\acc{\BoxSPT[n,\intervalleff01^d](H_i, H_i') \ge \zeta_i - \eps} }.
\end{align}
They are rate functions by Lemma \ref{lem : LD/RateFunctions}. Let $\zeta \in \intervalleff{a}{b}^d$. By Lemma \ref{lem : Intro/LD/UpperLowerRF}, it is sufficient to show that 
\begin{equation*}
    \RFCTinf(\zeta) = \RFCTsup(\zeta) = \min \set{ \MonRF(D) }{ D \in \Metrics, \; \forall i\in\intint1d, \; D(H_i, H_i') \ge \zeta_i}.
\end{equation*}
Moreover, Lemma \ref{lem : App/MonContraction} implies 
\begin{equation*}
    \MonRFCTinf(\zeta) = \MonRFCTsup(\zeta) = \min \set{ \MonRF(D) }{ D \in \Metrics, \; \forall i\in\intint1d, \; D(H_i, H_i') \ge \zeta_i},
\end{equation*}
thus we only need to show that
\begin{equation}\label{eqn : App/CT/LDP/WinCon}
    \RFCTinf(\zeta) \ge \MonRFCTinf(\zeta) \text{  and  } \RFCTsup(\zeta) \le \MonRFCTsup(\zeta).
\end{equation}

The first inequality is a consequence of the inclusion
\[
    \acc{\norme[\infty]{ \CT(n) - \zeta} \le \eps } \subseteq \bigcap_{i=1}^d\acc{\BoxSPT[n,\intervalleff01^d](H_i, H_i') \ge \zeta_i - \eps},
\]
for all $n\ge 1$ and $\eps>0$. We show the second one by a modification argument similar to the one used in the proof of Corollary \ref{cor : Intro/applications/point_point}. For all $i\in \intint1d$, we denote by $\gamma_i$ the straight discrete path from $0$ to $n\base i$. Let $\tau = \p{ \tau_e }_{e \in \bbE^d}$ and $\tau' = \p{ \tau'_e }_{e \in  \bigcup_{i=1}^d \Pathedges{\gamma_i}}$ be independent, i.i.d.\ weight families with distribution $\nu$. The family $\tau^* = \p{ \tau_e^* }_{e \in \bbE^d}$ defined by
    \[ \tau_e^* \dpe \begin{cases}
        \tau_e &\text{ if } e\notin \bigcup_{i=1}^d \Pathedges{\gamma_i},\\ \tau_e' &\text{ if } e\in \bigcup_{i=1}^d \Pathedges{\gamma_i},
    \end{cases}\]
    is also an i.i.d.\ family with distribution $\nu$. We construct an event \stepcounter{tile}$\cFav$ whose probability has the same order as $\Pb{\bigcap_{i=1}^d \acc{ \SPT(H_i, H_i') \ge \zeta_i - \eps}}$, such that
\begin{equation}
\label{eqn : App/CT/modification/Fav_inclusion}
 	\cFav \subseteq \acc{\norme[\infty]{ \FFastCT(n) - \zeta} \le2 \eps}.
 \end{equation} 

Let $n\ge 1$ and $1\le i \le d$. For all $0\le R \le n$, consider the modified configuration defined by
\begin{equation*}
	\FastEPT[e]{i,R} \dpe%
		\begin{cases}
			a &\text{ if } e\in\acc{(r\base i, (r+1)\base i), 0\le r \le R-1  }\\
            b &\text{ if } e\in\acc{(r\base i, (r+1)\base i), R\le r \le n-1  }\\
			\tau_e &\text{ otherwise.} 
		\end{cases}
\end{equation*}
Define the random variable 
\begin{equation*}
	R_i \dpe \min \acc{0\le R \le n \Biggm| \FastBoxSPT[n]{i,R}(H_i, H_i') \le \zeta_i }.
\end{equation*}
It is clear that $\FastBoxSPT[n]{i,n}(H_i, H_i') =a \le \zeta_i$, thus $R_i$ is well-defined. Note that for large $n$, on the event $\acc{ \SPT[n](H_i, H_i') \ge \zeta_i - \eps}$, 
\begin{equation}
\label{eqn : App/CT/modification/ideal}
	\zeta_i - \eps \le \FastBoxSPT[n]{i,R_i}(H_i, H_i') \le \zeta_i.
\end{equation}
Let
\begin{equation}
	\FastE \dpe \bigcup_{i=1}^d\acc{(r\base i, (r+1)\base i), 0\le r \le R_i-1 }.
\end{equation}
We can now define $\cFav$ as
\begin{align}
	\begin{split}
	\cFav &\dpe %
			\p{ \bigcap_{i=1}^d\acc{ \SPT[n](H_i, H_i') \ge \zeta_i - \eps} }%
            \cap\p{\bigcap_{i=1}^d \bigcap_{r=0}^{R_i-1} \acc{\EPT'\p{ r\base i, (r+1)\base i } \le a+\eps }}\\
		&\qquad \cap\p{\bigcap_{i=1}^d \bigcap_{r=R_i}^{n-1} \acc{\EPT'\p{ r\base i, (r+1)\base i } \ge b-\eps }}.
	\end{split}
	\intertext{By independence of $\EPT$ and $\EPT'$,}
	\label{eqn : App/CT/modification/Pb_Fav}
	\Pb{\cFav} &\ge %
		\Pb{ \bigcap_{i=1}^d \acc{ \BoxSPT[n,\intervalleff01^d](H_i, H_i') \ge \zeta_i - \eps} }%
		\cdot\p{ \nu\p{\intervalleff{a}{a+\eps}}\wedge \nu\p{\intervalleff{b-\eps}{b}} }^{nd}.
\end{align}
For all $1\le i \le d$, there exists a $\BoxPT$-geodesic $nH_i \Path\gamma nH_i'$ which does not use any edge in $nH_i$, therefore \[ \module{\FFastBoxSPT[n](H_i, H_i') - \FastBoxSPT[n]{i,R_i}(H_i, H_i') } \le \eps. \] Consequently, applying~\eqref{eqn : App/CT/modification/ideal}, we get~\eqref{eqn : App/CT/modification/Fav_inclusion} for large enough $n$. The inequality~\eqref{eqn : App/CT/modification/Pb_Fav} then implies
\begin{equation}\begin{split}
	&-\frac{1}{n^d} \log\Pb{\norme[\infty]{ \CT(n) - \zeta}\le 2\eps}\\
		&\quad\le -\frac{1}{n^d} \log\Pb{ \bigcap_{i=1}^d \acc{\BoxSPT[n,\intervalleff01^d](H_i, H_i') \ge \zeta_i - \eps} } -\frac{d}{n^{d-1}} \log\p{ \nu\p{\intervalleff{a}{a+\eps}}\wedge \nu\p{\intervalleff{b-\eps}{b}} }.\notag
\end{split}\end{equation}
Taking the upper limit as $n\to\infty$ then $\eps \to 0$ yields the second inequality in \eqref{eqn : App/CT/LDP/WinCon}.
\end{proof}

\begin{proof}[Proof of Lemma~\ref{lem : App/CT/MagicNorm} ]
Fix  $\zeta\in \intervalleff{a}{b}^d$. The inequality
\begin{equation}
	\RFCT(\zeta) \le \MonRF\p{\d_{\mathfrak g^\zeta}} = \MonRFelem\p{\mathfrak g^\zeta}
\end{equation}
is straightforward. Let us prove the converse inequality.

Consider the function
\begin{align}
	f : \intervalleff{0}{1}^d  &\longrightarrow \intervalleff01^d\eol
		z &\longmapsto 2\sum_{i=1}^d \p{1 - z_i}\wedge z_i \base i.
\end{align}
For all $D \in \Metrics$, we construct another pseudometric $s(D) \in \Metrics$ as follows. For all $z\in \intervalleff01^d$ and $u \in \R^d$, let
\[
    g_z(u) \dpe
    \begin{cases}
        b\norme u &\text{ if } z_i=1/2 \text{ for some }i\in\intint1d,\\
        (\grad D)_{f(z)}\p{ \sum_{i=1}^d \ind{z_i < 1/2}u_i \base i - \sum_{i=1}^d \ind{z_i > 1/2}u_i \base i } &\text{ otherwise.}
    \end{cases}
\]
For all $x,y\in \intervalleff01^d$, define
\begin{equation}\label{eqn : App/CT/MagicMap}
    s(D)(x,y) \dpe \inf \set{ \int_0^T g_{\gamma(t) }\p{ \gamma'(t)}\d t }{ \gamma : \intervalleff0T \rightarrow \intervalleff{0}{1}^d \text{ Lipschitz s.t.\ } \gamma(0)=x,\; \gamma(T)=y},
\end{equation}
i.e.\ $s(D)\dpe \mathbf D[g]$. See Figure~\ref{fig : App/CT/SymMetric}.
\begin{figure}
	\center
	\def\svgwidth{0.8\textwidth}
	\begingroup%
	  \makeatletter%
	  \providecommand\color[2][]{%
	    \errmessage{(Inkscape) Color is used for the text in Inkscape, but the package 'color.sty' is not loaded}%
	    \renewcommand\color[2][]{}%
	  }%
	  \providecommand\transparent[1]{%
	    \errmessage{(Inkscape) Transparency is used (non-zero) for the text in Inkscape, but the package 'transparent.sty' is not loaded}%
	    \renewcommand\transparent[1]{}%
	  }%
	  \providecommand\rotatebox[2]{#2}%
	  \newcommand*\fsize{\dimexpr\f@size pt\relax}%
	  \newcommand*\lineheight[1]{\fontsize{\fsize}{#1\fsize}\selectfont}%
	  \ifx\svgwidth\undefined%
	    \setlength{\unitlength}{458.40012918bp}%
	    \ifx\svgscale\undefined%
	      \relax%
	    \else%
	      \setlength{\unitlength}{\unitlength * \real{\svgscale}}%
	    \fi%
	  \else%
	    \setlength{\unitlength}{\svgwidth}%
	  \fi%
	  \global\let\svgwidth\undefined%
	  \global\let\svgscale\undefined%
	  \makeatother%
	  \begin{picture}(1,0.28324853)%
	    \lineheight{1}%
	    \setlength\tabcolsep{0pt}%
	    \put(0,0){\includegraphics[width=\unitlength,page=1]{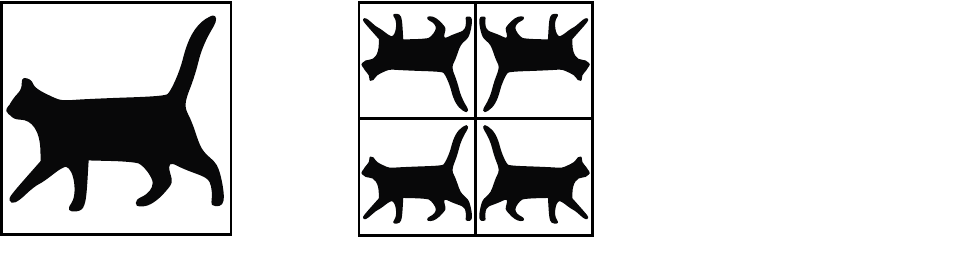}}%
	    \put(0.10215186,0.004016){\color[rgb]{0,0,0}\makebox(0,0)[lt]{\lineheight{1.25}\smash{\begin{tabular}[t]{l}$D$\end{tabular}}}}%
	    \put(0.47294798,0.00520475){\color[rgb]{0,0,0}\makebox(0,0)[lt]{\lineheight{1.25}\smash{\begin{tabular}[t]{l}$s(D)$\end{tabular}}}}%
	    \put(0,0){\includegraphics[width=\unitlength,page=2]{FIG_SymMetric.pdf}}%
	    \put(0.85124973,0.00401604){\color[rgb]{0,0,0}\makebox(0,0)[lt]{\lineheight{1.25}\smash{\begin{tabular}[t]{l}$s^2(D)$\end{tabular}}}}%
	  \end{picture}%
	\endgroup%
	\caption{Illustration of the first two iterations of $s$ in the case $d=2$.}
	\label{fig : App/CT/SymMetric}
\end{figure}
We postpone the proof of Claims~\ref{claim : App/CT/ptes_s_1} and~\ref{claim : App/CT/ptes_s_2} to the end of the subsection.
\begin{Claim}
	\label{claim : App/CT/ptes_s_1}
	Let $D\in \Metrics$. For all $1\le i\le d$,
	\begin{align}
	\label{eqn : App/CT/ptes_s/crossing}
		s(D)(H_i, H_i') &\ge D(H_i, H_i').
	\intertext{Moreover,}
	\label{eqn : App/CT/ptes_s/FdT}
		\MonRF\p{ s(D)} &\le \MonRF\p{ D}.
	\end{align}
\end{Claim}
\begin{Claim}
	\label{claim : App/CT/ptes_s_2}
	Let $D\in \Metrics$ and $n\ge 1$. For all $0\le k \le n-1$, $v\in\intint{0}{2^k-1}^d$ and $x,y\in \intervalleff{0}{2^{-k}}^d$,
    \begin{equation}\label{eqn : App/CT/ptes_s_2}
         s^n(D)\p{x + 2^{-k}v, y +  2^{-k}v} = s^n(D)(x,y).
	\end{equation}
\end{Claim}
Let $K$ denote the set of minimizers in~\eqref{eqn : App/CT/LDP}. Let $D\in K$. A straightforward induction argument using Claim~\ref{claim : App/CT/ptes_s_1} implies that for all $n\ge 1$,  $s^n(D)\in K$. Moreover $\MonRF$ is lower semi-continuous, therefore $K$ is compact, therefore $\p{ s^n(D) }_{n\ge 1}$ has an accumulation point $D'$ in $K$. Letting $n\to\infty$ in~\eqref{eqn : App/CT/ptes_s_2}, we obtain that for all $k \ge 0$, $v\in\intint{0}{2^k-1}^d$ and $x,y\in \intervalleff{0}{2^{-k}}^d$,
\begin{equation}\label{eqn : App/CT/ptes_s_2/Limit}
     D'\p{x + 2^{-k}v, y +  2^{-k}v} = D'(x,y).
\end{equation}
Proposition~\ref{prop : Gradient/derivee} further implies that for some seminorm $g\in \AdmNorms$, for almost every $z\in \intervalleff01^d$, $(\grad D')_z=g$. Consequently,
\begin{equation}
\label{eqn : App/CT/FdTCT_norme_presque}
	\RFCT(\zeta) = \MonRFelem(g).
\end{equation}
Moreover, $g$ is invariant under orthogonal reflections with respect to the hyperplanes $\acc{u\in \R^d \bigm| u_i=0}$, and for all $1\le i \le d$, $g(\base i)\ge \zeta_i$. 

Let $u=(u_1,\dots,u_d)\in \R^d$ and $1\le i \le d$. We have
\begin{align}
	u_i \base i &= 2^{-d+1}\sum_{\substack{\lambda \in \acc{-1,1}^d \\ \lambda_i =1} } \sum_{j=1}^d\lambda_j u_j \base j,\nonumber
	\intertext{thus}
	\module{u_i} g(\base i) &\le 2^{-d+1}\sum_{\substack{\lambda \in \acc{-1,1}^d :\\ \lambda_i =1} } g\p{ \sum_{j=1}^d\lambda_j u_j \base j }\nonumber,
     \intertext{thus}
	\zeta_i \module{u_i} &\le g(u),\nonumber
\end{align}
hence $g\ge g^{\zeta}$. Since $\MonRFelem$ is nondecreasing, plugging this into \eqref{eqn : App/CT/FdTCT_norme_presque} gives the lemma.
\end{proof}
\begin{proof}[Proof of Lemma~\ref{lem : App/CT/Ptys_RF}]We prove the different parts in order of difficulty.

Item~(\ref{item : App/CT/croissance}) is a consequence of the monotonicity of $\MonRFelem$.

Item~(\ref{item : App/CT/cas_infty}) is a consequence of Lemma~\ref{lem : ERF/LimitCases/infty}.

Let us prove Item~(\ref{item : App/CT/convexite}). Without loss of generality, it is sufficient to show that, given $a \le \zeta_2,\dots, \zeta_d \le b$,  $\RFCT(\cdot, \zeta_2, \dots , \zeta_d)$ is convex. Let $0\le \theta \le 1$ and $a \le \zeta_1, \zeta_1' \le b$. Consider the metric $D$ defined on $\intervalleff01^d$ by
\begin{equation}
    D(x,y) \dpe \inf \set{ \int_0^T g_{\gamma(t) }\p{ \gamma'(t)}\d t }{ \gamma : \intervalleff0T \rightarrow \intervalleff{0}{1}^d \text{ Lipschitz s.t.\ } \gamma(0)=x,\; \gamma(T)=y},
\end{equation}
with
\begin{equation*}
	g_z \dpe%
	\begin{cases}
		g^{(\zeta_1, \zeta_2,\dots, \zeta_d)} &\text{ if } 0\le z_1 \le \theta,\\
		g^{(\zeta_1', \zeta_2,\dots, \zeta_d)} &\text{ if } \theta < z_1 \le 1.
	\end{cases}
\end{equation*}
Then by Proposition \ref{prop : Gradient/courbe},
\begin{align}
    D(H_1, H_1') &\ge \theta \zeta_1 + (1-\theta) \zeta_1'\notag \\
\intertext{and for all $2\le i \le d$,}
    D(H_i, H_i') &\ge \zeta_i.\notag
\end{align}
Moreover, by~\eqref{eqn : MonLDP/Expression} and \eqref{eqn : Gradient/inegalite_gz_grad},
\begin{equation*}
    \MonRF(D) \le \theta\MonRFelem\p{g^{(\zeta_1, \zeta_2,\dots, \zeta_d)} } + (1-\theta)\MonRFelem\p{g^{(\zeta_1', \zeta_2,\dots, \zeta_d)} }.
\end{equation*}
Equations~\eqref{eqn : App/CT/LDP} and~\eqref{eqn : App/CT/MagicNorm} yield
\begin{equation}
\RFCT\p{\theta \zeta_1 + (1-\theta)\zeta_1', \zeta_2,\dots, \zeta_d}%
	\le \theta \RFCT\p{\zeta_1 , \zeta_2,\dots, \zeta_d} + (1-\theta)\RFCT\p{\zeta_1', \zeta_2,\dots, \zeta_d},
\end{equation}
thus Item (\ref{item : App/CT/convexite}).

To prove Item (\ref{item : App/CT/continuite}), we proceed differently depending on whether $\nu$ has an atom at $b$ or not.

\emph{Case 1: $\nu(\acc b)>0$.} Let $\zeta \in \intervalleff ab^d$ and $\eps>0$. Consider the metric $D$ on $\intervalleff01^d$ defined by \begin{equation}
    D(x,y) \dpe \inf \set{ \int_0^T g_{\gamma(t) }\p{ \gamma'(t)}\d t }{ \gamma : \intervalleff0T \rightarrow \intervalleff{0}{1}^d \text{ Lipschitz s.t.\ } \gamma(0)=x,\; \gamma(T)=y},
\end{equation}
with
\begin{equation*}
	g_z \dpe%
	\begin{cases}
		\mathfrak g^\zeta \qquad &\text{if } z\in \intervalleff0{1-\eps}^d,\\
		b\norme\cdot &\text{otherwise.}
	\end{cases}
\end{equation*}
Then by Proposition \ref{prop : Gradient/courbe}, for all $1\le i \le d$,
\begin{equation*}
	D(H_i, H_i') \ge \eps b + (1-\eps)\zeta_i \eqqcolon \zeta_i'.
\end{equation*}
Note that $U\dpe\acc{\hat\zeta \in \intervalleff ab^d \bigm| \forall 1\le i \le d, \hat\zeta_i \le \zeta_i'}$ is a neighbourhood of $\zeta$ in $\intervalleff{a}b^d$, regardless of whether $\zeta_i =b$ for some $i$ or not. Besides, by~\eqref{eqn : MonLDP/Expression} and \eqref{eqn : Gradient/inegalite_gz_grad},
\begin{equation*}
	\MonRF(D) \le (1-\eps)^d\MonRFelem\p{\mathfrak g^\zeta} + \cro{ 1- (1-\eps)^d}\MonRFelem\p{b\norme\cdot},
\end{equation*}
therefore 
\begin{equation*}
	\RFCT(\zeta') \le (1-\eps)^d\MonRFelem\p{\mathfrak g^\zeta} + \cro{ 1- (1-\eps)^d}\MonRFelem\p{b\norme\cdot}.
\end{equation*}
Applying Lemma \ref{lem : App/CT/MagicNorm} and Item~(\ref{item : App/CT/croissance}), we deduce that for all $\hat\zeta \in U$,
\begin{align*}
	\RFCT(\hat\zeta) &\le (1-\eps)^d\MonRFelem\p{\mathfrak g^\zeta} + \cro{ 1- (1-\eps)^d}\MonRFelem\p{b\norme\cdot}\\
		&\le \MonRFelem\p{\mathfrak g^\zeta} + \cro{ 1- (1-\eps)^d}\MonRFelem\p{b\norme\cdot}\\
		&= \RFCT(\zeta) + \cro{ 1- (1-\eps)^d}\MonRFelem\p{b\norme\cdot}
\end{align*}
therefore $\RFCT$ is upper semi-continuous on $\intervalleff ab^d$. Since $\RFCT$ is lower semi-continuous, it is continuous on $\intervalleff ab^d$. 

\emph{Case 2: $\nu(\acc b)=0$.} Proposition \ref{prop : ERF/Continuity} and Equation~\eqref{eqn : App/CT/MagicNorm} imply that $\RFCT$ is continuous on $\intervallefo ab^d$. By Proposition~\ref{prop : ERF/LimitCases}, $\RFCT(\zeta)=\infty$ if $\zeta_i=b$ for some $1\le i\le d$. Since $\RFCT$ is lower semi-continuous, it is continuous at such $\zeta$, thus~(\ref{item : App/CT/continuite}). This concludes the proof of Corollary~\ref{cor : Intro/applications/crossing}, up to the proof of Claims~\ref{claim : App/CT/ptes_s_1} and~\ref{claim : App/CT/ptes_s_2}.
\end{proof}

\begin{proof}[Proof of Claim~\ref{claim : App/CT/ptes_s_1}]
	Let $D$ be as in the claim. Without loss of generality, we only prove~\eqref{eqn : App/CT/ptes_s/crossing} for $i=1$. Let $\gamma : \intervalleff0T \rightarrow \intervalleff01^d$ be a Lipschitz path such that $\gamma(0)\in H_1$ and $\gamma(T) \in H_1'$. Note that for almost every $t\in \intervalleff0T$,
    \[ g_{\gamma(t)}\p{\gamma'(t)} \ge \frac12(\grad D)_{f\p{\gamma(t)} }\p{ (f\circ \gamma)'(t) }. \] 
    Hence by \eqref{eqn : Gradient/length_gamma}, 
    \[ \int_0^T g_{\gamma(t)}\p{\gamma'(t)}\d t \ge \frac12 D(f\circ \gamma). \]
    Since the path $f\circ \gamma$ contains two disjoint subpaths between $H_1$ and $H_1'$, this yields \eqref{eqn : App/CT/ptes_s/crossing}.
        
    By \eqref{eqn : Gradient/inegalite_gz_grad} and \eqref{eqn : MonLDP/Expression},
    \begin{align*}
        \MonRF\p{ s(D) } &\le \int_{\intervalleff01^d} \MonRFelem\p{g_z }\d z.
        \intertext{By invariance of the model with respect to symmetries of $\Z^d$, for almost every $z\in \intervalleff01^d$, $\MonRFelem(g_z) = \MonRFelem\p{ (\grad D)_{f(z)}}$, thus}
        \MonRF\p{ s(D) } &\le \int_{\intervalleff01^d} \MonRFelem\p{(\grad D)_{f(z)} }\d z\\
            &= \int_{\intervalleff01^d} \MonRFelem\p{(\grad D)_{z} }\d z = \MonRF(D).
    \end{align*}
\end{proof}
\begin{proof}[Proof of Claim~\ref{claim : App/CT/ptes_s_2}]
	Let $D \in \Metrics$. We proceed by induction on $n$. If $n = 1$, the claim is trivial. Fix $n \ge2$ and assume the claim holds for $n-1$. For all $z\in \intervalleff01^d$, we denote by $g_z$ the function that appears in the definition of $s^n(D)$ (see \eqref{eqn : App/CT/MagicMap}). Let $1\le k \le n-1$, $v\in \intint{0}{2^k-1}^d$ and $x,y \in \intervalleoo{0}{2^{-k}}^d$. Let $\gamma : \intervalleff0T \rightarrow \intervalleff01^d$ be a Lipschitz path such that $\gamma(0)=x + 2^{-k}v$ and $\gamma(T)=y + 2^{-k}v$. Denote $v' = \p{\ind{v_1 \ge 2^{k-1}}, \dots, \ind{v_d \ge 2^{k-1}} }$. Since $z \mapsto g_z$ is invariant under orthogonal symmetries with respect to the hyperplanes $\set{z\in \R^d}{z_i =1/2}$, excursions of $\gamma$ outside the box $\frac{v'}{2} + \intervalleff{0}{\frac12}^d$ may be replaced by paths inside the box without increasing the integral of $g_{\gamma(t)}\p{\gamma'(t)}\d t$. Consequently, we have 
    \begin{align*}
    s^{n}(D)(x  +2^{-k}v,y+2^{-k}v  ) &= \inf \set{ \int_0^T g_{\gamma(t) }\p{ \gamma'(t)}\d t }{ \begin{array}c \gamma : \intervalleff0T \rightarrow \frac{v'}{2} + \intervalleff{0}{\frac12}^d\text{ Lipschitz s.t.\ }\\  \gamma(0)=x  +2^{-k}v\; \gamma(T)=y+2^{-k}v  \end{array}}.
    \intertext{By the induction hypothesis, applying the same argument at every successive dyadic scale yields}
    s^{n}(D)(x  +2^{-k}v,y+2^{-k}v  ) &= \inf \set{ \int_0^T g_{\gamma(t) }\p{ \gamma'(t)}\d t }{ \begin{array}c \gamma : \intervalleff0T \rightarrow 2^{-k}v + \intervalleff{0}{2^{-k}}^d\text{ Lipschitz s.t.\ }\\ \gamma(0)=x  +2^{-k}v\; \gamma(T)=,y+2^{-k}v \end{array}},
    \end{align*}
    thus
    \[
        s^n(D)(x  +2^{-k}v,y+2^{-k}v  ) = s^n(D)(x,y).
    \]
\end{proof}

\subsection{Rescaled random ball: proof of Corollary~\ref{cor : Intro/applications/ball}}
\label{subsec : App/ball}
In this section we assume that $0<a<b<\infty$. Fix an integer $C\ge \frac1a$. Let $\Compacts$ denote the set of compact subsets of $\R^d$. Consider the map 
\begin{align}
	\hat \Phi : \GeoMetrics\p{\intervalleff{-C }{C }^d } &\longrightarrow \Compacts \eol
			D &\longmapsto \acc{x\in \intervalleff{-C}{C}^d \bigm| D(0,x)\le 1 }.
\end{align}
Let $D_1, D_2 \in \GeoMetrics\p{\intervalleff{-C }{C }^d } $ and $x\in \hat \Phi(D_1)$. By the definition of $\UnifDistance$,
\begin{align}
	D_2(0,x) &\le D_1(0,x) + \UnifDistance(D_1, D_2) \le 1 + \UnifDistance(D_1, D_2).\notag
	\intertext{Let $0\Path{\sigma} x$ be a $D_2$-geodesic and $y\dpe \sigma\p{ \cro{D_2(0,x) - \UnifDistance(D_1, D_2)}\vee 0 }$. Then $y\in \hat \Phi(D_2)$ and }
	\norme{x-y} &\le \frac 1a D_2(x,y) \le  \frac1a \UnifDistance(D_1, D_2).\notag
	\intertext{Consequently,}
	\hat \Phi(D_1) &\subseteq \hat \Phi(D_2) + \clball{0, \frac1a\UnifDistance(D_1, D_2)}.\notag
	\intertext{Swapping $D_1$ and $D_2$ leads to}
	\hat \Phi(D_2) &\subseteq \hat \Phi(D_1) + \clball{0, \frac1a\UnifDistance(D_1, D_2)}.\notag
\end{align}
Hence $\hat \Phi$ is $\frac1a$-Lipschitz for the Hausdorff distance.

Let $\Phi$ be any continuous extension on $\Metrics\p{\intervalleff{-C}{C}^d}$ of $\hat \Phi$. The lower bound in~\eqref{eqn : Intro/Main/LDP/GeoMetrics/Bounds} implies that for all $n\ge 1$, \[ \HD\p{ \RUB(n), \Phi(\BoxSPT[C,n]) } \le \frac 1{an}. \]Applying Lemma~\ref{lem : App/Contraction} completes the proof. \qed

\appendix
\section{Continuous extension of the random pseudometric}
 We check in this section that $\CPT$ is compatible with the usual definition of the passage time.
  \begin{Proposition}
  \label{prop : GeoMetrics/Geodesics/extension_tranquille}
    The maps $\PT_n$ and $\CPT$ defined by \eqref{eqn : Intro/Background/def_PT} and \eqref{eqn : Intro/Framework/BoxPT} coincide on $\intint0n^d\times \intint0n^d$.
  \end{Proposition}
  \begin{proof}
  	The key property of $\intervalleff0n^d$ used here is that for all $z_1,z_2\in \intervalleff{0}{n}^d$,
  	\begin{equation}
  	\label{eqn : GeoMetrics/Geodesics/monotone_paths_in_X}
  		\begin{split} \text{any finite sequence of points with monotone coordinates}\\ \text{and endpoints $z_1$ and $z_2$ is included in }\intervalleff 0n^d. \end{split}
  	\end{equation}
  	Fix $x,y \in \intint0n^d$. Let $\gamma = \p{x=x_0,x_1,\dots,x_r=y}$ be a finite sequence in $\intervalleff0n^d$. It is sufficient to show that there exists a discrete path from $x$ to $y$ with passage time at most $\PathPT{\gamma}$.

  	\paragraph{Step 1: Removing points outside $\bbE^d$.} For all $1\le j \le r-1$ such that $x_j\notin \bigcup_{e\in\bbE^d} e$,
  	\begin{align}
     		\tau(x_{j-1}, x_{j+1})%
     			&\le b\norme{x_{j-1} - x_{j+1}} \eol
     			&\le b\norme{x_{j-1} - x_{j}} + b\norme{x_{j} - x_{j+1}}\eol
     			&= \PathPT{x_{j-1} , x_{j}} + \PathPT{x_{j} , x_{j+1}},\notag
   	\end{align}
   	thus removing such points from $\gamma$ can only decrease the passage time. From now on we may therefore assume that for all $0\le j \le r$, $x_j \in \bigcup_{e\in \edges{\intint0n^d} }e $.

   	\paragraph{Step 2: Inserting integer points.} Let $0\le j \le r-1$ such that $x_j$ and $x_{j+1}$ are not both integer points. Without loss of generality we may assume that $x_j \le x_{j+1}$ for coordinate-wise order. All coordinates of $x_j$ except possibly one are integer, so $x_j$ and $\floor{x_j}$ belong to a common edge. The same goes for $x_{j+1}$ and $\ceil{x_{j+1}}$.

    \emph{Case 2.1: Assume that $\ceil{x_j}\le \floor{x_{j+1} }$.} Property \eqref{eqn : GeoMetrics/Geodesics/monotone_paths_in_X} implies that $\ceil{x_j}, \floor{x_{j+1} }\in \intint0n^d$. Moreover the sequence $\gamma_j\dpe\p{x_j, \ceil{x_j}, \floor{x_{j+1}} , x_{j+1} }$ satisfies
    	\begin{equation*}
    		\PathPT{\gamma_j} \le  \PathPT{x_j, x_{j+1}}.
    	\end{equation*}

    \emph{Case 2.2: Assume that $\ceil{x_j}\nleq \floor{x_{j+1}}$.} Then there exist $1\le i \le d$ and $k\in \Z$ such that 
    \begin{equation*}
      k < \ps{x_j}{\mathrm e_i}, \ps{x_{j+1}}{\mathrm e_i} < k+1,
    \end{equation*}
    and all other coordinates of $x_j$ and $x_{j+1}$ are integers. Thus $x_j$ and $x_{j+1}$ belong to the strip
    \begin{equation*}
    \edges[i,k]{\intint0n^d}\dpe \bigcup\limits_{\substack{v=(v_1,\dots, v_d)\in \intint0n^d\\ v_i=k}} \intervalleff{v}{v+\mathrm e_i }.
  \end{equation*}
  Consequently, by inserting the sequence $\gamma_j$ between $x_j$ and $x_{j+1}$ for all $j$ satisfying Case~2.1, one constructs a sequence $\gamma'\dpe \p{x=y_0, \dots, y_s=y}$ such that for all $0\le j \le s-1$, at least one of the following assertions is true:
  \begin{align}
  	&\text{Both $y_j$ and $y_{j+1}$ are integer points}.\\
  	\label{eqn : GeoMetrics/Geodesics/same_strip}
  	&\text{Both $y_j$ and $y_{j+1}$ belong to a common strip $\edges[i,k]{\intint0n^d}$}.
  \end{align}

  \paragraph{Step 3: Deleting non integer points.} Let $0\le j_1 \le j_1+2 \le j_2 \le s$ be indices such that $y_{j_1}$ and $y_{j_2}$ are integer points but for all $j_1<j<j_2$, $y_j$ is not. We claim that the subsequence of $\gamma'$ between $j_1$ and $j_2$ may be replaced by a sequence of points in $\intint0n^d$ with a lesser or equal passage time. For all $j_1<j<j_2$, there exists a unique pair $\p{i(j), k(j)}$ such that $y_j \in \edges[i(j),k(j)]{\intint0n^d}$. Hence,~\eqref{eqn : GeoMetrics/Geodesics/same_strip} yields the existence of a pair $(i,k)$ such that for all $j_1\le j \le j_2$, $y_j\in\edges[i,k]{\intint0n^d}$.

  \emph{Case 3.1: Assume that $\ps{y_{j_1}- y_{j_2}}{\mathrm{e}_i}=0$.} Then 
  	\begin{align*}
   		 b\norme{y_{j_1}- y_{j_2}} %
   		 	&\le \sum_{i'\neq i } \sum_{j=j_1}^{j_2-1} b\module{ \ps{y_j - y_{j+1}}{\mathrm e_{i'}} } \\
      		&\le \sum_{j=j_1}^{j_2-1} \PathPT{y_j , y_{j+1}},
  	\end{align*}
  therefore $\p{y_{j_1}, y_{j_2}}$ is a suitable sequence (see Figure~\ref{fig : GeoMetrics/Geodesics/strip}, top).

  \emph{Case 3.2: Assume that $\ps{y_{j_1}- y_{j_2}}{\mathrm{e_i}}=\pm 1$.} We only treat the subcase ${\ps{y_{j_1}- y_{j_2}}{\mathrm{e_i}}= 1}$. The other one is similar. We have
  \begin{align*}
    \sum_{j=j_1}^{j_2-1} \PathPT{y_j, y_{j+1}} %
    	&\ge \sum_{j=j_1}^{j_2-1} \p{\module{\ps{y_j - y_{j+1}}{\mathrm e_{i}}} \min_{v\in V(j_1, j_2)} \EPT{\p{v, v+\mathrm e_i}} + \sum_{i'\neq i}b\module{\ps{y_j - y_{j_2}}{\mathrm e_{i'}}}  },
    \intertext{where $V(j_1, j_2)\dpe \acc{\floor{y_j},\quad j_1 < j  < j_2  }$. Thus, by fixing $v_0 \in \argmin\limits_{v\in V(j_1, j_2)}\EPT{\p{v, v+\mathrm e_i}}, $}
    \sum_{j=j_1}^{j_2-1} \tau\p{y_j, y_{j+1}} %
    	&\ge b\norme{y_{j_1} - v_0} + \p{\sum_{j=j_1}^{j_2-1} \module{\ps{y_j - y_{j+1}}{\mathrm e_{i}} } }\EPT{\p{v_0, v_0 + \mathrm e_i}}+ b\norme{v_0 + \mathrm e_i - y_{j+1}}  \\
      &\ge b\norme{y_{j_1} - v_0} +  \EPT{\p{v_0, v_0 + \mathrm e_i}} + b\norme{v_0 + \mathrm e_i - y_{j_2}}.
  \end{align*}
  Moreover $v_0, v_0 + \mathrm e_i \in \intint0n^d$, thus $\p{y_{j_1}, v_0 + \mathrm e_i, v_0,  y_{j_2}}$ is a suitable sequence (see Figure~\ref{fig : GeoMetrics/Geodesics/strip}, bottom).
  \begin{figure}
  \begin{tikzpicture}[scale =1.95]
    \draw[step= 1cm,gray,thin] (-0.2,-0.2) grid (7.2,1.2);
    \draw[semithick] (-0.5, 0) -- (7.5,0)
    (-0.5, 1) -- (7.5,1);
    \draw[anchor=south]%
  (1,1) node[fill=white] {$y_{j_1}$}
  (5,1) node[fill=white] {$y_{j_2}$};
  \filldraw (1,1) circle (1pt)
            (5,1) circle (1pt)
            (1,0.5) circle (1pt)
            (2,0.5) circle (1pt)
            (4,0.25) circle (1pt)
            (4,0.75) circle (1pt)
            (6,0.33) circle (1pt);
    \draw[-stealth, very thick] (1,1)--(1,0.5)--(2,0.5)--(4,0.25)--(4,0.75)--(6,0.33)-- (5,1);
    \draw[-stealth, dotted, very thick] (1,1)--(5,1);
    \draw[decorate, decoration={brace, amplitude=5pt},thick ] (7.6,1)--(7.6,0) node[pos=0.5, right=4pt] {$\edges[i,k]{X}$};
    \begin{scope}[yshift = -1.6cm]
      \draw[step= 1cm,gray,thin] (-0.2,-0.2) grid (7.2,1.2);
    \draw[semithick] (-0.5, 0) -- (7.5,0)
    (-0.5, 1) -- (7.5,1);
    \draw[anchor=south]%
  (1,0) node[fill=white, anchor = north] {$y_{j_1}$}
  (5,1) node[fill=white] {$y_{j_2}$};
  \filldraw (1,0) circle (1pt)
            (5,1) circle (1pt)
            (3,0.5) circle (1pt)
            (6,0.25) circle (1pt)
            (7,0.5) circle (1pt)
            (7,0.75) circle (1pt)
            (6,0) circle (1pt)
            (6,1) circle (1pt);%
    \draw[-stealth, very thick] (1,0)--(3,0.5)--(6,0.25)--(7,0.5)--(7,0.75)-- (5,1);
    \draw[-stealth, dotted, very thick] (1,0)--(6,0)--(6,1)--(5,1);    
    \draw (6,0) node[anchor= north west] {$v_0$};
    \draw[decorate, decoration={brace, amplitude=5pt},thick ] (7.6,1)--(7.6,0) node[pos=0.5, right=4pt] {$\edges[i,k]{X}$};
    \end{scope}
  \end{tikzpicture}
  \caption{The subsequence of $\gamma'$ between $y_{j_1}$ and $y_{j_2}$ (solid line) can be replaced by a faster sequence (dotted lines) using only integer points. }
  \label{fig : GeoMetrics/Geodesics/strip}
  \end{figure}
  Consequently, by considering every pair of indices $(j_1,j_2)$ corresponding to consecutive integer points in $\gamma'$, one constructs a sequence $\gamma''$ of points in $\intint0n^d$, with passage time at most $\PathPT{\gamma'}$, whose endpoints are $x$ and $y$. Finally, due to property~\eqref{eqn : GeoMetrics/Geodesics/monotone_paths_in_X}, $\gamma''$ can be assumed to be a discrete path. 
  \end{proof}


\section{Large deviation tools}
\label{appsec : LD}
In this section we prove the properties of the rate functions introduced throughout the article.
\begin{Lemma}
    \label{lem : LD/RateFunctions}
    Let $(\cX, \d_\cX)$ be a metric space. Consider for all $n\ge 1$, $x\in \cX$ and $\eps>0$, an event $E_n(x,\eps)$. Assume that:
    \begin{enumerate}[(i)]
        \item for all $n\ge 1$, $x\in \cX$, $E_n(x,\cdot)$ is nondecreasing for the inclusion,
        \item there exists a constant $C<\infty$ such that for all $n\ge 1$, $x,y\in \cX$ and $\eps>0$,
        \[ E_n(x,\eps) \subseteq E_n\p{y, \eps + C\d_\cX(x,y)}.  \]
    \end{enumerate}
    For all $x\in \cX$, define
    \begin{align*}
        \overline J(x) &\dpe \lim_{\eps\to 0} \limsup_{n\to\infty} -\frac1{n^d}\log \Pb{ E_n(x,\eps)} \\
        \text{and } \underline J(x) &\dpe \lim_{\eps\to 0} \liminf_{n\to\infty} -\frac1{n^d}\log \Pb{ E_n(x,\eps)}.
    \end{align*}
    Then $\overline J$ and $\underline J$ are rate functions on $\cX$. In particular, the functions $\RF$, $\MonRFelem$, $I_{C, \mathrm{pp}(x)}$, $\RFppx$, $\RFPtS$,  and $\RFCT$, as well as their variants with the superscript $+$, an upper or a lower bar are all rate functions.
    \end{Lemma}
\begin{proof}
    Let $(x_k)_{k\ge 1}$ be a sequence converging to $x$ in $\cX$ and $\eps >0$. Then for all $n\ge 1$,
	\begin{align*}
		E_n(x_k, \eps) &\subseteq E_n\p{ x, \eps + C\d_\cX(x,x_k)},
		\intertext{thus}
		-\frac1{n^d} \log \Pb{ E_n(x_k, \eps)} &\ge -\frac1{n^d}\log \Pb{ E_n\p{ x, \eps + C\d_\cX(x,x_k)} }.
		\intertext{Taking the upper limit as $n\to\infty$, we get}
		\overline{J}(x_k)&\ge \limsup_{n\to\infty} -\frac1{n^d}\log \Pb{ E_n\p{ x, \eps + C\d_\cX(x,x_k)} }.
		\intertext{In particular, for large enough $k$,}
		\overline{J}(x_k)&\ge \limsup_{n\to\infty} -\frac1{n^d}\log \Pb{ E_n\p{ x, 2\eps} } ,
		\intertext{therefore}
		\liminf_{k\to\infty} \overline{J}(x_k)&\ge \limsup_{n\to\infty} -\frac1{n^d}\log \Pb{ E_n\p{ x, 2\eps} }.
		\intertext{Letting $\eps\to0$ yields}
		\liminf_{k\to\infty} \overline{J}(x_k)&\ge \overline{J}(x),
	\end{align*}
	i.e. $\overline J$ is a rate function. Likewise $\underline J$ is a rate function.
\end{proof}

\begin{proof}[Proof of Lemma~\ref{lem : Intro/LD/UpperLowerRF}]Let $\cX$, $(X_n)$, $\overline J$ and $\underline J$ as in the lemma. We denote by $\clball[\cX]{x,r}$ the closed ball of center $x$ and radius $r$ in $(\cX, \d_\cX)$.

\emph{Proof of (i).} Consequence of Lemma \ref{lem : LD/RateFunctions}.

\emph{Proof of (ii).} Let $U\subseteq \cX$ be an open set and $x\in U$. There exists $\eps>0$ such that $\clball[\cX]{x,\eps}\subseteq U$. In particular, for all $n\ge 1$,
	\begin{equation*}
		 \acc{\d_\cX(x,X_n)\le \eps } \subseteq \acc{X_n \in U}, 
	\end{equation*}
	thus
	\begin{equation*}
		\limsup_{n\to\infty}-\frac{1}{n^d}\log \Pb{X_n \in U} \le \limsup_{n\to\infty}-\frac{1}{n^d}\log \Pb{\d_\cX(x,X_n)\le \eps } \le \overline J (x).
	\end{equation*}
	Taking the infimum over $x\in U$, we get~\eqref{eqn : Intro/LD/UpperLowerRF/UB}.

\emph{Proof of (iii).} Let $K\subseteq \cX$ be a compact set and $\eps>0$. There exists a finite family $(x_p)_{p=1}^P$ of elements of $K$ such that
	\begin{equation*}
		K \subseteq \bigcup_{p=1}^P \clball[\cX]{x_p, \eps}.
	\end{equation*}
	Consequently by union bound, for all $n\ge1$,
	\begin{align*}
		\Pb{X_n \in K} &\le \sum_{p=1}^P \Pb{\d_\cX(x_p, X_n) \le \eps} \\
			&\le P \max_{1\le p \le P} \Pb{\d_\cX(x_p, X_n) \le \eps}.
	\end{align*}
	Taking the $\log$, dividing by $-n^d$ and letting $n\to \infty$, we get
	\begin{align}
		\liminf_{n\to\infty} -\frac{1}{n^d}\log \Pb{X_n \in K}%
			&\ge \liminf_{n\to\infty} -\frac{1}{n^d} \log \p{ \max_{1\le p \le P} \Pb{\d_\cX(x_p, X_n) \le \eps} }\nonumber \\
			&= \liminf_{n\to\infty} \min_{1\le p \le P} \p{ -\frac{1}{n^d} \log \Pb{\d_\cX(x_p, X_n) \le \eps} }\nonumber \\
			&= \min_{1\le p \le P} \liminf_{n\to\infty}  -\frac{1}{n^d} \log \Pb{\d_\cX(x_p, X_n) \le \eps}.\label{eqn : LD/inv_liminf_min}
	\end{align}
	Let us denote by $\hat x(\eps)$ an element $x_p$ on which the minimum is attained. By compactness there exists a sequence $(\eps_k)_{k\ge 1}$ converging to $0$, such that $\p{\hat x(\eps_k)}_{k\ge 1}$ converges in $K$. We write $\hat x$ its limit. Then for all $k\ge 1$,
	\begin{align*}
		\liminf_{n\to\infty} -\frac{1}{n^d}\log \Pb{X_n \in K}%
			&\ge \liminf_{n\to\infty}  -\frac{1}{n^d} \log \Pb{\d_\cX\p{ \hat x(\eps_k), X_n } \le \eps_k}\\
		 	&\ge  \liminf_{n\to\infty}  -\frac{1}{n^d} \log \Pb{\d_\cX\p{\hat x, X_n } \le \eps_k + \d_\cX\p{\hat x , \hat x(\eps_k)} }.
		 \intertext{Letting $k\to\infty$ gives}
		 \liminf_{n\to\infty} -\frac{1}{n^d}\log \Pb{X_n \in K}%
		 	&\ge \underline{J}(\hat x),
	\end{align*}
	thus~\eqref{eqn : Intro/LD/UpperLowerRF/LB}.
\end{proof}
\begin{Remark}
	\label{rk : LD/inv_limsup_min_non}
	The permutation of $\liminf$ and $\min$ in Equation~\eqref{eqn : LD/inv_liminf_min} is false in general with $\limsup$ instead of $\liminf$, thus this proof may not be adapted to show
	\begin{equation*}
		\limsup_{n\to\infty}-\frac{1}{n^d} \log \Pb{X_n \in K} \ge \min_{x\in K} \overline J(x).
	\end{equation*}
\end{Remark}

\begin{proof}[Proof of Lemma \ref{lem : App/MonContraction}]
    Let $\Phi : \Metrics \rightarrow \R^k$ be a continuous, nondecreasing function and $\zeta \in \R^k$. Let $\omega$ denote a modulus of continuity for $\Phi$ on the compact set $\Metrics$. If for all $D\in \Metrics$, $\Phi(D) \ngeq \zeta$, then $\underline I_\Phi^+(\zeta) = \infty$, and the result is clear. Assume that for some $D\in\Metrics$, $\Phi(D) \ge \zeta$. Let us prove the first two equalities in \eqref{eqn : App/MonContraction}. It is sufficient to prove that
    \begin{align}
        \label{eqn : App/MonContraction/UB}
        \overline I_\Phi^+( \zeta) &\le \inf_{\substack{D \in \Metrics \\ \Phi(D) \ge \zeta}  } \MonRF(D)
        \intertext{and}
        \label{eqn : App/MonContraction/LB}
        \underline I_\Phi^+(\zeta) &\ge \inf_{\substack{D \in \Metrics \\ \Phi(D) \ge \zeta}  } \MonRF(D).
    \end{align}

    Let $D\in \Metrics$ such that $\Phi(D) \ge \zeta$, and $\eps>0$. For all $n\ge 1$, 
    \begin{align*}
        \acc{ \forall x,y \in \intervalleff 01^d, \; \SPT(x,y) \ge D(x,y) - \eps } 
            &\subseteq \acc{ \UnifDistance\p{\SPT \vee D, \SPT} \le \eps } \\
            &\subseteq \acc{ \norme[\infty]{ \Phi(\SPT \vee D) - \Phi(\SPT)} \le \omega(\eps) }.
        \intertext{Since $\Phi$ is nondecreasing,}
        \acc{ \forall x,y \in \intervalleff 01^d, \; \SPT(x,y) \ge D(x,y) - \eps }%
            &\subseteq \acc{\forall i\in \intint1k,\; \Phi_i(\SPT) \ge \zeta_i - \omega(\eps) }
    \end{align*}
    Consequently,
    \[
        \limsup_{n\to \infty}-\frac1{n^d}\log \Pb{ \forall x,y \in \intervalleff 01^d, \; \SPT(x,y) \ge D(x,y) - \eps }%
            \ge \limsup_{n\to \infty}-\frac1{n^d}\log \Pb{\forall i\in \intint1k,\; \Phi_i(\SPT) \ge \zeta_i - \omega(\eps) }.
    \]    
    Letting $\eps\to 0$ yields \eqref{eqn : App/MonContraction/UB}.
    
    For all $\eps\ge 0$, define the compact set
    \[ K_\eps \dpe \set{D \in \Metrics}{\forall i\in \intint1k, \quad \Phi_i(D) \ge \zeta_i - \eps}. \]
    We follow the proof of the third part of Lemma \ref{lem : Intro/LD/UpperLowerRF}. There exists a finite family $(D_p)_{p=1}^P$ of elements of $K_\eps$ such that
	\begin{equation*}
		K_\eps \subseteq \bigcup_{p=1}^P \clball[\UnifDistance]{D_p, \eps}.
	\end{equation*}
	Consequently by union bound, for all $\eps>0$ and $n\ge1$,
	\begin{align*}
		\Pb{\SPT \in K_\eps} &\le \sum_{p=1}^P \Pb{\UnifDistance(D_p, \SPT) \le \eps} \\
			&\le P \max_{1\le p \le P} \Pb{\UnifDistance(D_p, \SPT) \le \eps}.
	\end{align*}
	Taking the $\log$, dividing by $-n^d$ and letting $n\to \infty$, we get
	\begin{align}
		\liminf_{n\to\infty} -\frac{1}{n^d}\log \Pb{\SPT \in K_\eps}%
			&\ge \liminf_{n\to\infty} -\frac{1}{n^d} \log \p{ \max_{1\le p \le P} \Pb{\UnifDistance(D_p, \SPT) \le \eps} }\nonumber \\
			&= \liminf_{n\to\infty} \min_{1\le p \le P} \p{ -\frac{1}{n^d} \log \Pb{\UnifDistance(D_p, \SPT) \le \eps} }\nonumber \\
			&= \min_{1\le p \le P} \liminf_{n\to\infty}  -\frac{1}{n^d} \log \Pb{\UnifDistance(D_p, \SPT) \le \eps}, \nonumber
        \intertext{thus}
        \underline I_\Phi^+(\zeta)%
            &\ge \min_{1\le p \le P} \liminf_{n\to\infty}  -\frac{1}{n^d} \log \Pb{\UnifDistance(D_p, \SPT) \le \eps}
            \label{eqn : LD/Mon/inv_liminf_min}
	\end{align}
	Let us denote by $\hat D(\eps)$ an element $D_p$ on which the minimum is attained. By compactness there exist a sequence $(\eps_k)_{k\ge 1}$ converging to $0$, such that $\p{\hat D(\eps_k)}_{k\ge 1}$ converges in $\Metrics$. We write $\hat D$ its limit. Then for all $k\ge 1$,
	\begin{align*}
		\underline I_\Phi^+(\zeta)%
			&\ge \liminf_{n\to\infty}  -\frac{1}{n^d} \log \Pb{\UnifDistance\p{ \hat D(\eps_k), \SPT } \le \eps_k}\\
		 	&\ge  \liminf_{n\to\infty}  -\frac{1}{n^d} \log \Pb{\UnifDistance\p{\hat D, \SPT } \le \eps_k + \UnifDistance\p{\hat D , \hat D(\eps_k)} }.
		 \intertext{Letting $k\to\infty$ gives}
		 \underline I_\Phi^+(\zeta)%
		 	&\ge \underline{I}^+(\hat D).
	\end{align*}
    Besides, $\hat D \in K_\eps$, for all $\eps>0$, therefore $\hat D \in K_0$, thus~\eqref{eqn : App/MonContraction/LB}.

    To prove the last equality in \eqref{eqn : App/MonContraction}, note that for all $g\in \cG\p{\intervalleff01^d} $, $\mathbf D[g]\in \Metrics$, thus 
    \[
        \inf \set{\MonRF(D)}{D\in \Metrics \text{ s.t.\ } \Phi(D)\ge \zeta }%
        \le \inf \set{ \int_{\intervalleff{0}{1}^d }\MonRFelem(g_z)\d z }{g\in \cG\p{\intervalleff{0}{1}^d} \text{ s.t.\ }\Phi(\mathbf D[g]) \ge \zeta }.
    \]
    Besides, by Propositions \ref{prop : Gradient/courbe} and \ref{prop : Gradient/inverse}, for all $D\in \Metrics$, 
    \[
        \mathbf D\cro{\grad D} \ge D \text{  and  } \MonRF\p{ \mathbf D\cro{\grad D} } \le \MonRF(D),
    \]
    thus giving the converse inequality and completing the proof of \eqref{eqn : App/MonContraction}.

    Let us finally show \eqref{eqn : App/MonContraction/ContinuousCase}. Fix a point of right-continuity $\zeta < \max \Phi$ for $\MonRFPhi$. By definition of $\MonRFPhi$, for all $\eps>0$,
    \[
        \MonRFPhi(\zeta+\eps) %
            \ge \limsup_{n\to\infty}-\frac{1}{n^d}\log \Pb{\Phi(\SPT) \ge \zeta}%
            \ge \liminf_{n\to\infty}-\frac{1}{n^d}\log \Pb{\Phi(\SPT) \ge \zeta}%
            \ge \MonRFPhi(\zeta),
    \]
    so letting $\eps\to0$ yields \eqref{eqn : App/MonContraction/ContinuousCase}.
\end{proof}

\end{document}